\documentclass[12pt,a4, reqno]{amsart}
\usepackage[a4paper, left=26mm, right=26mm, top=28mm, bottom=34mm]{geometry}
\usepackage[all]{xy}
\usepackage{amsmath}
\usepackage{amscd}
\usepackage{amssymb}
\usepackage{amsthm}
\usepackage{url}
\usepackage{mathrsfs}
\usepackage{comment}
\usepackage{bbm}

\newtheorem{thm}{Theorem}[subsection]
\newtheorem{lem}[thm]{Lemma}
\newtheorem{cor}[thm]{Corollary}
\newtheorem{prop}[thm]{Proposition}
\newtheorem{conj}[thm]{Conjecture}

\theoremstyle{definition}
\newtheorem{ass}[thm]{Assumption}
\newtheorem{rem}[thm]{Remark}
\newtheorem{defn}[thm]{Definition}

\newtheorem{ex}[thm]{Example}

\def\F{{\mathbb F}}
\def\G{{\mathbb G}}
\def\O{{\mathcal O}}
\def\Q{{\mathbb Q}}

\def\Z{{\mathbb Z}}

\def\L{{\mathbb L}}

\def\Art{\mathop{\mathrm{Art}}\nolimits}

\def\Frob{\mathop{\mathrm{Frob}}\nolimits}

\def\Fil{\mathop{\mathrm{Fil}}\nolimits}

\def\Gal{\mathop{\mathrm{Gal}}\nolimits}
\def\Lie{\mathop{\mathrm{Lie}}\nolimits}
\def\Hom{\mathop{\mathrm{Hom}}\nolimits}

\def\Ker{\mathop{\mathrm{Ker}}\nolimits}
\def\id{\mathop{\mathrm{id}}\nolimits}

\def\GL{\mathop{\mathrm{GL}}\nolimits}

\def\Rep{\mathrm{Rep}}
\def\Spa{\mathop{\rm Spa}}
\def\Spec{\mathop{\rm Spec}}
\def\Spf{\mathop{\rm Spf}}

\def\D{\mathbb{D}}

\def\diag{\mathop{\mathrm{diag}}}

\def\KS{\mathop{\mathrm{KS}}\nolimits}

\def\Fil{\mathop{\mathrm{Fil}}\nolimits}
\def\dR{\mathop{\mathrm{dR}}}

\newcommand{\et}{\mathrm{\acute{e}t}}

\def\Perf{\mathop{\mathrm{Perf}}\nolimits}

\def\arc{\mathop{\mathrm{arc}}\nolimits}

\def\Vect{\mathop{\mathrm{Vect}}\nolimits}
\def\alg{\mathop{\mathrm{alg}}\nolimits}

\def\op{\mathop{\mathrm{op}}\nolimits}
\def\Def{\mathop{\mathrm{Def}}\nolimits}
\def\Per{\mathop{\mathrm{Per}}\nolimits}
\def\Lift{\mathop{\mathrm{Lift}}\nolimits}

\def\univ{\mathop{\mathrm{univ}}\nolimits}
\def\ev{\mathop{\mathrm{ev}}\nolimits}

\def\alg{\mathop{\mathrm{alg}}\nolimits}
\def\Spd{\mathop{\mathrm{Spd}}\nolimits}
\def\sht{\mathop{\mathrm{sht}}\nolimits}
\def\red{\mathop{\mathrm{red}}\nolimits}
\def\int{\mathop{\mathrm{int}}\nolimits}

\usepackage{relsize} 
\usepackage[bbgreekl]{mathbbol} 
\usepackage{amsfonts} 
\usepackage{color}
\DeclareSymbolFontAlphabet{\mathbb}{AMSb} 
\DeclareSymbolFontAlphabet{\mathbbl}{bbold} 
\newcommand{\Prism}{{\mathlarger{\mathbbl{\Delta}}}}

\numberwithin{equation}{section}

\usepackage[T1]{fontenc}
\makeatletter
\@namedef{subjclassname@2020}{\textup{2020} Mathematics Subject Classification}
\makeatother

\begin{document}

\title[Deformation theory for prismatic $G$-displays]{Deformation theory for prismatic $G$-displays}

\author{Kazuhiro Ito}
\address{Mathematical Institute, Tohoku University, 6-3, Aoba, Aramaki, Aoba-Ku, Sendai 980-8578, Japan}
\email{kazuhiro.ito.c3@tohoku.ac.jp}




\begin{abstract}
For a smooth affine group scheme $G$ over the ring of $p$-adic integers and a cocharacter $\mu$ of $G$, we develop the deformation theory for $G$-$\mu$-displays over the prismatic site of Bhatt--Scholze, and discuss how our deformation theory can be interpreted in terms of prismatic $F$-gauges introduced by Drinfeld and Bhatt--Lurie.
As an application,
we prove the local representability and the formal smoothness of integral local Shimura varieties with hyperspecial level structure.
We also revisit and extend some classification results of $p$-divisible groups.
\end{abstract}

\maketitle

\setcounter{tocdepth}{1}
\tableofcontents

\section{Introduction} \label{Section:Introduction}

Let $p$ be a prime number.
Let
$G$
be
a smooth affine group scheme over the ring of $p$-adic integers $\Z_p$.
Let
$
\mu \colon \G_m \to G_{W(k)}:=G \times_{\Spec \Z_p} \Spec W(k)
$
be a cocharacter defined over $W(k)$,
where $k$ is a perfect field of characteristic $p$ and $W(k)$ is the ring of $p$-typical Witt vectors of $k$.
The theory of \textit{$G$-$\mu$-displays} was originally introduced by B\"ultel \cite{Bultel} and B\"ultel--Pappas \cite{Bultel-Pappas} to study integral models of Shimura varieties and Rapoport--Zink spaces.
In \cite{Ito-K23},
we studied
the notion of $G$-$\mu$-displays over the prismatic site of Bhatt--Scholze \cite{BS}.
The purpose of this paper is to develop the deformation theory for $G$-$\mu$-displays over the prismatic site, which has applications in the theory of integral models of Shimura varieties and local Shimura varieties.
Our deformation theory for $G=\GL_N$ provides new perspectives on some classification results of $p$-divisible groups.
We also discuss how our results can be interpreted in terms of \textit{prismatic $F$-gauges} introduced by Drinfeld and Bhatt--Lurie.

\subsection{Main results}\label{Main results introduction}

Let $(A, I)$ be a bounded prism in the sense of \cite{BS}, such that $A$ is a $W(k)$-algebra.
We study the groupoid of $G$-$\mu$-displays over $(A, I)$
\[
G\mathchar`-\mathrm{Disp}_\mu(A, I)
\]
which is introduced in \cite{Ito-K23} following the theory of $G$-$\mu$-displays over higher frames developed by Lau \cite{Lau21} and the work of Bartling \cite{Bartling}.
The groupoid $G\mathchar`-\mathrm{Disp}_\mu(A, I)$ is equivalent to the groupoid of \textit{$G$-Breuil--Kisin modules of type $\mu$} over $(A, I)$, which may be more familiar to the reader.
See Section \ref{Subsection:G-mu-diplays} for details.

Let us summarize the main results of this paper.
In the remainder of this introduction, we assume that $\mu$ is \textit{1-bounded}, that is, the weights of the action of $\G_m$ on
the Lie algebra of $G_{W(k)}$ induced by $g \mapsto \mu(t)^{-1}g\mu(t)$ are $\leq 1$ (\cite[Definition 6.3.1]{Lau21}).
If $G$ is reductive, then $\mu$ is 1-bounded if and only if $\mu$ is minuscule.

\begin{defn}
    For a $p$-adically complete ring $R$ over $W(k)$,
the groupoid of \textit{prismatic $G$-$\mu$-displays} over $R$ is defined to be
\[
G\mathchar`-\mathrm{Disp}_\mu((R)_{\Prism}):= {2-\varprojlim}_{(A, I) \in (R)_{\Prism}} G\mathchar`-\mathrm{Disp}_\mu(A, I).
\]
Here $(R)_{\Prism}$ is the absolute prismatic site of $R$, whose objects are the bounded prisms $(A, I)$ equipped with a homomorphism $g \colon R \to A/I$.
For a prismatic $G$-$\mu$-display $\mathfrak{Q}$ over $R$ and an object
$((A, I), g \colon R \to A/I) \in (R)_{\Prism}$,
the image of $\mathfrak{Q}$ in $G\mathchar`-\mathrm{Disp}_\mu(A, I)$ is denoted by 
$\mathfrak{Q}_{(A, I)}$ or $\mathfrak{Q}_{g}$.
\end{defn}

Let $\mathcal{C}_{W(k)}$ be the category of complete regular local rings $R$ over $W(k)$ with residue field $k$.
Let $R \in \mathcal{C}_{W(k)}$.
The prism
$(W(k), (p))$ 
with the natural homomorphism $R \to k$ is an object of $(R)_{\Prism}$, which we consider as ``the base point''.
Let $\mathcal{Q}$ be a $G$-$\mu$-display over $(W(k), (p))$ (or equivalently a prismatic $G$-$\mu$-display over $k$).

\begin{defn}\label{Definition:deformations introduction}
A \textit{deformation} of $\mathcal{Q}$ over $R$ is
a prismatic $G$-$\mu$-display $\mathfrak{Q}$ over $R$
together with an isomorphism
$\mathfrak{Q}_{(W(k), (p))} \overset{\sim}{\to} \mathcal{Q}$
of $G$-$\mu$-displays over $(W(k), (p))$.
We say that 
$\mathfrak{Q}$
is a \textit{universal deformation} of $\mathcal{Q}$ if, for any $R' \in \mathcal{C}_{W(k)}$ and
any deformation $\mathfrak{Q}'$ of $\mathcal{Q}$ over $R'$,
there exists a unique local homomorphism
$
h \colon R \to R'
$
over $W(k)$ such that
the base change of $\mathfrak{Q}$
along $h$ is isomorphic to
$\mathfrak{Q}'$
as a deformation of $\mathcal{Q}$ over $R'$.
\end{defn}

Our first main result concerns the existence of a universal deformation of $\mathcal{Q}$.
Let
$U^{-}_{\mu} \subset G_{W(k)}$
be the closed subgroup scheme
consisting of elements $g \in G_{W(k)}$
such that $\lim_{t \to 0} \mu(t)^{-1}g\mu(t)=1$.
Let
$R_{G, \mu}$
be the completed local ring of $U^{-}_{\mu}$ at the identity element $1 \in U^{-}_{\mu}$.
Since $U^{-}_{\mu}$ is smooth,
there exists an isomorphism over $W(k)$
\[
R_{G, \mu} \simeq W(k)[[t_1, \dotsc, t_r]],
\]
and in particular $R_{G, \mu} \in \mathcal{C}_{W(k)}$.

\begin{thm}[Theorem \ref{Theorem:Existence of universal deformation}]\label{Theorem:existence of universal deformation introduction}
    Let $\mathcal{Q}$ be a $G$-$\mu$-display over $(W(k), (p))$.
    There exists a universal deformation
    $\mathfrak{Q}^{\mathrm{univ}}$
of $\mathcal{Q}$ over
$R_{G, \mu}$.
\end{thm}

We will show that $\mathfrak{Q}^{\mathrm{univ}}$ possesses a similar universal property for a larger class of deformations of $\mathcal{Q}$.
To explain this, we need to introduce some notation.

We consider a prism \textit{of Breuil--Kisin type}
\[
(\mathfrak{S}, (\mathcal{E})):=(W(k)[[t_1, \dotsc, t_n]], (\mathcal{E}))
\]
where $\mathcal{E} \in W(k)[[t_1, \dotsc, t_n]]$ is a formal power series whose constant term is $p$
and the Frobenius $\phi$ of $W(k)[[t_1, \dotsc, t_n]]$ is such that $\phi(t_i)=t^p_i$ for $1 \leq i \leq n$.
Here $n$ could be any nonnegative integer.
We set
$
R:=\mathfrak{S}/\mathcal{E},
$
which belongs to the category $\mathcal{C}_{W(k)}$.

The following fact plays a crucial role in our work.

\begin{rem}\label{Remark:Breuil-Kisin prism and deformation introduction}
Let 
$\mathfrak{Q}$
be a deformation of $\mathcal{Q}$ over $R$.
The associated $G$-$\mu$-display
$\mathfrak{Q}_{(\mathfrak{S}, (\mathcal{E}))}$
over $(\mathfrak{S}, (\mathcal{E}))$
is a deformation of $\mathcal{Q}$, i.e.\ it is equipped with an isomorphism between $\mathcal{Q}$ and the base change of $\mathfrak{Q}_{(\mathfrak{S}, (\mathcal{E}))}$ along the map 
$(\mathfrak{S}, (\mathcal{E})) \to (W(k), (p))$ defined by $t_i \mapsto 0$.
The construction
$\mathfrak{Q} \mapsto \mathfrak{Q}_{(\mathfrak{S}, (\mathcal{E}))}$
induces an equivalence from the category of deformations of $\mathcal{Q}$ over $R$
to that of deformations of $\mathcal{Q}$ over $(\mathfrak{S}, (\mathcal{E}))$.
This is a consequence of \cite[Theorem 1.2.1]{Ito-K23} (see also Theorem \ref{Theorem:main result on G displays over complete regular local rings}).
\end{rem}

Let 
$\mathfrak{Q}$
be a deformation of
$\mathcal{Q}$ over $R_{G, \mu}$.
Let
$\Hom(R_{G, \mu}, R)_e$
be the set of local homomorphisms $R_{G, \mu} \to R$ over $W(k)$.
Given a homomorphism $g \in \Hom(R_{G, \mu}, R)_e$,
we can regard $(\mathfrak{S}, (\mathcal{E}))$ as an object of $(R_{G, \mu})_{\Prism}$.
The construction
$g \mapsto \mathfrak{Q}_g$ induces a map of sets
\[
\ev_{\mathfrak{Q}} \colon 
\Hom(R_{G, \mu}, R)_e
\to
\left\{
\begin{tabular}{c}
      isomorphism classes of \\
     deformations of $\mathcal{Q}$ over $(\mathfrak{S}, (\mathcal{E}))$
\end{tabular}
\right\},
\]
which we call the \textit{evaluation map}.
If $\mathfrak{Q}$ is a universal deformation,
then the map $\ev_{\mathfrak{Q}}$ is bijective by Remark \ref{Remark:Breuil-Kisin prism and deformation introduction}.
(In fact, $\mathfrak{Q}$ is a universal deformation if and only if for every prism
$
(\mathfrak{S}, (\mathcal{E}))
$
of Breuil--Kisin type, the map $\ev_{\mathfrak{Q}}$ is bijective.)

We also consider the following classes of deformations:

\begin{defn}\label{Definition:Breuil-Kisin type intro}
    For every integer $m \geq 1$, we set
    \[
    (\mathfrak{S}_m, (\mathcal{E})):=(W(k)[[t_1, \dotsc, t_n]]/(t_1, \dotsc, t_n)^m, (\mathcal{E})),
    \]
    which is naturally a bounded prism.
    Let
    $\Hom(R_{G, \mu}, R/\mathfrak{m}^{m}_R)_e$
    be the set of local homomorphisms $R_{G, \mu} \to R/\mathfrak{m}^{m}_R$ over $W(k)$, where $\mathfrak{m}_R \subset R$ is the maximal ideal.
    We have the following map defined by $g \mapsto \mathfrak{Q}_g$:
    \begin{equation}\label{equation:evaluation map BK intro}
    \ev_{\mathfrak{Q}} \colon
    \Hom(R_{G, \mu}, R/\mathfrak{m}^{m}_R)_e
    \to 
    \left\{
    \begin{tabular}{c}
      isomorphism classes of \\
     deformations of $\mathcal{Q}$ over $(\mathfrak{S}_m, (\mathcal{E}))$
    \end{tabular}
    \right\}.
    \end{equation}
\end{defn}

\begin{defn}\label{Definition:perfectoid type intro}
Let $S$ be a perfectoid ring over $W(k)$ (\cite[Definition 3.5]{BMS}).
Let $a^\flat \in S^\flat$ be an element such that $a:=\theta([a^\flat]) \in S$ is a nonzerodivisor and $p=0$ in $S/a$.
(See Section \ref{Subsection:OE prisms} for the notation used here.)
    For every integer $m \geq 1$,
    the pair
    \[
    (W(S^\flat)/[a^\flat]^m, I_S)
    \]
    is naturally a bounded prism, where $I_S$ is the kernel of the natural homomorphism
    $\theta \colon W(S^\flat)/[a^\flat]^m \to S/a^m$.
    Let
    $\Hom(R_{G, \mu}, S/a^m)_{e}$
    be
    the set of homomorphisms $R_{G, \mu} \to S/a^m$ over $W(k)$ lifting
    the composition $R_{G, \mu} \to k \to S/a$.
    We regard $\mathcal{Q}$ as a $G$-$\mu$-display over $(W(S^\flat)/[a^\flat], I_S)$
    by base change.
    We have the following map defined by $g \mapsto \mathfrak{Q}_g$:
    \begin{equation}\label{equation:evaluation map perfectoid intro}
        \ev_{\mathfrak{Q}} \colon \Hom(R_{G, \mu}, S/a^m)_{e}
    \to
    \left\{
    \begin{tabular}{c}
      isomorphism classes of \\
     deformations of $\mathcal{Q}$ over $(W(S^\flat)/[a^\flat]^m, I_S)$
    \end{tabular}
    \right\}.
    \end{equation}
\end{defn}

We can now state our second main result:

\begin{thm}[Theorem \ref{Theorem:Existence of universal deformation}]\label{Theorem:universal deformation (BK) and (Perfd) introduction}
If $\mathfrak{Q}$ is a universal deformation of $\mathcal{Q}$ over
$R_{G, \mu}$,
then the maps $(\ref{equation:evaluation map BK intro})$ and $(\ref{equation:evaluation map perfectoid intro})$
are bijective.
\end{thm}

Using the prism $(W(k)[[t]]/t^2, (p))$, we have the following characterization of universal deformations of $\mathcal{Q}$.

\begin{thm}[Theorem \ref{Theorem:characterization of universal family}]\label{Theorem:characterization of universal deformation introduction}
Let $\mathfrak{Q}$ be a deformation of $\mathcal{Q}$ over $R_{G, \mu}$.
If the map
\[
    \ev_{\mathfrak{Q}} \colon \Hom(R_{G, \mu}, k[[t]]/t^2)_{e} \to 
    \left\{
    \begin{tabular}{c}
      isomorphism classes of \\
     deformations of $\mathcal{Q}$ over $(W(k)[[t]]/t^2, (p))$
    \end{tabular}
    \right\}
\]
is surjective, then $\mathfrak{Q}$ is a universal deformation of $\mathcal{Q}$.
\end{thm}

The map $\ev_{\mathfrak{Q}}$ as in Theorem \ref{Theorem:characterization of universal deformation introduction} is called the \textit{Kodaira--Spencer map} in this paper.

\begin{rem}
In fact, by using \textit{$\O_E$-prisms} as introduced in \cite{Marks} and \cite{Ito-K23},
we will formulate and prove our results for a smooth affine group scheme $G$ over the ring of integers $\O_E$ of any finite extension $E$ of $\Q_p$.
(In \cite{Marks}, $\O_E$-prisms are called $E$-typical prisms.)
The argument for general $\O_E$ is the same as that for the case where $\O_E=\Z_p$.
For simplicity, the reader may assume that $\O_E=\Z_p$ on a first reading.
\end{rem}

Let $R$ be a quasisyntomic ring over $W(k)$ in the sense of \cite[Definition 4.10]{BMS2}.
In \cite[Section 8.2]{Ito-K23},
we defined the groupoid
$
G\mathchar`-F\mathchar`-\mathrm{Gauge}_\mu(R)
$
of \textit{prismatic $G$-$F$-gauges of type $\mu$} over $R$ and constructed a fully faithful functor
    \[
    G\mathchar`-F\mathchar`-\mathrm{Gauge}_\mu(R) \to G\mathchar`-\mathrm{Disp}_\mu((R)_{\Prism}),
    \]
following the theory of prismatic $F$-gauges introduced by Drinfeld and Bhatt--Lurie (cf.\ \cite{Drinfeld22}, \cite{BL}, \cite{BL2}, \cite{BhattGauge}).
If $R$ belongs to $\mathcal{C}_{W(k)}$, then the above functor is an equivalence (Proposition \ref{Proposition:G-F-gauges to G-displays}).
Thus Theorem \ref{Theorem:existence of universal deformation introduction} can be rephrased as follows:

\begin{thm}[Theorem \ref{Theorem:deformation theory for prismatic G-F-gauges of type mu}]\label{Theorem:prismatic G-F-gauge of type mu intro}
Let
$
\mathscr{Q} \in G\mathchar`-F\mathchar`-\mathrm{Gauge}_\mu(k)
$
be a prismatic $G$-$F$-gauge of type $\mu$ over $k$.
There exists a deformation $\mathscr{Q}^{\mathrm{univ}} \in G\mathchar`-F\mathchar`-\mathrm{Gauge}_\mu(R_{G, \mu})$ of $\mathscr{Q}$ over $R_{G, \mu}$ which is universal among deformations of $\mathscr{Q}$ over $R$ with $R \in \mathcal{C}_{W(k)}$.
\end{thm}

\begin{rem}\label{Remark:prismatic F-gauge intro}
    In light of Theorem \ref{Theorem:universal deformation (BK) and (Perfd) introduction}, one can expect that
    $\mathscr{Q}^{\mathrm{univ}}$ is also universal among deformations of $\mathscr{Q}$ over $R/\mathfrak{m}^{m}_R$ with $R \in \mathcal{C}_{W(k)}$ such that $\dim R =1$
    and over $S/a^m$ as in Definition \ref{Definition:perfectoid type intro}.
    We note that $R/\mathfrak{m}^{m}_R$ (with $\dim R =1$) and $S/a^m$ are quasisyntomic.
    In principle, it should be possible to develop the deformation theory for prismatic $G$-$F$-gauges of type $\mu$ in a more conceptual (but rather abstract) way, and the results presented in this paper should provide a practical way to understand it by using $G$-$\mu$-displays over prisms.
    We will only make some conjectures in Section \ref{Section:Consequences on prismatic F-gauges}, and will not pursue this direction here.
    See also Remark \ref{Remark:GMM}.
\end{rem}

\begin{rem}\label{Remark:comparison with Bultel Pappas work}
Assume that $G$ is reductive.
In \cite{Bultel-Pappas},
B\"ultel--Pappas
introduced 
the notion of
$G$-$\mu$-displays over $\Spec R$ for any $W(k)$-algebra $R$, by using the ring $W(R)$ of $p$-typical Witt vectors.
(In \cite{Bultel-Pappas}, $G$-$\mu$-displays are called $(G, \mu)$-displays.)
The groupoid of $G$-$\mu$-displays over $\Spec k$ in their sense is equivalent to 
$G\mathchar`-\mathrm{Disp}_\mu(W(k), (p))$.
Let
$\mathcal{Q}$
be a (banal) $G$-$\mu$-display over $\Spec k$.
Under the assumption that
$\mathcal{Q}$ is \textit{adjoint nilpotent} in the sense of \cite[Definition 3.4.2]{Bultel-Pappas},
they proved that the deformation functor of $\mathcal{Q}$
defined on the category $\Art_{W(k)}$ of artinian local rings over $W(k)$ with residue field $k$ is pro-representable by $R_{G, \mu}$.
See \cite[3.5.9]{Bultel-Pappas} for details.
The deformation theory developed in this paper applies to complete regular local rings $R \in \mathcal{C}_{W(k)}$ and their quotients $R/\mathfrak{m}^m_R$ without imposing any additional conditions on $G$-$\mu$-displays, such as the adjoint nilpotent condition.
\end{rem}

\begin{rem}\label{Remark:GMM}
    After this work was completed, and during the refereeing process,
    Gardner--Madapusi \cite{Gardner-Madapusi} developed the deformation theory for prismatic $G$-$F$-gauges of type $\mu$ using the stacky approach of Drinfeld and Bhatt--Lurie, from which the conjectures stated in Section \ref{Section:Consequences on prismatic F-gauges} will follow.
    More precisely, using their result, we can extend the definition of prismatic $G$-$F$-gauges of type $\mu$ to (not necessarily quasisyntomic) $p$-adically complete rings and then generalize our results to arbitrary artinian local rings in $\Art_{W(k)}$.
    Compared to \cite{Gardner-Madapusi}, the method developed in this paper can not be applied to arbitrary artinian local rings, but it does not require the use of some powerful results on derived algebraic geometry, such as Lurie's generalization of Artin's representability theorem.
    See also \cite{IKY} for the relation between our results and those in \cite{Gardner-Madapusi}.
\end{rem}

\subsection{Applications}\label{Subsection:Applications intro}

We give two applications of our results.
Assume that $G$ is a connected reductive group scheme over $\Z_p$ (or over $\O_E$).
The first application concerns the local representability 
and
the formal smoothness
of the \textit{integral local Shimura variety}
$\mathcal{M}^{\int}_{G, b, \mu}$
with hyperspecial level structure, which is introduced in \cite{Scholze-Weinstein} as a $v$-sheaf on the category of perfectoid spaces of characteristic $p$.
More precisely, we prove in Theorem \ref{Theorem:conjecture of Pappas-Rapoport} that the formal completions
$\widehat{\mathcal{M}^{\int}_{G, b, \mu}}_{/ x}$
of $\mathcal{M}^{\int}_{G, b, \mu}$ introduced in \cite{Gleason} are representable by the formal scheme $\Spf R_{G, \mu}$
(or rather $\Spf R_{G, \mu^{-1}}$, depending on the sign convention) by relating $\widehat{\mathcal{M}^{\int}_{G, b, \mu}}_{/ x}$ to universal deformations of prismatic $G$-$\mu$-displays.

\begin{rem}\label{Remark:Pappas-Rapoport conjecture intro}
The above result (Theorem \ref{Theorem:conjecture of Pappas-Rapoport}) implies that a conjecture of Pappas--Rapoport \cite[Conjecture 3.3.5]{PappasRapoport21}
in the hyperspecial case, which was originally proposed by Gleason \cite[Conjecture 1]{Gleasonthesis}, holds true.
In \cite{PappasRapoport22}, Pappas--Rapoport proved their conjecture (in the more general case where $G$ is parahoric) under the assumption that $p \geq 3$ and $\mathcal{M}^{\int}_{G, b, \mu}$ is of abelian type.
If $p=2$, the same result was obtained under the additional assumption that $G_{\Q_p}$ is of type $A$ or $C$.
Our approach is completely different from that of \cite{PappasRapoport22},
and we do not need to impose these assumptions.
Our result also generalizes a result of Bartling \cite{Bartling}, in which the adjoint nilpotent condition of \cite{Bultel-Pappas} is imposed.
See Section \ref{Subsection:Moduli spaces of $G$-shtukas} for details.
\end{rem}

\begin{rem}
    In \cite{IKY},
    our universal deformations of prismatic $G$-$\mu$-displays are used to give a prismatic characterization of integral canonical models of (global) Shimura varieties of abelian type with hyperspecial level structure, which is inspired by \cite[Theorem 4.2.4]{PappasRapoport21}.
\end{rem}

As a second application,
we revisit and extend some classification results of $p$-divisible groups.

\begin{rem}\label{Remark:the part which can be deduced from Anschutz-LeBras}
Let $R$ be a quasisyntomic ring.
In \cite{Anschutz-LeBras},
Ansch\"utz--Le Bras constructed a contravariant functor from the category of $p$-divisible groups over $\Spec R$
to the category 
of admissible prismatic Dieudonn\'e crystals over $R$ (\cite[Definition 4.5]{Anschutz-LeBras}), which is called the \textit{prismatic Dieudonn\'e functor}.
They proved that the prismatic Dieudonn\'e functor is an anti-equivalence; see \cite[Theorem 4.74]{Anschutz-LeBras}.
For any $R \in \mathcal{C}_{W(k)}$,
our prismatic $\GL_N$-$\mu$-displays over $R$ can be viewed as
admissible prismatic Dieudonn\'e crystals over $R$ (with some additional conditions).
Then Theorem \ref{Theorem:existence of universal deformation introduction}
and Theorem \ref{Theorem:characterization of universal deformation introduction}
follow immediately from \cite[Theorem 4.74]{Anschutz-LeBras}, together with the deformation theory for $p$-divisible groups.
On the other hand,
Theorem \ref{Theorem:universal deformation (BK) and (Perfd) introduction}
is a new result as far as we know.
The proofs of 
Theorem \ref{Theorem:existence of universal deformation introduction}, Theorem \ref{Theorem:universal deformation (BK) and (Perfd) introduction}, and Theorem \ref{Theorem:characterization of universal deformation introduction} presented in this paper do not rely on $p$-divisible groups.
\end{rem}

By using our results, we can give an alternative proof
of the classification result of Ansch\"utz--Le Bras mentioned above for any $R \in \mathcal{C}_{W(k)}$.
In fact,
Theorem \ref{Theorem:universal deformation (BK) and (Perfd) introduction}
enables us to prove the following results:
\begin{enumerate}
    \item 
    (Theorem \ref{Theorem:classification of p-divisible group:torsion regular local ring}).
    Let the notation be as in Definition \ref{Definition:Breuil-Kisin type intro}.
    For every $m \geq 1$, the prismatic Dieudonn\'e functor induces
an anti-equivalence
\[
\left.
\left\{
\begin{tabular}{c}
     $p$-divisible groups \\ over $\Spec R/\mathfrak{m}^m_R$
\end{tabular}
\right\}
\right.
\overset{\sim}{\to}
\left\{
\begin{tabular}{c}
     minuscule Breuil--Kisin modules \\
     over \ $(\mathfrak{S}_m, (\mathcal{E}))$
\end{tabular}
\right\}.
\]
    \item (Theorem \ref{Theorem:classification of p-divisible group:torsion valuation ring}). Let $\O_C$ be a $p$-adically complete valuation ring of rank $1$ with algebraically closed fraction field $C$.
    Let $\varpi \in \O_C$ be a pseudo-uniformizer.
    For every $m \geq 1$, the prismatic Dieudonn\'e functor induces
    an anti-equivalence
    \[
\left.
\left\{
\begin{tabular}{c}
     $p$-divisible groups \\ over $\Spec \O_C/\varpi^m$
\end{tabular}
\right\}
\right.
\overset{\sim}{\to}
\left\{
\begin{tabular}{c}
     minuscule Breuil--Kisin modules \\
     over \ $(W(\O_{C^\flat})/[\varpi^\flat]^{m}, I_{\O_C})$
\end{tabular}
\right\}.
\]
\end{enumerate}
See Example \ref{Example:GLn displays} for the definition of minuscule Breuil--Kisin modules.

\begin{rem}
    An equivalence as in (1) was previously obtained by Lau \cite{Lau14} using
Dieudonn\'e displays
and crystalline Dieudonn\'e theory, where special efforts were made in the case of $p=2$.
(This result plays an essential role in the construction of $2$-adic integral canonical models of Shimura varieties of abelian type with hyperspecial level structure in \cite{Kim-MadapusiPera}.)

By combining (2) with a theorem of Fargues \cite[Theorem 14.1.1]{Scholze-Weinstein},
we obtain an alternative proof of a result of Scholze--Weinstein \cite[Theorem B]{ScholzeWeinsteinModuli},
which says that, in the case where $C$ is of characteristic $0$, there exists an equivalence between the category of $p$-divisible groups over $\Spec \O_C$
and the category of free $\Z_p$-modules $T$ of finite rank together with a $C$-subspace of $T \otimes_{\Z_p} C$.
See Section \ref{Subsection:Classifications of p-divisible groups} for details.
\end{rem}

\subsection{Strategy of the proof}\label{Subsection:Strategy of the proof}

We remark that if there exists
a deformation
$\mathfrak{Q}$ of $\mathcal{Q}$ over $R_{G, \mu}$
such that
the maps $(\ref{equation:evaluation map BK intro})$ and $(\ref{equation:evaluation map perfectoid intro})$
for $\mathfrak{Q}$
are bijective,
then $\mathfrak{Q}$ is a universal deformation, and moreover Theorem \ref{Theorem:characterization of universal deformation introduction} follows immediately.
We briefly explain how to construct such a deformation $\mathfrak{Q}$.
We choose a certain prism
$(\mathfrak{S}^{\univ}, (\mathcal{E}^{\univ}))$
of Breuil--Kisin type with an isomorphism
$R_{G, \mu} \simeq \mathfrak{S}^{\univ}/\mathcal{E}^{\univ}$ over $W(k)$.
By Remark \ref{Remark:Breuil-Kisin prism and deformation introduction}, giving a deformation of $\mathcal{Q}$ over $R_{G, \mu}$ is equivalent to giving a deformation of $\mathcal{Q}$ over  $(\mathfrak{S}^{\univ}, (\mathcal{E}^{\univ}))$.
In practice the latter is much easier.
We will construct a deformation $\mathscr{Q}$ of $\mathcal{Q}$ over $(\mathfrak{S}^{\univ}, (\mathcal{E}^{\univ}))$ such that 
the maps (\ref{equation:evaluation map BK intro})
and
(\ref{equation:evaluation map perfectoid intro}) for the corresponding deformation $\mathfrak{Q}$ over $R_{G, \mu}$ are bijective.

The key ingredient is the \textit{Grothendieck--Messing deformation theory} for $G$-$\mu$-displays developed in Section \ref{Section:The Grothendieck--Messing deformation theory for G displays}.
Let
$\mathcal{Q}'$
be a $G$-$\mu$-display over
$(\mathfrak{S}_m, (\mathcal{E}))$
(resp.\ $(W(S^\flat)/[a^\flat]^m, I_S)$).
Then the Grothendieck--Messing deformation theory says that
deformations of $\mathcal{Q}'$ over 
$(\mathfrak{S}_{m+1}, (\mathcal{E}))$
(resp.\ $(W(S^\flat)/[a^\flat]^{m+1}, I_S)$)
are classified by lifts of the Hodge filtration of $\mathcal{Q}'$; see Theorem \ref{Theorem:GM deformation} for the precise statement.
Here the 1-boundedness of $\mu$ is essential.
In view of this result, it is not difficult to find a candidate for $\mathscr{Q}$.
In fact, our construction is an analogue of the construction of universal $p$-divisible groups given by Faltings \cite[Section 7]{Faltings99}.

The proof of the bijectivity of the maps (\ref{equation:evaluation map BK intro}) and (\ref{equation:evaluation map perfectoid intro}) for $\mathfrak{Q}$ goes as follows.
We fix a homomorphism
$g \in \Hom(R_{G, \mu}, S/a^{m})_{e}$
and let
$\Hom(R_{G, \mu}, S/a^{m+1})_{g}$
be the set of homomorphisms $h \colon R_{G, \mu} \to S/a^{m+1}$ over $W(k)$
which are lifts of $g$.
To prove that the map $(\ref{equation:evaluation map perfectoid intro})$ is bijective, it suffices to show that (for all $m$ and $g$) the map
\begin{equation}\label{equation:evaluation map linear intro}
    \ev_{\mathfrak{Q}} \colon \Hom(R_{G, \mu}, S/a^{m+1})_{g}
    \to
    \left\{
    \begin{tabular}{c}
      isomorphism classes of \\
     deformations of $\mathfrak{Q}_g$ over $(W(S^\flat)/[a^\flat]^{m+1}, I_S)$
    \end{tabular}
    \right\}
\end{equation}
defined by $h \mapsto \mathfrak{Q}_h$ is bijective.
After choosing a lift $h \in \Hom(R_{G, \mu}, S/a^{m+1})_{g}$, we can endow
the set $\Hom(R_{G, \mu}, S/a^{m+1})_{g}$
with the structure of an $S/a$-module such that
$h$ is the zero element.
The target of the map (\ref{equation:evaluation map linear intro})
can be also regarded as an $S/a$-module by the Grothendieck--Messing deformation theory.
In fact, we will show that the map
(\ref{equation:evaluation map linear intro})
is not only a bijection, but also an \textit{$S/a$-linear isomorphism} (cf.\ Theorem \ref{Theorem:existence of deformations with (Perfd-lin)}).
For this, we may assume that any element of $S$ admits a $p$-power root by Andr\'e's flatness lemma, and then we can prove the assertion by explicit computation.
From this more precise statement, we can deduce that the map $(\ref{equation:evaluation map BK intro})$ is bijective (cf.\ Proposition \ref{Proposition:(Perfd-lin) implies (BK-lin)}).

\subsection{The structure of this paper}\label{Subsection:The structure of this paper}

This paper is organized as follows.
In Section \ref{Section:Review of prismatic G-mu-displays}, we review the theory of $\O_E$-prisms and $G$-$\mu$-displays over (bounded) $\O_E$-prisms.
In Section \ref{Section:The Grothendieck--Messing deformation theory for G displays}, we prove some basic results on deformations of $G$-$\mu$-displays,
and establish the Grothendieck--Messing deformation theory for $G$-$\mu$-displays.
In Section \ref{Section:Universal deformations}, we state and prove the main results of this paper.

In Section \ref{Section:Integral models of local Shimura varieties} and Section \ref{Section:Comparison with universal deformations of $p$-divisible groups}, we give applications of our results to integral local Shimura varieties and $p$-divisible groups, respectively.
Finally, in Section \ref{Section:Consequences on prismatic F-gauges}, we discuss some consequences of our deformation theory on prismatic $F$-gauges, or more precisely, on prismatic $G$-$F$-gauges of type $\mu$.

\subsection*{Notation} \label{Subsection:Notation}
This paper is a continuation of our previous work \cite{Ito-K23}.
Unless explicitly stated otherwise, we use definitions and notation from \cite{Ito-K23}.

We recall some notation from \cite{Ito-K23}.
All rings are commutative and unital.
For a scheme $X$ over $\Spec R$ and a ring homomorphism
$f \colon R \to R'$,
the base change $X \times_{\Spec R} \Spec R'$ is denoted by
$X_{R'}$ or $f^*X$.
We use similar notation for the base change of group schemes,
$p$-divisible groups, etc.
All actions of groups will be right actions, unless otherwise stated.
A groupoid is a category whose morphisms are all invertible.
For a category $\mathcal{C}$,
let
$\mathcal{C}^{\simeq}$
be the largest groupoid contained in $\mathcal{C}$.

\section{Review of prismatic $G$-$\mu$-displays}\label{Section:Review of prismatic G-mu-displays}

Throughout this paper, we fix a prime number $p$.
Let $E$ be a finite extension of $\Q_p$ with ring of integers $\O_E$ and residue field $\F_q$.
Here $\F_q$ is a finite field with $q$ elements.
We fix a uniformizer $\pi \in \O_E$ for simplicity.
Let $k$ be a perfect field containing $\F_q$ and
we set
$
\O := W(k) \otimes_{W(\F_q)} \O_E.
$
Let $G$ be a smooth affine group scheme over $\O_E$
and
let
$
\mu \colon \G_m \to G_{\O}:=G \times_{\Spec \O_E} \Spec \O
$
be a cocharacter.

In this section, we first recall the notion of $\O_E$-prisms as introduced in \cite{Marks} and \cite{Ito-K23}.
We also provide some preliminary results on certain specific $\O_E$-prisms which will be used in this paper.
Then we review the definition and basic properties of $G$-$\mu$-displays over bounded $\O_E$-prisms developed in \cite{Ito-K23}.
More details can be found in \cite{Ito-K23}.

\subsection{$\O_E$-prisms}\label{Subsection:OE prisms}

Let $A$ be an $\O_E$-algebra.
    A \textit{$\delta_E$-structure} on $A$ is a map $\delta_E \colon A \to A$ of sets with the following properties:
\begin{itemize}
    \item $\delta_E(xy)=x^q\delta_E(y)+y^q\delta_E(x)+\pi\delta_E(x)\delta_E(y)$.
    \item $\delta_E(x+y)=\delta_E(x)+\delta_E(y)+(x^q+y^q-(x+y)^q)/\pi$.
    \item For an element $x \in \O_E$, we have $\delta_E(x)=(x-x^q)/\pi$.
\end{itemize}
A \textit{$\delta_E$-ring} is
an $\O_E$-algebra $A$ equipped with a $\delta_E$-structure.
We define
\[
\phi_A \colon A \to A, \quad x \mapsto x^q+\pi\delta_E(x).
\]
Then $\phi_A$ is a homomorphism of $\O_E$-algebras and is a lift of
the $q$-th power Frobenius $A/\pi \to A/\pi$, $x \mapsto x^q$.
The homomorphism $\phi_A$ is called the \textit{Frobenius} of the $\delta_E$-ring $A$.
When there is no ambiguity, we omit the subscript and simply write $\phi=\phi_A$.

An \textit{$\O_E$-prism} is a pair $(A, I)$ of a $\delta_E$-ring $A$ and a Cartier divisor $I \subset A$
such that $A$ is derived $(\pi, I)$-adically complete and $\pi \in I + \phi(I)A$.
We say that $(A, I)$ is
\textit{bounded}
if $A/I$ has bounded $p^\infty$-torsion, i.e.\ $(A/I)[p^n]=(A/I)[p^\infty]$ for some integer $n \geq 1$.
In this case $A$ is $(\pi, I)$-adically complete (\cite[Remark 2.3.2]{Ito-K23}).
If $\O_E=\Z_p$, then bounded $\O_E$-prisms are the same as bounded prisms in the sense of \cite{BS}.
We say that $(A, I)$ is \textit{orientable}
if $I$ is principal.
An $\O_E$-prism $(A, I)$ with a homomorphism $\O \to A$ of $\delta_E$-rings is called an $\O_E$-prism \textit{over} $\O$.
We refer to \cite{Marks} and \cite{Ito-K23} for details.

We give some examples of $\O_E$-prisms, which play a central role in this paper.
We set
\[
\mathfrak{S}_\O:=\O[[t_1, \dotsc, t_n]]
\]
for $n \geq 0$, which admits a unique $\delta_E$-structure such that
the Frobenius
$\phi \colon \mathfrak{S}_\O \to \mathfrak{S}_\O$ 
is given by $\phi(t_i)=t^q_i$ ($1 \leq i \leq n$).
For every integer $m \geq 1$,
the quotient
\[
\mathfrak{S}_{\O, m}:=\O[[t_1, \dotsc, t_n]]/(t_1, \dotsc, t_n)^m
\]
admits a unique $\delta_E$-structure that is compatible with the one on $\mathfrak{S}_\O$.

\begin{ex}\label{Example:Breuil-Kisin type frame}
Let $\mathcal{E} \in \mathfrak{S}_\O$ be a formal power series whose constant term is a uniformizer of $\O$.
The pair
$
(\mathfrak{S}_\O, (\mathcal{E}))
$
is a bounded $\O_E$-prism over $\O$ (\cite[Proposition 2.3.8]{Ito-K23}),
which we call an $\O_E$-prism \textit{of Breuil--Kisin type} in this paper.
Here $n$ could be any nonnegative integer.
For any $m \geq 1$,
the pair
$
(\mathfrak{S}_{\O, m}, (\mathcal{E}))
$
is also a bounded $\O_E$-prism over $\O$.
Here we denote the image of $\mathcal{E}$ in $\mathfrak{S}_{\O, m}$ by the same symbol.
\end{ex}

Our next example is related to (integral) perfectoid rings.
We refer to \cite[Section 3]{BMS} and \cite[Section 2]{CS} for the definition and basic properties of perfectoid rings.

Let $S$ be a perfectoid ring over $\O$
(i.e.\ $S$ is a perfectoid ring with a ring homomorphism $\O \to S$).
Let
$
S^\flat:=\varprojlim_{x \mapsto x^p} S/p
$
be the \textit{tilt} of $S$,
which is a perfect $k$-algebra.
Then
\[
W_{\O_E}(S^\flat):=W(S^\flat) \otimes_{W(\F_q)} \O_E
\]
is naturally an $\O$-algebra.
Let 
$\phi \colon W_{\O_E}(S^\flat) \to W_{\O_E}(S^\flat)$
be
the base change of the $q$-th power Frobenius of $W(S^\flat)$.
Since $W_{\O_E}(S^\flat)$ is $\pi$-torsion free,
we obtain the corresponding $\delta_E$-structure on $W_{\O_E}(S^\flat)$.
Moreover, for any element $x \in S^\flat$,
the quotient
$W_{\O_E}(S^\flat)/[x]$
admits a unique $\delta_E$-structure that is compatible with the one on $W_{\O_E}(S^\flat)$. Here $[-]$ denotes the Teichm\"uller lift.

We note the following fact.

\begin{lem}\label{Lemma:pi torsion free of witt vectors}
For any element $x \in S^\flat$,
the quotient $W_{\O_E}(S^\flat)/[x]$ is $\pi$-torsion free (or equivalently, $p$-torsion free).
\end{lem}

\begin{proof}
Since $W(\F_q) \to \O_E$ is flat,
it is enough to show that $W(S^\flat)/[x]$ is $p$-torsion free.
Let
$y \in W(S^\flat)$ be an element such that
$py=[x]z$ for some $z \in W(S^\flat)$.
We want to show that $y \in ([x])$.
For the Witt vector expansions
$y=(y_0, y_1, \dotsc)$ and
$z=(z_0, z_1, \dotsc)$,
the equality
$py=[x]z$
implies that
$y^p_i=x^{p^{i+1}}z_{i+1}$
for every $i \geq 0$.
Since $S^\flat$ is perfect, we have $y_i=x^{p^{i}}z^{1/p}_{i+1}$
for every $i \geq 0$, and thus
$
y=(y_0, y_1, \dotsc)=[x]\cdot(z^{1/p}_{1}, z^{1/p}_{2}, \dotsc) \in ([x]).
$
\end{proof}

Let 
$
\theta \colon W(S^\flat) \to S
$
be the unique ring homomorphism whose reduction modulo $p$ is the projection map
$S^\flat \to S/p$, $(x_0, x_1, \dotsc) \mapsto x_0$.
Let $\theta_{\O_E} \colon W_{\O_E}(S^\flat) \to S$ be
the homomorphism induced from $\theta$.
We write
$I_S:=\Ker \theta_{\O_E}$ for the kernel of $\theta_{\O_E}$.


\begin{prop}\label{Proposition:perfectoid type}
Let $S$ be a perfectoid ring over $\O$.
\begin{enumerate}
    \item 
    The pair
    $
    (W_{\O_E}(S^\flat), I_S)
    $
    is an orientable and bounded $\O_E$-prism over $\O$.
    \item 
    Let $a^\flat \in S^\flat$ be an element such that
    $a:=\theta([a^\flat]) \in S$ is a nonzerodivisor and we have $\pi \in (a)$ in $S$.
    Let $m \geq 1$ be an integer.
    Then
    \[
    (W_{\O_E}(S^\flat)/[a^\flat]^m, I_S)
    \]
    is an orientable and bounded $\O_E$-prism over $\O$.
    Here we denote the image of $I_S$ in $W_{\O_E}(S^\flat)/[a^\flat]^m$ by the same notation.
\end{enumerate}
\end{prop}

\begin{proof}
(1) See \cite[Proposition 2.4.3]{Ito-K23}.

(2) We write $A_m:=W_{\O_E}(S^\flat)/[a^\flat]^m$.
Let $\xi \in I_S$ be a generator.
We shall show that $\xi$ is a nonzerodivisor in $A_m$.
Let $x, y \in W_{\O_E}(S^\flat)$ be elements satisfying $\xi x=[a^\flat]^m y$.
Since $a \in S$ is a nonzerodivisor, we have $y \in I_S=(\xi)$.
Since $\xi$ is a nonzerodivisor in $W_{\O_E}(S^\flat)$, we obtain $x \in ([a^\flat]^m)$.
This proves that $\xi$ is a nonzerodivisor in $A_m$.
Moreover, since $\pi \in (\xi, \phi(\xi))$ in $W_{\O_E}(S^\flat)$ by (1), we have $\pi \in (\xi, \phi(\xi))$ in $A_m$ as well.

The above argument, together with the fact that $W_{\O_E}(S^\flat)$ is $\xi$-adically complete, implies that $W_{\O_E}(S^\flat)$ is $[a^\flat]$-torsion free.
Since $W_{\O_E}(S^\flat)$ is derived $(\pi, \xi)$-adically complete, so is $A_m$.
It is clear that
$A_m/\xi=S/a^m$ has bounded $p^\infty$-torsion
(since $p^m$=0 in $S/a^m$).
In conclusion, $(A_m, I_S)$ is an orientable and bounded $\O_E$-prism.
\end{proof}


For the purposes of this paper, it will be convenient to introduce the following (slightly nonstandard) definition:

\begin{defn}\label{Definition:perfectoid pair}
A pair
$(S, a^\flat)$
consisting of a perfectoid ring $S$ over $\O$ and
an element $a^\flat \in S^\flat$ is called a \textit{perfectoid pair} over $\O$ if $a:=\theta([a^\flat]) \in S$ is a nonzerodivisor and we have $\pi \in (a)$ in $S$.
\end{defn}

\subsection{Prismatic sites}\label{Subsection:Prismatic sites}

We say that
a map
$(A, I) \to (A', I')$ of bounded $\O_E$-prisms
is a \textit{flat map}
(resp.\ a \textit{faithfully flat map})
if $A \to A'$ is 
$(\pi,I)$-completely flat
(resp.\
$(\pi,I)$-completely faithfully flat) in the sense of \cite[Notation 1.2]{BS}.

Let $R$ be a $\pi$-adically complete $\O_E$-algebra.
As in \cite[Definition 2.5.3]{Ito-K23}, let
\[
(R)_{\Prism, \O_E}
\]
be the category of bounded $\O_E$-prisms $(A, I)$ together with a homomorphism $R \to A/I$ of $\O_E$-algebras.
We endow
the opposite category
$
(R)^{\op}_{\Prism, \O_E}
$
with the flat topology, that is, the topology generated by the faithfully flat maps.
It follows from \cite[Remark 2.5.4]{Ito-K23}
that
$
(R)^{\op}_{\Prism, \O_E}
$
is a site.
If $\O_E=\Z_p$, then $(R)_{\Prism, \O_E}$ is the same as the category $(R)_{\Prism}$ introduced in \cite[Remark 4.7]{BS}.

\begin{rem}\label{Remark:structure sheaf}
The functors
\begin{align*}
    \O_{\Prism} &\colon (R)_{\Prism, \O_E} \to \mathrm{Set}, \quad (A, I) \mapsto A, \\
    \O_{\overline{\Prism}} &\colon
(R)_{\Prism, \O_E} \to \mathrm{Set}, \quad (A, I) \mapsto A/I
\end{align*}
form sheaves with respect to the flat topology.
Here $\mathrm{Set}$ is the category of sets.
\end{rem}

Let $(A, I)$ be a bounded $\O_E$-prism.
We write
\[
(A, I)_\et
\]
for the category of $(\pi, I)$-completely \'etale $A$-algebras (in the sense of \cite[Notation 1.2]{BS}).
We endow
$
(A, I)^{\op}_\et
$
with the $(\pi, I)$-completely \'etale topology, that is, the topology generated by the $(\pi, I)$-completely \'etale coverings $B \to B'$ (i.e.\ $B \to B'$ is $(\pi, I)$-completely \'etale and $(\pi, I)$-completely faithfully flat).
Every $B \in (A, I)_\et$ admits a unique $\delta_E$-structure compatible with that on $A$, and the pair $(B, IB)$ is a bounded $\O_E$-prism by \cite[Lemma 2.5.10]{Ito-K23}.

We recall the following fact (see \cite[Example 2.5.11]{Ito-K23} and the references therein):

\begin{ex}\label{Example:perfectoid ring etale morphism}
Let $S$ be a perfectoid ring over $\O_E$
and let $S \to S'$ be a $\pi$-completely \'etale homomorphism.
Then $S'$ is a perfectoid ring, and the induced homomorphism
$W_{\O_E}(S^\flat) \to W_{\O_E}(S'^\flat)$ is $(\pi, I_S)$-completely \'etale.
\end{ex}


\begin{lem}\label{Lemma:perfectoid ring etale morphism:torsion case}
Let 
$(S, a^\flat)$
be a perfectoid pair over $\O$.
\begin{enumerate}
    \item For a $\pi$-completely flat $S$-algebra $S'$, the element $a$ is a nonzerodivisor in $S'$.
    In particular, if $S'$ is a perfectoid ring, then $(S', a^\flat)$ is a perfectoid pair over $\O$.
    \item 
    For a $\pi$-completely \'etale $S$-algebra $S'$, the homomorphism
    \[
    W_{\O_E}(S^\flat)/[a^\flat]^m \to W_{\O_E}(S'^\flat)/[a^\flat]^m
    \]
    is $(\pi, I_S)$-completely \'etale for every $m \geq 1$.
    Conversely, any $(\pi, I_S)$-completely \'etale $W_{\O_E}(S^\flat)/[a^\flat]^m$-algebra is of the form
    $W_{\O_E}(S'^\flat)/[a^\flat]^m$ for some
    $\pi$-completely \'etale $S$-algebra $S'$.
\end{enumerate}
\end{lem}

\begin{proof}
(1) Since $\pi=0$ in $S/a$, we see that $S' \otimes^{\L}_S S/a$ is concentrated in degree $0$.
This means that $a$ is a nonzerodivisor in $S'$.

(2) This immediately follows from (1). See also the proof of \cite[Lemma 2.5.9]{Ito-K23}.
\end{proof}

\subsection{Display groups}\label{Subsection:Display groups}

Let $A$ be an $\O$-algebra
and $I \subset A$ an ideal generated by a nonzerodivisor $d \in A$.
We assume that $A$ is $I$-adically complete.
We set $A[1/I]:=A[1/d]$.
As in \cite[Section 4]{Ito-K23}, the \textit{display group}
$G_\mu(A, I)$
is defined by
\[
G_\mu(A, I):=\{ \, g \in G(A) \, \vert \, \mu(d)g\mu(d)^{-1} \, \, \text{lies in} \, \, G(A) \subset G(A[1/I]) \, \}.
\]
We shall recall a structural result about $G_\mu(A, I)$.

Let
$P_\mu, U^{-}_{\mu} \subset G_\O$
be the closed subgroup schemes defined by, for every $\O$-algebra $R$,
\begin{align*}
    P_\mu(R)&=\{ \, g \in G(R) \, \vert \, \lim_{t \to 0} \mu(t)g\mu(t)^{-1} \, \text{exists} \, \},\\
    U^{-}_{\mu}(R)&=\{ \, g \in G(R) \, \vert \, \lim_{t \to 0} \mu(t)^{-1}g\mu(t)=1 \, \}.
\end{align*}
The group schemes $P_\mu$ and $U^{-}_{\mu}$ are smooth over $\O$.

\begin{defn}[{\cite[Definition 6.3.1]{Lau21}}]\label{Definition:1-bounded}
The cocharacter
$\mu \colon \G_m \to G_{\O}$
is called \textit{1-bounded} if the weights of the action of $\G_m$ on the Lie algebra
$\Lie(G_\O)$
induced by $g \mapsto \mu(t)^{-1}g\mu(t)$
are $\leq 1$.
\end{defn}

\begin{rem}
    If $G$ is a reductive group scheme over $\O_E$, then $\mu$ is 1-bounded if and only if $\mu$ is minuscule.
\end{rem}

\begin{prop}\label{Proposition:BB isomorphism}
\ 
\begin{enumerate}
    \item We have $P_\mu(A) \subset G_\mu(A, I)$, and the image of $G_\mu(A, I)$ under the the projection $G(A) \to G(A/I)$ is equal to $P_\mu(A/I)$.
    \item The multiplication map
\[
    (U^{-}_{\mu}(A) \cap G_\mu(A, I)) \times P_\mu(A) \to G_\mu(A, I)
\]
is bijiective.
    \item Assume that $\mu$ is 1-bounded.
Then $G_\mu(A, I)$ coincides with the inverse image of $P_\mu(A/I)$ in $G(A)$ under the projection $G(A) \to G(A/I)$.
Moreover, we have the following bijection:
\[
G(A)/G_\mu(A, I) \overset{\sim}{\to} G(A/I)/P_\mu(A/I).
\]
\end{enumerate}
\end{prop}

\begin{proof}
    The assertion (1) follows from \cite[Lemma 4.2.2]{Ito-K23}.
    The assertion (2) is \cite[Proposition 4.2.8]{Ito-K23}.
    For the assertion (3), see \cite[Proposition 4.2.9]{Ito-K23}.
\end{proof}

\subsection{$G$-$\mu$-displays}\label{Subsection:G-mu-diplays}

Let
$(A, I)$
be a bounded $\O_E$-prism over $\O$.
We recall the definition of $G$-$\mu$-displays over $(A, I)$.

We assume for a moment that $(A, I)$ is orientable.
The results in Section \ref{Subsection:Display groups} apply to $(A, I)$.
In particular, we have the display group $G_\mu(A, I)$.
For each generator $d \in I$, we define the following homomorphism:
\[
    \sigma_{\mu, d} \colon G_\mu(A, I) \to G(A), \quad g \mapsto \phi(\mu(d)g\mu(d)^{-1}).
\]
Let $G(A)_d$ be the set
$G(A)$ together with the following action of $G_\mu(A, I)$:
\[
    G(A) \times G_\mu(A, I) \to G(A), \quad (X, g) \mapsto g^{-1}X\sigma_{\mu, d}(g).
\]
For another generator $d' \in I$, we have $d=ud'$ for a unique $u \in A^\times$.
The map $G(A)_d \to G(A)_{d'}$ defined by $X \mapsto X\phi(\mu(u))$ is $G_\mu(A, I)$-equivariant.
Thus we can define
\[
G(A)_I := {\varprojlim}_{d} G(A)_d
\]
where $d$ runs over the set of generators $d \in I$.
The set $G(A)_I$ carries a natural action of $G_\mu(A, I)$.
The projection map $G(A)_I \to G(A)_d$ is a 
$G_\mu(A, I)$-equivariant bijection.
For an element $X \in G(A)_I$, let
\begin{equation}\label{equation:d-component}
    X_d \in G(A)_d
\end{equation}
denote the image of $X$.

Let
$G_{\mu, A, I}$
and
$G_{\Prism, A, I}$
be the sheaves on $(A, I)^{\op}_\et$ defined by
\[
G_{\mu, A, I}(B):=G_\mu(B, IB) \quad \text{and} \quad
G_{\Prism, A, I}(B):=G(B)_{IB},
\]
respectively.
The sheaf $G_{\Prism, A, I}$ is equipped with a natural action of the group sheaf $G_{\mu, A, I}$.
In fact, we can extend these definitions to (not necessarily orientable) bounded $\O_E$-prisms $(A, I)$ over $\O$; see \cite[Section 4.3]{Ito-K23} for details.

\begin{defn}[{$G$-$\mu$-display}]\label{Definition:G mu display over oriented prisms}
Let
$(A, I)$
be a bounded $\O_E$-prism over $\O$.
A \textit{$G$-$\mu$-display} over
    $(A, I)$ is a pair
    \[
    (\mathcal{Q}, \alpha_\mathcal{Q})
    \]
    where $\mathcal{Q}$ is a $G_{\mu, A, I}$-torsor and
    $
    \alpha_\mathcal{Q} \colon \mathcal{Q} \to G_{\Prism, A, I}
    $
    is a $G_{\mu, A, I}$-equivariant map of sheaves on $(A, I)^{\op}_\et$.
    When there is no possibility of confusion, we write $\mathcal{Q}$ instead of $(\mathcal{Q}, \alpha_\mathcal{Q})$.
    We say that $(\mathcal{Q}, \alpha_\mathcal{Q})$ is \textit{banal} if $\mathcal{Q}$ is a trivial $G_{\mu, A, I}$-torsor.
\end{defn}

Isomorphisms between $G$-$\mu$-displays over $(A, I)$ are defined in the obvious way.
We write
\[
G\mathchar`-\mathrm{Disp}_\mu(A, I) \quad \text{and} \quad G\mathchar`-\mathrm{Disp}_\mu(A, I)_{\mathrm{banal}}
\]
for the groupoid of $G$-$\mu$-displays over $(A, I)$ and the groupoid of banal $G$-$\mu$-displays over $(A, I)$, respectively.

\begin{rem}\label{Remark:quotient groupoid banal G-displays}
Assume that $(A, I)$ is orientable.
Let
$
[G(A)_I/G_\mu(A, I)]
$
be the groupoid whose objects are the elements $X \in G(A)_I$ and whose morphisms are defined by
\[
\Hom(X, X')=\{\, g \in G_\mu(A, I) \, \vert \, X'\cdot g=X  \, \}.
\]
Here $(-)\cdot g$ denotes the action of $g \in G_\mu(A, I)$.
To each
$X \in G(A)_I$,
we attach a banal $G$-$\mu$-display
\[
\mathcal{Q}_X:=(G_{\mu, A, I}, \alpha_X)
\]
over 
$(A, I)$
where
$\alpha_X \colon G_{\mu, A, I} \to G_{\Prism, A, I}$ is given by $1 \mapsto X$.
This construction gives an equivalence
$[G(A)_I/G_\mu(A, I)]
\overset{\sim}{\to}
G\mathchar`-\mathrm{Disp}_\mu(A, I)_{\mathrm{banal}}$
of groupoids.
\end{rem}

For a map
$f \colon (A, I) \to (A', I')$
of bounded $\O_E$-prisms over $\O$,
we have a base change functor
\[
f^* \colon G\mathchar`-\mathrm{Disp}_\mu(A, I) \to G\mathchar`-\mathrm{Disp}_\mu(A', I').
\]
More precisely, we have 
a fibered category
\[
(A, I) \mapsto G\mathchar`-\mathrm{Disp}_\mu(A, I)
\]
over $(\O)^{\op}_{\Prism, \O_E}$.

\begin{rem}\label{Remark:base change of G mu displays}
Assume that $(A, I)$ is orientable.
For an element $X \in G(A)_I$, let
$f(X) \in G(A')_{I'}$
be the unique element such that, for any generator $d \in I$, we have
$f(X)_{f(d)}=f(X_d)$ in $G(A')$.
Then we have $f^*(\mathcal{Q}_X)=\mathcal{Q}_{f(X)}$ for every $X \in G(A)_I$.
\end{rem}

In the following, we recall more concrete descriptions of $G$-$\mu$-displays.
Let $(A, I)$ be a bounded $\O_E$-prism over $\O$.

\begin{defn}\label{Definition:G-BK module of type mu}
A \textit{$G$-Breuil--Kisin module} over $(A, I)$ is a $G$-torsor $\mathcal{P}$ over $\Spec A$
(with respect to the \'etale topology)
with 
an isomorphism
$
F_\mathcal{P} \colon (\phi^*\mathcal{P})[1/I] \overset{\sim}{\to} \mathcal{P}[1/I]
$
of $G$-torsors over $\Spec A[1/I]$.
We call $F_\mathcal{P}$ the Frobenius of $\mathcal{P}$.
We say that $\mathcal{P}$ is \textit{of type $\mu$} if $(\pi, I)$-completely \'etale locally on $A$,
there exists some trivialization
$\mathcal{P} \simeq G_A$ under which the isomorphism $F_\mathcal{P}$ is given by $g \mapsto Yg$ for an element $Y$ in the double coset
\[
G(A)\mu(d)G(A) \subset G(A[1/I]),
\]
where $d \in I$ is a generator.
\end{defn}

\begin{prop}\label{Proposition:G-displays to G-BK}
    There is an equivalence
    $\mathcal{Q} \mapsto \mathcal{Q}_{\mathrm{BK}}$
    from the groupoid
$G\mathchar`-\mathrm{Disp}_\mu(A, I)$ to the groupoid of $G$-Breuil--Kisin modules of type $\mu$ over $(A, I)$.
This equivalence is compatible with base change along any map $(A, I) \to (A', I')$.
\end{prop}

\begin{proof}
    See \cite[Proposition 5.3.8]{Ito-K23}.
    For the construction of the functor $\mathcal{Q} \mapsto \mathcal{Q}_{\mathrm{BK}}$, see \cite[Definition 5.3.6]{Ito-K23}.
\end{proof}

\begin{ex}\label{Example:G-displays to G-BK}
    Assume that $(A, I)$ is orientable.
    Let $d \in I$ be a generator.
    For an element $X \in G(A)_I$,
    the trivial $G$-torsor $G_A$ with the isomorphism
    \[
    (\phi^*G_A)[1/I]=G_A[1/I] \overset{\sim}{\to} G_A[1/I], \quad g \mapsto (\mu(d)X_d)g
    \]
    is a $G$-Breuil--Kisin module of type $\mu$ over $(A, I)$, and this is isomorphic to the one $(\mathcal{Q}_X)_{\mathrm{BK}}$
    associated with $\mathcal{Q}_X \in G\mathchar`-\mathrm{Disp}_\mu(A, I)$.
\end{ex}


\begin{ex}[{Minuscule Breuil--Kisin module}]\label{Example:GLn displays}
We assume that $G=\GL_N$.
Let
$\mu \colon \G_m \to \GL_{N}$
be the cocharacter
defined by
\[
t \mapsto \diag{(\underbrace{t, \dotsc, t}_s, \underbrace{1, \dotsc, 1}_{N-s})},
\]
which is 1-bounded.
A \textit{minuscule Breuil--Kisin module} over $(A, I)$ is a finite projective $A$-module $M$ together with an $A$-linear homomorphism
\[
F_M \colon \phi^*M=A \otimes_{\phi, A} M \to M
\]
whose cokernel is killed by $I$.
We set
$
\Fil^1(\phi^*M):= \{ \, x \in \phi^*M \, \vert \, F_M(x) \in IM  \, \}
$
and let
$
P^1 \subset (\phi^*M)/I(\phi^*M)
$
be the image of $\Fil^1(\phi^*M)$.
By \cite[Proposition 3.1.6]{Ito-K23},
we see that $P^1$ is a direct summand of $(\phi^*M)/I(\phi^*M)$.
We say that $M$ is \textit{of type} $\mu$ if the rank of 
$M$
(resp.\ $P^1$) is constant and equal to 
$N$
(resp.\ $s$).
Let
$
\mathrm{BK}_\mu(A, I)^{\simeq}
$
be the groupoid of minuscule Breuil--Kisin modules over $(A, I)$ of type $\mu$.
In \cite[Example 5.3.3, Corollary 5.3.11]{Ito-K23},
we constructed a natural equivalence of groupoids:
\[
\mathrm{BK}_\mu(A, I)^{\simeq} \overset{\sim}{\to} \GL_N\mathchar`-\mathrm{Disp}_\mu(A, I), \quad M \mapsto \mathcal{Q}(M).
\]
We set
$\Fil^1_\mu:= A^s \oplus I^{N-s} \subset A^N$.
Then the underlying $G_{\mu, A, I}$-torsor of $\mathcal{Q}(M)$
is the functor
\[
\underline{\mathrm{Isom}}_{\Fil}(A^N, \phi^*M)
\colon (A, I)_\et \to \mathrm{Set}
\]
defined by sending $B \in (A, I)_\et$ to the set of isomorphisms
$B^N \overset{\sim}{\to} (\phi^*M) \otimes_A B$ under which $\Fil^1_\mu$ agrees with $\Fil^1(\phi^*M)$.
\end{ex}

\subsection{Hodge filtrations and $G$-$\phi$-modules}\label{Subsection:Hodge filtrations and G-torsors}

Let
$(A, I)$
be an orientable and bounded $\O_E$-prism over $\O$.
Here we recall the definitions of the Hodge filtration
$P(\mathcal{Q})_{A/I}$
and
the underlying $G$-$\phi$-module
$\mathcal{Q}_\phi$
of a $G$-$\mu$-display $\mathcal{Q}$ over $(A, I)$ introduced in \cite{Ito-K23}.

Let
$G_{\Prism, A}$,
$G_{\overline{\Prism}, A}$,
and $(P_{\mu})_{\overline{\Prism}, A}$
be the sheaves on $(A, I)^{\op}_\et$ defined by
\[
G_{\Prism, A}(B):=G(B),\quad 
G_{\overline{\Prism}, A}(B):=G(B/IB), \quad \text{and} \quad
(P_{\mu})_{\overline{\Prism}, A}(B):=P_\mu(B/IB),
\]
respectively.
Let
$\tau \colon G_{\mu, A, I} \hookrightarrow G_{\Prism, A}$
be the inclusion.
The composition of $\tau$ with the projection $G_{\Prism, A} \to G_{\overline{\Prism}, A}$
is denoted by $\overline{\tau}$, which factors through
a homomorphism
$\overline{\tau}_P \colon G_{\mu, A, I} \to (P_{\mu})_{\overline{\Prism}, A}$ by Proposition \ref{Proposition:BB isomorphism}.
We have a commutative diagram
\[
\xymatrix{
G_{\mu, A, I} \ar^-{\tau}[r]  \ar[d]_-{\overline{\tau}_P} \ar^-{\overline{\tau}}[rd] &  G_{\Prism, A}   \ar[d]_-{} \\
(P_{\mu})_{\overline{\Prism}, A} \ar[r]^-{} & G_{\overline{\Prism}, A}.
}
\]

\begin{rem}\label{Remark:schematic torsors}
    For a $G$-torsor $\mathcal{P}$ over $\Spec A$ (with respect to the \'etale topology),
the sheaf on
$(A, I)^{\op}_\et$ defined by $B \mapsto \mathcal{P}(B)$
is a $G_{\Prism, A}$-torsor.
By \cite[Proposition 4.3.1]{Ito-K23},
this construction
induces an equivalence of categories
\[
\{ \, G\text{-torsors over} \ \Spec A \, \} \overset{\sim}{\to} \{ \, G_{\Prism, A}\text{-torsors} \, \}.
\]
We will make no distinction between a $G$-torsor over $\Spec A$ and the corresponding $G_{\Prism, A}$-torsor.
Similarly, the category of
$G$-torsors over $\Spec A/I$
(resp.\ $P_\mu$-torsors over $\Spec A/I$)
is equivalent to that of
$G_{\overline{\Prism}, A}$-torsors
(resp.\ $(P_{\mu})_{\overline{\Prism}, A}$-torsors).
\end{rem}

Let
$(\mathcal{Q}, \alpha_{\mathcal{Q}}) \in G\mathchar`-\mathrm{Disp}_\mu(A, I)$.
We write
\[
\mathcal{Q}_{A}:=\mathcal{Q}^{\tau} \quad
(\text{resp.}\ 
\mathcal{Q}_{A/I}:=\mathcal{Q}^{\overline{\tau}},\, \text{resp.}\ P(\mathcal{Q})_{A/I}:=\mathcal{Q}^{\overline{\tau}_P})
\]
for the pushout of $\mathcal{Q}$ along $\tau$
(resp.\ $\overline{\tau}$, resp.\ $\overline{\tau}_P$).

\begin{defn}\label{Definition:Hodge filtration of G-displays}
We regard $P(\mathcal{Q})_{A/I}$ as a $P_\mu$-torsor over $\Spec A/I$.
We call
$P(\mathcal{Q})_{A/I}$
(or the morphism $P(\mathcal{Q})_{A/I} \to \mathcal{Q}_{A/I}$)
the \textit{Hodge filtration} of $\mathcal{Q}_{A/I}$.
We also say that $P(\mathcal{Q})_{A/I}$ is the Hodge filtration of $\mathcal{Q}$.
\end{defn}

We recall the following useful fact:

\begin{prop}\label{Proposition:G display with trivial Hodge filtration is banal}
A $G$-$\mu$-display
$\mathcal{Q}$
over $(A, I)$ is banal if and only if the Hodge filtration
$P(\mathcal{Q})_{A/I}$
is a trivial $P_{\mu}$-torsor over $\Spec A/I$.
\end{prop}

\begin{proof}
See \cite[Proposition 5.4.5]{Ito-K23}.
\end{proof}

A \textit{$G$-$\phi$-module} over $(A, I)$
is a $G$-torsor $\mathcal{P}$ over $\Spec A$
with an isomorphism
\[
\phi_\mathcal{P} \colon (\phi^*\mathcal{P}) \times_{\Spec A} \Spec A[1/\phi(I)] \overset{\sim}{\to} \mathcal{P} \times_{\Spec A} \Spec A[1/\phi(I)]
\]
of $G$-torsors over $\Spec A[1/\phi(I)]$.
Here we set $A[1/\phi(I)]:=A[1/\phi(d)]$ for a generator $d \in I$.
We call $\phi_\mathcal{P}$ the Frobenius of $\mathcal{P}$.

In \cite[Section 5.5]{Ito-K23},
we constructed
a functor
\[
    G\mathchar`-\mathrm{Disp}_\mu(A, I)
    \to \{ \, G\mathchar`-\phi\mathchar`-\text{modules over} \ (A, I) \, \}, \quad \mathcal{Q} \mapsto \mathcal{Q}_\phi:=(\mathcal{Q}_A, \phi_{\mathcal{Q}_A}).
\]
If $\mathcal{Q}=\mathcal{Q}_X$ is the banal $G$-$\mu$-display over $(A, I)$ associated with an element $X \in G(A)_I$ (see Remark \ref{Remark:quotient groupoid banal G-displays}), then the isomorphism $\phi_{\mathcal{Q}_A}$ is defined as follows.
We set
\begin{equation}\label{equation:Xphi}
    X_\phi := X_d\phi(\mu(d)) \in G(A[1/\phi(I)]),
\end{equation}
which is independent of the choice of a generator $d \in I$.
(See (\ref{equation:d-component}) for the notation $X_d$.)
We have $\mathcal{Q}_A=G_A$, and hence
$\phi^*(\mathcal{Q}_A)=G_A$.
The isomorphism $\phi_{\mathcal{Q}_A}$ is then given by
$
g \mapsto X_\phi g.
$

\begin{defn}\label{Definition:underlying phi G torsor}
The $G$-$\phi$-module 
$
\mathcal{Q}_{\phi}=(\mathcal{Q}_A, \phi_{\mathcal{Q}_A})
$
over $(A, I)$ is called the \textit{underlying $G$-$\phi$-module} of $\mathcal{Q}$.
\end{defn}

\begin{ex}\label{Example:phi GLn torsor and Hodge filtration}
We retain the notation of Example \ref{Example:GLn displays}.
Let
$\mathcal{Q}:=\mathcal{Q}(M)$
be
the $\GL_N$-$\mu$-display over $(A, I)$
associated with
a
minuscule Breuil--Kisin module $M$ over $(A, I)$ of type $\mu$.
The underlying 
$\GL_N$-$\phi$-module
$\mathcal{Q}_{\phi}$
is the one naturally associated with 
the pair
\[
(\phi^*M, \phi^*(F_M)).
\]
For the Hodge filtration $P(\mathcal{Q})_{A/I}$, we have
\[
P(\mathcal{Q})_{A/I} \simeq \underline{\mathrm{Isom}}_{\Fil}((A/I)^N, (\phi^*M)/I(\phi^*M))
\]
where the right hand side is the scheme of isomorphisms
$f \colon (A/I)^N \overset{\sim}{\to} (\phi^*M)/I(\phi^*M)$
such that $f((A/I)^s \oplus 0)=P^1$.
\end{ex}

\subsection{Prismatic $G$-$\mu$-displays over complete regular local rings}\label{Subsection:G-displays over complete regular local rings}

In this subsection, we recall the main result of \cite{Ito-K23}.

\begin{defn}\label{Definition:G displays over prismatic sites}
Let $R$ be a $\pi$-adically complete $\O$-algebra.
We define the following groupoid:
\[
G\mathchar`-\mathrm{Disp}_\mu((R)_{\Prism, \O_E}):= {2-\varprojlim}_{(A, I) \in (R)_{\Prism, \O_E}} G\mathchar`-\mathrm{Disp}_\mu(A, I).
\]
A \textit{prismatic $G$-$\mu$-display} over $R$ is an object of $G\mathchar`-\mathrm{Disp}_\mu((R)_{\Prism, \O_E})$.
\end{defn}

For a homomorphism
$h \colon R \to R'$ of $\pi$-adically complete $\O$-algebras, we have a base change functor
\[
h^* \colon G\mathchar`-\mathrm{Disp}_\mu((R)_{\Prism, \O_E}) \to G\mathchar`-\mathrm{Disp}_\mu((R')_{\Prism, \O_E}).
\]

\begin{rem}\label{Remark:alternative definition of G displays over prismatic sites}
Giving a prismatic $G$-$\mu$-display $\mathfrak{Q}$ over $R$
is equivalent to giving a section
of the fibered category
$
(A, I) \mapsto G\mathchar`-\mathrm{Disp}_\mu(A, I)
$
over $(R)^{\op}_{\Prism, \O_E}$.
We thus have the associated $G$-$\mu$-display $\mathfrak{Q}_{(A, I)}$ over $(A, I)$ for each $(A, I) \in (R)_{\Prism, \O_E}$
and a natural isomorphism
\[
\gamma_f \colon f^*(\mathfrak{Q}_{(A, I)}) \overset{\sim}{\to} \mathfrak{Q}_{(A', I')}
\]
for each morphism $f \colon (A, I) \to (A', I')$ in $(R)_{\Prism, \O_E}$.
We call $\mathfrak{Q}_{(A, I)}$ the \textit{value} of $\mathfrak{Q}$ at $(A, I) \in (R)_{\Prism, \O_E}$.
Recall that an object of $(R)_{\Prism, \O_E}$
is given by a bounded $\O_E$-prism $(A, I)$ over $\O$ together with a homomorphism $g \colon R \to A/I$ over $\O$.
We also write
$
\mathfrak{Q}_{g}:=\mathfrak{Q}_{(A, I)}
$
in order to emphasize that it depends on the homomorphism $g$.
For a homomorphism
$h \colon R \to R'$ and
an object
$((A, I), g \colon R' \to A/I) \in (R')_{\Prism, \O_E}$,
we have
\[
(h^*\mathfrak{Q})_g=\mathfrak{Q}_{g \circ h}.
\]
\end{rem}

\begin{ex}\label{Example:prismatic display over k}
    The functor
    \[
    G\mathchar`-\mathrm{Disp}_{\mu}((k)_{\Prism, \O_E}) \overset{\sim}{\to} G\mathchar`-\mathrm{Disp}_\mu(\O, (\pi)),  \quad \mathfrak{Q} \mapsto \mathfrak{Q}_{(\O, (\pi))}
    \]
    is an equivalence
    since $(\O, (\pi)) \in (k)_{\Prism, \O_E}$ is an initial object (\cite[Example 2.5.8]{Ito-K23}).
\end{ex}

Let $\mathcal{C}_\O$ be the category of complete regular local rings $R$ with a local homomorphism $\O \to R$ which induces an isomorphism on the residue fields.
The morphisms in $\mathcal{C}_\O$ are the local homomorphisms over $\O$.

An important feature of the category $\mathcal{C}_\O$ is the following:

\begin{rem}\label{Remark:regular local ring and Breuil-Kisin module}
    Let $(\mathfrak{S}_\O, (\mathcal{E}))$ be an $\O_E$-prism of Breuil--Kisin type over $\O$.
    Then the quotient $\mathfrak{S}_\O/\mathcal{E}$ belongs to $\mathcal{C}_\O$.
    Conversely, any $R \in \mathcal{C}_\O$ is of the form
    $
    R \simeq \mathfrak{S}_\O/\mathcal{E}
    $
    for some $(\mathfrak{S}_\O, (\mathcal{E}))$
    of Breuil--Kisin type.
    (See for example \cite[Section 3.3]{ChengChuangxun}.)
\end{rem}

Let $R \in \mathcal{C}_\O$.
We choose
an $\O_E$-prism
$(\mathfrak{S}_\O, (\mathcal{E}))$
of Breuil--Kisin type
with an isomorphism
$
R \simeq \mathfrak{S}_\O/\mathcal{E}
$
over $\O$.
The following result is a key ingredient in our deformation theory:

\begin{thm}\label{Theorem:main result on G displays over complete regular local rings}
The functor
    \[
 G\mathchar`-\mathrm{Disp}_\mu((R)_{\Prism, \O_E}) \to G\mathchar`-\mathrm{Disp}_\mu(\mathfrak{S}_\O, (\mathcal{E})), \quad \mathfrak{Q} \mapsto \mathfrak{Q}_{(\mathfrak{S}_\O, (\mathcal{E}))}
 \]
is an equivalence if the cocharacter $\mu$ is 1-bounded.
\end{thm}

\begin{proof}
    See \cite[Theorem 6.1.3]{Ito-K23}.
\end{proof}

\section{The Grothendieck--Messing deformation theory for $G$-$\mu$-displays}\label{Section:The Grothendieck--Messing deformation theory for G displays}

In this section,
we establish an analogue of the Grothendieck--Messing deformation theory for $G$-$\mu$-displays.
The contents of this section can be summarized as follows.
In Section \ref{Subsection:Deformations}, we collect some basic results on deformations of $G$-$\mu$-displays.
In Section \ref{Subsection:special nilpotent thickening}, we introduce a class of maps
$(A', I') \to (A, I)$
of orientable and bounded $\O_E$-prisms over $\O$, called \textit{special nilpotent thickenings}, for which we formulate the analogue of the Grothendieck--Messing deformation theory.
This class includes the maps
$(\mathfrak{S}_{\O, m+1}, (\mathcal{E})) \to (\mathfrak{S}_{\O, m}, (\mathcal{E}))$
and
$(W_{\O_E}(S^\flat)/[a^\flat]^{m+1}, I_S) \to (W_{\O_E}(S^\flat)/[a^\flat]^m, I_S)$.
For a $G$-$\mu$-display $\mathcal{Q}$
over $(A, I)$ and its deformation $\mathscr{Q}$ over $(A', I')$,
we construct a \textit{period map}
$\Per_{\mathscr{Q}}$
from
the set
$\Def(\mathcal{Q})_{(A', I')}$
of isomorphism classes of deformations of $\mathcal{Q}$ over $(A', I')$
to the set of isomorphism classes of lifts of the Hodge filtration $P(\mathcal{Q})_{A/I}$ in $\mathscr{Q}_{A'/I'}$, and prove that $\Per_{\mathscr{Q}}$ is bijective when $\mu$ is 1-bounded; see Section \ref{Subsection:Period map} for details.
In Section \ref{Subsection:Normalized period map}, under a mild additional assumption on $(A', I') \to (A, I)$,
we show that
the set
$\Def(\mathcal{Q})_{(A', I')}$
has a natural structure of a torsor under an $A/I$-module when $\mu$ is 1-bounded.
This structure plays an important role in Section \ref{Section:Universal deformations}.

\subsection{Deformations}\label{Subsection:Deformations}

Let $f \colon (A', I') \to (A, I)$ be a map of orientable and bounded $\O_E$-prisms over $\O$
such that $f \colon A' \to A$ is surjective.

We can define a deformation of a $G$-$\mu$-display in the usual way:

\begin{defn}\label{Definition:deformation}
Let
$\mathcal{Q}$
be a $G$-$\mu$-display over $(A, I)$.
\begin{enumerate}
    \item A \textit{deformation}
of
$\mathcal{Q}$
over $(A', I')$
is a pair
$
(\mathscr{Q}, h)
$
consisting of a $G$-$\mu$-display
$\mathscr{Q}$
over $(A', I')$ and an isomorphism
$
h \colon f^*\mathscr{Q} \overset{\sim}{\to} \mathcal{Q}
$
of $G$-$\mu$-displays over $(A, I)$.
If there is no ambiguity, we simply write $\mathscr{Q}$ instead of $(\mathscr{Q}, h)$.
    \item An isomorphism
    $
    g \colon (\mathscr{Q}, h) \overset{\sim}{\to} (\mathscr{R}, i)
    $
    of deformations of $\mathcal{Q}$ over $(A', I')$ is an isomorphism
    $g \colon \mathscr{Q} \overset{\sim}{\to} \mathscr{R}$
    of $G$-$\mu$-display over $(A', I')$ such that
    $i \circ (f^*g)=h$.
\end{enumerate}
The set of isomorphism classes of deformations of $\mathcal{Q}$
over $(A', I')$ is denoted by
\[
\Def(\mathcal{Q})_{(A', I')}.
\]
The groupoid of deformations of $\mathcal{Q}$
over $(A', I')$ is denoted by
$
\mathbf{Def}(\mathcal{Q})_{(A', I')},
$
using boldface letters.
\end{defn}

We collect some preliminary results on deformations which will be used in the sequel.
Let $J \subset A'$ (resp.\ $K \subset A'/I'$) be the kernel of
$A' \to A$ (resp.\ $A'/I' \to A/I$).

We first note the following fact:

\begin{lem}\label{Lemma:deformation of banal is banal}
Assume that $A'/I'$ is $K$-adically complete.
Let $\mathcal{Q}$ be a banal $G$-$\mu$-display over $(A, I)$.
Then any deformation $\mathscr{Q}$ of $\mathcal{Q}$ over $(A', I')$ is banal. 
\end{lem}

\begin{proof}
By Proposition \ref{Proposition:G display with trivial Hodge filtration is banal}, it suffices 
to prove that
the Hodge filtration
$P(\mathscr{Q})_{A'/I'}$
is trivial as a $P_\mu$-torsor over $\Spec A'/I'$.
The base change $P(\mathscr{Q})_{A'/I'} \times_{\Spec A'/I'} \Spec A/I$ is trivial
since 
$\mathcal{Q}$ is banal.
Since $P(\mathscr{Q})_{A'/I'}$ is smooth over $\Spec A'/I'$ and 
$A'/I'$ is $K$-adically complete, we see that $P(\mathscr{Q})_{A'/I'}$ is trivial, as desired.
\end{proof}

The next lemma is an analogue of \cite[Lemma 7.1.4 (b)]{Lau21}.

\begin{lem}\label{Lemma:Gmu surjective nilpotent thickening}
If $A'$ is $J$-adically complete,
then the homomorphism
$f \colon G_\mu(A', I') \to G_\mu(A, I)$
is surjective.
\end{lem}

\begin{proof}
Since $A'$ is $J$-adically complete and $P_\mu$ is smooth,
the homomorphism $P_\mu(A') \to P_\mu(A)$ is surjective.
It is easy to show that
$U^{-}_{\mu}(A') \cap G_\mu(A', I') \to U^{-}_{\mu}(A) \cap G_\mu(A, I)$
is surjective; see \cite[Remark 4.2.7]{Ito-K23}.
The claim then follows from Proposition \ref{Proposition:BB isomorphism}.
\end{proof}

We have the following useful description of deformations for banal $G$-$\mu$-displays.

\begin{prop}\label{Proposition:deformation of banal G displays}
Assume that $A'$ is $J$-adically complete and $A'/I'$ is $K$-adically complete.
Let $X \in G(A)_I$ be an element and
$\mathcal{Q}_X$ the banal $G$-$\mu$-display over $(A, I)$
associated with $X$.
Let
\[
\mathbf{Def}(X)_{(A', I')}
\]
be the groupoid whose objects are the elements $Y \in G(A')_{I'}$ such that $f(Y)=X$, and whose morphisms are given by
    \[
    \Hom_{\mathbf{Def}(X)_{(A', I')}}(Y, Z)=\{\, g \in G_\mu(A', I') \, \vert \, Z \cdot g =Y \, \, \mathrm{and} \, \, f(g)=1  \, \}.
    \]
For each $Y \in \mathbf{Def}(X)_{(A', I')}$,
we have a natural isomorphism
$f^*(\mathcal{Q}_Y) \overset{\sim}{\to} \mathcal{Q}_X$ given by $1 \in G_{\mu}(A, I)$.
The construction
$Y \mapsto (\mathcal{Q}_Y, 1)$
gives an equivalence
\[
\mathbf{Def}(X)_{(A', I')} \overset{\sim}{\to} \mathbf{Def}(\mathcal{Q}_X)_{(A', I')}.
\]
\end{prop}

\begin{proof}
For the map $f \colon G(A')_{I'} \to G(A)_I$, see Remark \ref{Remark:base change of G mu displays}.
It is clear that $Y \mapsto (\mathcal{Q}_Y, 1)$ gives a fully faithful functor.
We shall prove that this functor is essentially surjective.
Let
$
(\mathscr{Q}, h) \in \mathbf{Def}(\mathcal{Q}_X)_{(A', I')}.
$
It follows from Lemma \ref{Lemma:deformation of banal is banal} that $\mathscr{Q}$ is banal.
We can identify $\mathscr{Q}$ with $\mathcal{Q}_Y$ for some $Y \in G(A')_{I'}$, and then $h$ is an element of $G_{\mu}(A, I)$ satisfying
the equality
$
X\cdot h=f(Y).
$
By Lemma \ref{Lemma:Gmu surjective nilpotent thickening}, there exists an element $h' \in G_\mu(A', I')$ with $f(h')=h$.
We define
$Y':=Y \cdot (h')^{-1} \in G(A')_{I'}$.
Then
$Y' \in \mathbf{Def}(X)_{(A', I')}$
and
$h'$ induces
$
(\mathcal{Q}_Y, h) \overset{\sim}{\to} (\mathcal{Q}_{Y'}, 1).
$
\end{proof}

For later use, we also record the following result.

\begin{prop}\label{Proposition:limit of G displays}
Let $(A, I)$ be an orientable and bounded $\O_E$-prism over $\O$.
Let $J \subset A$ be an ideal satisfying the following conditions:
\begin{enumerate}
    \item[(a)] $A$ is $J$-adically complete.
    \item[(b)] $J$ is stable under $\delta_E$. In particular $A/J^m$ admits a $\delta_E$-structure that is compatible with the one on $A$ for every $m \geq 1$.
    \item[(c)] For every $m \geq 1$, the pair
    $(A/J^m, I)$ is an orientable and bounded $\O_E$-prism over $\O$.
    Here we abuse notation and denote the image of $I$ in $A/J^m$ by the same symbol.
\end{enumerate}
Then the following assertions hold:
\begin{enumerate}
    \item We have the following equivalence of groupoids
\[
G\mathchar`-\mathrm{Disp}_\mu(A, I)_{\mathrm{banal}} \overset{\sim}{\to} {2-\varprojlim}_{m} G\mathchar`-\mathrm{Disp}_\mu(A/J^m, I)_{\mathrm{banal}}.
\]
    \item We further assume that, for any $G$-$\mu$-display $\mathcal{Q}$ over $(A, I)$ 
    (resp.\ $(A/J, I)$), there is a finite \'etale covering
    $A \to B$ such that the base change of $\mathcal{Q}$ to $(B, IB)$ 
    (resp.\ $(B/JB, I(B/JB))$)
    is banal.
    Then we have
\[
G\mathchar`-\mathrm{Disp}_\mu(A, I) \overset{\sim}{\to} {2-\varprojlim}_{m} G\mathchar`-\mathrm{Disp}_\mu(A/J^m, I).
\]
\end{enumerate}
\end{prop}

\begin{proof}
(1)
We claim that
\begin{equation}\label{equation:inverse limit display group}
    G(A)_I \overset{\sim}{\to} {\varprojlim}_{m} G(A/J^m)_I \quad \text{and} \quad G_\mu(A, I)\overset{\sim}{\to}{\varprojlim}_{m} G_\mu(A/J^m, I).
\end{equation}
Indeed, the first equality is clear since $G$ is affine.
The second equality follows from Proposition \ref{Proposition:BB isomorphism} (2).

In order to prove (1), it suffices to show that
the natural functor
\[
[G(A)_I/G_\mu(A, I)] \to \mathcal{C}:={2-\varprojlim}_{m} [G(A/J^m)_I/G_\mu(A/J^m, I)]
\]
is an equivalence (cf.\ Remark \ref{Remark:quotient groupoid banal G-displays}).
The fully faithfulness follows from (\ref{equation:inverse limit display group}).
We shall prove that the functor is essentially surjective.
An object of
$\mathcal{C}$
can be identified with a family of objects
$\{ X_m \}_{m \geq 1}$, where $X_m \in G(A/J^m)_I$, together with 
a family of isomorphisms
$\{ g_m \}_{m \geq 1}$, where $g_m \in G_\mu(A/J^m, I)$ is such that
the equality
$X_m\cdot g_m=X_{m+1}$
holds
in $G(A/J^m)_I$.
Applying Proposition \ref{Proposition:deformation of banal G displays} to
the natural map
$p_m \colon (A/J^{m+1}, I) \to (A/J^{m}, I)$ repeatedly, we can find an element
$(X'_m)_{m \geq 1} \in \varprojlim_{m} G(A/J^m)_I$
such that 
the corresponding object of $\mathcal{C}$
is isomorphic to the object $\{ X_m \}_{m \geq 1}$.
Let
$X \in G(A)_I$
be
the element corresponding to $(X'_m)_{m \geq 1}$.
Then the image of $X$ under the above functor is isomorphic to the object $\{ X_m \}_{m \geq 1}$.

(2) An object of ${2-\varprojlim}_{m} G\mathchar`-\mathrm{Disp}_\mu(A/J^m, I)$ can be identified with
a family
$\{ \mathcal{Q}_m \}_{m \geq 1}$, where $\mathcal{Q}_m$ is a $G$-$\mu$-display over $(A/J^m, I)$,
together with isomorphisms
$p^*_m(\mathcal{Q}_{m+1}) \overset{\sim}{\to} \mathcal{Q}_{m}$ ($m \geq 1$).
Our assumption and Lemma \ref{Lemma:deformation of banal is banal} imply that
there is a finite \'etale covering
$A \to B$
such that for any $m \geq 1$,
the base change of $\mathcal{Q}_m$ to $(B/J^mB, I(B/J^mB))$ is banal.
We note that for a finite \'etale covering $A \to B$, the ideal $JB$ satisfies the above conditions (a), (b), (c).
Using these observations and finite \'etale descent for $G$-$\mu$-displays,
we see that (2) follows from (1).
\end{proof}

\begin{ex}\label{Example:banal over finite etale covering}
Let $(\mathfrak{S}_\O, (\mathcal{E}))$ be an $\O_E$-prism of Breuil--Kisin type as in Example \ref{Example:Breuil-Kisin type frame}, where $\mathfrak{S}_\O=\O[[t_1, \dotsc, t_n]]$.
Let $\mathcal{Q}$ be a $G$-$\mu$-display over $(\mathfrak{S}_\O, (\mathcal{E}))$.
By Proposition \ref{Proposition:G display with trivial Hodge filtration is banal}, there is a finite extension $\widetilde{k}$ of $k$ such that the base change of
$\mathcal{Q}$ to
$(\mathfrak{S}_{\widetilde{\O}}, (\mathcal{E}))$ is banal,
where $\widetilde{\O}:= W(\widetilde{k}) \otimes_{W(\F_q)} \O_E$ and $\mathfrak{S}_{\widetilde{\O}}:=\mathfrak{S}_\O \otimes_\O \widetilde{\O} = \widetilde{\O}[[t_1, \dotsc, t_n]]$.
Therefore, by Proposition \ref{Proposition:limit of G displays}, we have
\[
G\mathchar`-\mathrm{Disp}_\mu(\mathfrak{S}_\O, (\mathcal{E})) \overset{\sim}{\to} {2-\varprojlim}_{m} G\mathchar`-\mathrm{Disp}_\mu(\mathfrak{S}_{\O, m}, (\mathcal{E})),
\]
where $\mathfrak{S}_{\O, m}=\O[[t_1, \dotsc, t_n]]/(t_1, \dotsc, t_n)^m$.
\end{ex}

\subsection{Special nilpotent thickenings}\label{Subsection:special nilpotent thickening}

Let $f \colon (A', I') \to (A, I)$ be a map of orientable and bounded $\O_E$-prisms over $\O$
such that $f \colon A' \to A$ is surjective.
Let $J \subset A'$ be the kernel of $f$.
We assume that $J^n=0$ for some $n \geq 1$.
In this case, we say that $f$ is a \textit{nilpotent thickening}.
The functor
\[
(A', I')_\et \to (A, I)_\et, \quad B' \mapsto B
\]
where $B$ is the $(\pi, I)$-adic completion of $B' \otimes_{A'} A$, is an equivalence (see \cite[Lemma 2.5.9]{Ito-K23} for example).

\begin{lem}\label{Lemma:thickening etale morphism}
Let the notation be as above.
Then the induced homomorphism
$f_B \colon B' \to B$
is surjective, and we have $J^n_B=0$ for the kernel $J_B$ of $f_B$.  
If furthermore $\phi_{A'}(J) = 0$,
then $\phi_{B'}(J_B) = 0$.
\end{lem}

\begin{proof}
We have the following exact sequence for every $m \geq 1$
\[
0 \to (J+(\pi, I')^m)/(\pi, I')^m \to A'/(\pi, I')^m \to A/(\pi, I)^m \to 0.
\]
Since $A'/(\pi, I')^m \to B'/(\pi, I')^m$ is flat,
the induced sequence
\[
0 \to (JB'+(\pi, I')^m)/(\pi, I')^m \to B'/(\pi, I')^m \to B/(\pi, I)^m \to 0
\]
is also exact.
By taking the limit over $m$, we see that $B' \to B$ is surjective, and
\[
J_B \simeq {\varprojlim}_m (JB'+(\pi, I')^m)/(\pi, I')^m.
\]
Using this isomorphism, we see that $J^n_B=0$, and that $\phi_{B'}(J_B) = 0$ if $\phi_{A'}(J) =0$.
\end{proof}

\begin{defn}[Special nilpotent thickening]\label{Definition:special nilpotent thickening}
Let
$
f \colon (A', I') \to (A, I)
$
be a nilpotent thickening of orientable and bounded $\O_E$-prisms over $\O$.
Let $J \subset A'$ be the kernel of $f \colon A' \to A$.
We say that $f$ is a \textit{special nilpotent thickening} if
it satisfies the following two conditions:
\begin{enumerate}
    \item We have $\phi_{A'}(J) = 0$.
    \item Let $d \in I'$ be a generator.
    For any $(\pi, I')$-completely \'etale $A'$-algebra $B'$,
    the element $\phi_{B'}(d)$ is a nonzerodivisor in $B'$.
\end{enumerate}
\end{defn}

\begin{rem}\label{Remark:thickening etale}
It follows from Lemma \ref{Lemma:thickening etale morphism} that
if $f \colon (A', I') \to (A, I)$ is a
nilpotent thickening
(resp.\ a special nilpotent thickening), then for any $B' \in (A', I')_\et$ with corresponding $B \in (A, I)_\et$,
the induced map $(B', I'B') \to (B, IB)$ is also a nilpotent thickening
(resp.\ a special nilpotent thickening).
\end{rem}


\begin{ex}\label{Example:phi(d) non zero divisor:Breuil Kisin}
    If $(A, I)$ is either $(\mathfrak{S}_\O, (\mathcal{E}))$ or $(\mathfrak{S}_{\O, m}, (\mathcal{E}))$
    as in Example \ref{Example:Breuil-Kisin type frame}, then $(A, I)$ satisfies
    the condition (2) in Definition \ref{Definition:special nilpotent thickening}.
    Indeed, it is clear that $\phi(\mathcal{E})$ is a nonzerodivisor in $A$.
    Since $A$ is noetherian, any $(\pi, I)$-completely \'etale $A$-algebra $B$ is flat over $A$, which in turn implies that $\phi(\mathcal{E})$ is a nonzerodivisor in $B$.
\end{ex}

\begin{ex}\label{Example:phi(d) non zero divisor:perfectoid}
Let $(S, a^\flat)$ be a perfectoid pair over $\O$ (Definition \ref{Definition:perfectoid pair}).
The $\O_E$-prism
    $
    (W_{\O_E}(S^\flat)/[a^\flat]^m, I_S)
    $
    as in Proposition \ref{Proposition:perfectoid type} satisfies the condition (2) in Definition \ref{Definition:special nilpotent thickening}.
    Indeed, by Lemma \ref{Lemma:perfectoid ring etale morphism:torsion case}, it suffices to show that
    $\phi(\xi)$ is a nonzerodivisor in $W_{\O_E}(S^\flat)/[a^\flat]^m$.
    Since
    $
    \phi \colon W_{\O_E}(S^\flat)/[(a^\flat)^{1/q}]^m \to W_{\O_E}(S^\flat)/[a^\flat]^m
    $
    is bijective and $\xi$ is a nonzerodivisor in $W_{\O_E}(S^\flat)/[(a^\flat)^{1/q}]^m$ (by Proposition \ref{Proposition:perfectoid type}), the assertion follows.
\end{ex}

\begin{ex}\label{Example:special nilpotent thickening}
Let the notation be as in Example \ref{Example:phi(d) non zero divisor:Breuil Kisin} and Example \ref{Example:phi(d) non zero divisor:perfectoid}.
In this paper, we say that a nilpotent thickening
$f \colon (A', I') \to (A, I)$
is \textit{of Breuil--Kisin type} (resp.\ \textit{of perfectoid type})
if it is of the form
\begin{align*}
    (\mathfrak{S}_{\O, m+1}, (\mathcal{E})) &\to (\mathfrak{S}_{\O, m}, (\mathcal{E})) \\
    (\text{resp.\ } (W_{\O_E}(S^\flat)/[a^\flat]^{m+1}, I_S)
    &\to (W_{\O_E}(S^\flat)/[a^\flat]^{m}, I_S)).
\end{align*}
In either case $f$ is a special nilpotent thickening.
\end{ex}

Deformations of $G$-$\phi$-modules are defined in the usual way.
If $\mathcal{Q}$ is a $G$-$\mu$-display over $(A, I)$ and 
$\mathscr{Q}$
is a deformation of $\mathcal{Q}$
over $(A', I')$, then
$\mathscr{Q}_\phi$ is naturally a deformation of $\mathcal{Q}_\phi$ over $(A', I')$.
Here $\mathscr{Q}_\phi$ and $\mathcal{Q}_\phi$
are the underlying $G$-$\phi$-modules; see Definition \ref{Definition:underlying phi G torsor}.
The following proposition is the key ingredient in the construction of period maps introduced in Section \ref{Subsection:Period map} below.

\begin{prop}\label{Proposition:canonical isomorphism of deformations of phi G torsors}
Assume that $f \colon (A', I') \to (A, I)$ is a special nilpotent thickening.
Let
$\mathcal{Q}$
be a $G$-$\mu$-display over $(A, I)$.
Let $(\mathscr{Q}, h)$ and $(\mathscr{R}, i)$
be two deformations
of $\mathcal{Q}$ over $(A', I')$.
Then the following assertions hold:
\begin{enumerate}
    \item There exists a unique isomorphism
$
\psi \colon \mathscr{Q}_{\phi} \overset{\sim}{\to} \mathscr{R}_{\phi}
$
of deformations of $\mathcal{Q}_{\phi}$ over $(A', I')$.
    \item There is at most one isomorphism
    between two deformations $(\mathscr{Q}, h)$ and $(\mathscr{R}, i)$.
\end{enumerate}
\end{prop}

\begin{proof}
(1)
The problem is $(\pi, I)$-completely \'etale local on $A$ by
$(\pi, I)$-completely \'etale descent
for $G$-$\phi$-modules (which follows from the same argument as in \cite[Remark 5.1.3]{Ito-K23}) and the uniqueness assertion.
Thus we may assume that
$\mathscr{Q}$, $\mathscr{R}$, and $\mathcal{Q}$
are banal.
We can identify $\mathcal{Q}$ with $\mathcal{Q}_X$ for some $X \in G(A)_I$.
By Proposition \ref{Proposition:deformation of banal G displays},
we may assume that
$(\mathscr{Q}, h)$ and $(\mathscr{R}, i)$
are of the form
$(\mathcal{Q}_Y, 1)$ and $(\mathcal{Q}_Z, 1)$, respectively, for
some
$Y, Z \in \mathbf{Def}(X)_{(A', I')}$.
It suffices to prove that
there exists a unique element $\psi \in G(A')$ such that $f(\psi)=1$ in $G(A)$ and
\begin{equation}\label{equation:deformation equalities}
    \psi^{-1}Z_\phi \phi(\psi) =Y_\phi \quad \text{in} \quad G(A'[1/\phi(I')]).
\end{equation}
Here $Y_\phi, Z_\phi \in G(A'[1/\phi(I')])$ are the elements defined in (\ref{equation:Xphi}).

We fix a generator $d \in I'$.
We shall show that the element
\[
\psi:=Z_d(Y_d)^{-1} \in G(A')
\]
satisfies $f(\psi)=1$ and
(\ref{equation:deformation equalities}).
Indeed, since $X=f(Y)=f(Z)$,
we have $f(\psi)=1$.
This implies that $f(\mu(d)\psi\mu(d)^{-1})=1$.
By our assumption that $\phi_{A'}(J)=0$ for the kernel $J \subset A'$ of $f \colon A' \to A$,
it then follows that $\phi(\mu(d)\psi\mu(d)^{-1})=1$, whence (\ref{equation:deformation equalities}) holds.

It remains to show the uniqueness of $\psi$.
Let $\psi \in G(A')$ be an element satisfying $f(\psi)=1$ and
(\ref{equation:deformation equalities}).
As explained above, the equality $f(\psi)=1$ implies $\phi(\mu(d)\psi\mu(d)^{-1})=1$.
Then it follows from (\ref{equation:deformation equalities}) that $\psi=Z_d(Y_d)^{-1}$ in $G(A'[1/\phi(I')])$.
Since the natural homomorphism $G(A') \to G(A'[1/\phi(I')])$ is injective by the condition (2) in Definition \ref{Definition:special nilpotent thickening}, we conclude that $\psi=Z_d(Y_d)^{-1}$ in $G(A')$.

(2)
We note that the functor
$\mathcal{Q} \mapsto \mathcal{Q}_{\phi}$
from
$G\mathchar`-\mathrm{Disp}_\mu(A, I)$
to the groupoid of $G$-$\phi$-modules over $(A, I)$ is faithful
(since the functor
from
the groupoid of $G_{\mu, A, I}$-torsors to that of $G_{\Prism, A}$-torsors
obtained by pushout along the inclusion 
$\tau \colon G_{\mu, A, I} \hookrightarrow G_{\Prism, A}$
is faithful).
Then (2) follows from the uniqueness assertion in (1).
\end{proof}

\begin{rem}\label{Remark:isomorphism is unique for deformations}
Let the notation be as in Proposition \ref{Proposition:deformation of banal G displays}.
Let $Y, Z \in \mathbf{Def}(X)_{(A', I')}$.
For the deformations
$(\mathcal{Q}_Y, 1)$
and
$(\mathcal{Q}_Z, 1)$
of the $G$-$\mu$-display
$\mathcal{Q}_X$,
the proof of Proposition \ref{Proposition:canonical isomorphism of deformations of phi G torsors} shows that for any generator $d \in I'$,
\[
(\mathcal{Q}_Y, 1) \simeq (\mathcal{Q}_Z, 1) \quad \text{if and only if} \quad Z_d(Y_d)^{-1} \in G_{\mu}(A', I').
\]
If this is the case, the unique isomorphism
$(\mathcal{Q}_Y, 1) \overset{\sim}{\to} (\mathcal{Q}_Z, 1)$
is given by $Z_d(Y_d)^{-1}$.
\end{rem}

\begin{cor}\label{Corollary:deformations form a sheaf}
    Assume that $f \colon (A', I') \to (A, I)$ is a special nilpotent thickening.
    Let
    $\mathcal{Q}$
    be a $G$-$\mu$-display over $(A, I)$.
    For each $B' \in (A', I')_\et$ with corresponding $B \in (A, I)_\et$,
    we denote the base change of
    $\mathcal{Q}$
    to $(B, IB)$ by
    $\mathcal{Q}_{(B, IB)}$.
Then the functor
\[
\underline{\Def}(\mathcal{Q}) \colon (A', I')_\et \to \mathrm{Set}, \quad B' \mapsto \underline{\Def}(\mathcal{Q})(B')=\Def(\mathcal{Q}_{(B, IB)})_{(B', I'B')}
\]
forms a sheaf with respect to the $(\pi, I')$-completely \'etale topology.
\end{cor}

\begin{proof}
    This is an immediate consequence of Proposition \ref{Proposition:canonical isomorphism of deformations of phi G torsors} (2) and $(\pi, I)$-completely \'etale descent for $G$-$\mu$-displays.
\end{proof}

\begin{cor}\label{Corollary:Uniqueness of isomorphisms between deformations:BK and Perfd case}
Let the notation be as in Example \ref{Example:phi(d) non zero divisor:Breuil Kisin} and Example \ref{Example:phi(d) non zero divisor:perfectoid}.
\begin{enumerate}
    \item Let $\mathcal{Q}$ be a $G$-$\mu$-display over $(\O, (\pi))=(\mathfrak{S}_{\O, 1}, (\mathcal{E}))$.
    Let $m \geq 1$ be an integer.
    Any deformation of $\mathcal{Q}$ over
$(\mathfrak{S}_{\O, m}, (\mathcal{E}))$
has no nontrivial automorphisms.
    Moreover, any deformation of $\mathcal{Q}$ over
$(\mathfrak{S}_\O, (\mathcal{E}))$
has no nontrivial automorphisms.
    \item 
    Let $m \geq 1$ be an integer.
    For a $G$-$\mu$-display $\mathcal{Q}$ over $(W_{\O_E}(S^\flat)/[a^\flat], I_S)$, any deformation of $\mathcal{Q}$ over $(W_{\O_E}(S^\flat)/[a^\flat]^m, I_S)$ has no nontrivial automorphisms.
\end{enumerate}
\end{cor}

\begin{proof}
(1)
By Example \ref{Example:banal over finite etale covering}, it suffices to prove the first assertion.
This follows by applying Proposition \ref{Proposition:canonical isomorphism of deformations of phi G torsors} (2) to special nilpotent thickenings 
$(\mathfrak{S}_{\O, l+1}, (\mathcal{E})) \to (\mathfrak{S}_{\O, l}, (\mathcal{E}))$
$(1 \leq l \leq m-1)$ repeatedly.

(2) This also follows from Proposition \ref{Proposition:canonical isomorphism of deformations of phi G torsors} (2).
\end{proof}

\subsection{Period maps}\label{Subsection:Period map}

Let $f \colon (A', I') \to (A, I)$ be a
special nilpotent thickening
of orientable and bounded $\O_E$-prisms over $\O$.
Let
$\mathcal{Q}$
be a $G$-$\mu$-display over $(A, I)$.
Recall the Hodge filtration
$P(\mathcal{Q})_{A/I} \to \mathcal{Q}_{A/I}$
from Section \ref{Subsection:Hodge filtrations and G-torsors}.

\begin{defn}\label{Definition:lift of Hodge filtration}
Let $\mathscr{Q}$ be a $G$-torsor over $\Spec A'/I'$ with an isomorphism
\[
f^*\mathscr{Q}:=\mathscr{Q} \times_{\Spec A'/I'} \Spec A/I \overset{\sim}{\to} \mathcal{Q}_{A/I}
\]
of $G$-torsors over $\Spec A/I$.
A \textit{lift} of the Hodge filtration
$P(\mathcal{Q})_{A/I} \to \mathcal{Q}_{A/I}$
in $\mathscr{Q}$
is
a pair
$(\mathscr{P}, \iota)$
of
a $P_\mu$-torsor
$\mathscr{P}$
over $\Spec A'/I'$
and
a $P_\mu$-equivariant morphism
$\iota \colon \mathscr{P} \to \mathscr{Q}$
such that the isomorphism 
$f^*\mathscr{Q} \overset{\sim}{\to} \mathcal{Q}_{A/I}$
restricts to an isomorphism
$f^*\mathscr{P} \overset{\sim}{\to} P(\mathcal{Q})_{A/I}$.
\end{defn}

There is an obvious notion of isomorphism of lifts of the Hodge filtration.
We note that there is at most one isomorphism between two lifts.
Let
\[
\Lift(P(\mathcal{Q})_{A/I}, \mathscr{Q})
\]
be the set of isomorphism classes of lifts of the Hodge filtration
$P(\mathcal{Q})_{A/I} \to \mathcal{Q}_{A/I}$
in $\mathscr{Q}$.

\begin{ex}\label{Example:lift of Hodge filtration}
For a deformation
$\mathscr{Q}$
of $\mathcal{Q}$ over $(A', I')$,
the pair
\[
(P(\mathscr{Q})_{A'/I'}, P(\mathscr{Q})_{A'/I'} \to \mathscr{Q}_{A'/I'})
\]
is a lift of the Hodge filtration $P(\mathcal{Q})_{A/I}$ in $\mathscr{Q}_{A'/I'}$.
\end{ex}

\begin{defn}[{Period map}]\label{Definition:period map}
Let
$\mathscr{Q}$
be a deformation
of $\mathcal{Q}$ over $(A', I')$.
We define a map of sets
\[
\Per_{\mathscr{Q}} \colon \Def(\mathcal{Q})_{(A', I')} \to
\Lift(P(\mathcal{Q})_{A/I}, \mathscr{Q}_{A'/I'})
\]
as follows.
Let $\mathscr{R} \in \Def(\mathcal{Q})_{(A', I')}$
be a deformation of $\mathcal{Q}$ over $(A', I')$.
By Proposition \ref{Proposition:canonical isomorphism of deformations of phi G torsors}, there exists a unique isomorphism
$
\psi \colon \mathscr{R}_{\phi} \overset{\sim}{\to} \mathscr{Q}_{\phi}
$
of deformations of the underlying $G$-$\phi$-module $\mathcal{Q}_{\phi}$.
In particular $\psi$ induces an isomorphism
$\mathscr{R}_{A'/I'} \overset{\sim}{\to} \mathscr{Q}_{A'/I'}$ of
$G$-torsors over $\Spec A'/I'$.
We define the map $\Per_{\mathscr{Q}}$ by sending $\mathscr{R}$ to the lift
\[
(P(\mathscr{R})_{A'/I'}, \, P(\mathscr{R})_{A'/I'} \to \mathscr{R}_{A'/I'} \overset{\sim}{\to} \mathscr{Q}_{A'/I'}) \in
\Lift(P(\mathcal{Q})_{A/I}, \mathscr{Q}_{A'/I'}).
\]
We call $\Per_{\mathscr{Q}}$ the \textit{period map} associated with the deformation $\mathscr{Q}$.
\end{defn}

The main result of this section is the following theorem, which can be seen as an analogue of the Grothendieck--Messing deformation theory.

\begin{thm}\label{Theorem:GM deformation}
Let $f \colon (A', I') \to (A, I)$ be a special nilpotent thickening of orientable and bounded $\O_E$-prisms over $\O$.
Let
$\mathcal{Q}$
be a $G$-$\mu$-display over $(A, I)$ and
we fix a deformation
$\mathscr{Q}$
of $\mathcal{Q}$ over $(A', I')$.
If the cocharacter $\mu$ is 1-bounded,
then the period map
\[
\Per_{\mathscr{Q}} \colon \Def(\mathcal{Q})_{(A', I')} \to \Lift(P(\mathcal{Q})_{A/I}, \mathscr{Q}_{A'/I'})
\]
is bijective.
\end{thm}

\begin{proof}
By Corollary \ref{Corollary:deformations form a sheaf},
we may work $(\pi, I')$-completely \'etale locally on $A'$.
We may thus assume that $\mathcal{Q}=\mathcal{Q}_X$ for some $X \in G(A)_I$.
We shall prove that $\Per_{\mathscr{Q}}$ is injective.
Let $\mathscr{R}$ and $\mathscr{S}$
be two deformations
of $\mathcal{Q}$ over $(A', I')$
such that
$P(\mathscr{R})_{A'/I'} = P(\mathscr{S})_{A'/I'}$
in $\Lift(P(\mathcal{Q})_{A/I}, \mathscr{Q}_{A'/I'})$.
By Proposition \ref{Proposition:deformation of banal G displays},
we may further assume that
\[
\mathscr{Q}=(\mathcal{Q}_{X'}, 1), \quad \mathscr{R}=(\mathcal{Q}_Y, 1), \quad \text{and} \quad \mathscr{S}=(\mathcal{Q}_Z, 1)
\]
for some $X', Y, Z \in \mathbf{Def}(X)_{(A', I')}$.
Then
$\mathscr{Q}_{A'/I'}$
is naturally identified with $G_{A'/I'}$,
and the isomorphism
$f^*(\mathscr{Q}_{A'/I'}) \overset{\sim}{\to} \mathcal{Q}_{A/I}$ agrees with the canonical isomorphism
$f^*G_{A'/I'} \overset{\sim}{\to} G_{A/I}$.
We fix a generator $d \in I'$.
The period map $\Per_{\mathscr{Q}}$
sends $(\mathcal{Q}_Y, 1)$ and $(\mathcal{Q}_Z, 1)$ to
\[
((P_{\mu})_{A'/I'}, \iota_{X'_d(Y_d)^{-1}}) \quad \text{and} \quad ((P_{\mu})_{A'/I'}, \iota_{X'_d(Z_d)^{-1}}),
\]
respectively, where $\iota_{X'_d(Y_d)^{-1}} \colon (P_{\mu})_{A'/I'} \hookrightarrow G_{A'/I'}$ is defined by $g \mapsto X'_d(Y_d)^{-1}g$,
and similarly for $\iota_{X'_d(Z_d)^{-1}}$.
(See (\ref{equation:d-component}) for the notation $X_d$.)
Since the two lifts are isomorphic,
the image of $Z_d(Y_d)^{-1}$ in $G(A'/I')$ belongs to
$P_\mu(A'/I')$.
This implies that $Z_d(Y_d)^{-1} \in G_\mu(A', I')$ by Proposition \ref{Proposition:BB isomorphism} since $\mu$ is 1-bounded.
Then $Z_d(Y_d)^{-1}$ gives
$(\mathcal{Q}_Y, 1) \overset{\sim}{\to} (\mathcal{Q}_Z, 1)$
by Remark \ref{Remark:isomorphism is unique for deformations}.

We next prove that $\Per_{\mathscr{Q}}$
is surjective.
We may assume that $\mathscr{Q}=(\mathcal{Q}_{X'}, 1)$ as above.
Let $(\mathscr{P}, \iota) \in \Lift(P(\mathcal{Q})_{A/I}, \mathscr{Q}_{A'/I'})$.
By the same argument as in the proof of Lemma \ref{Lemma:deformation of banal is banal},
we see that $\mathscr{P}$ is trivial as a $P_\mu$-torsor.
Then, with the notation as above, this lift can be identified with
$
((P_{\mu})_{A'/I'}, \iota_{x})
$
for some $x \in G(A'/I')$ whose image in $G(A/I)$ is $1$.
By Lemma \ref{Lemma:used in GM deformation} below,
after replacing $x$ by $x t$ for some $t \in P_\mu(A'/I')$, 
we can find an element $\widetilde{x} \in G(A')$ such that $f(\widetilde{x})=1$ in $G(A)$
and
the image of $\widetilde{x}$ in $G(A'/I')$ coincides with $x$.
We then define $Y \in \mathbf{Def}(X)_{(A', I')}$ to be the unique object such that $Y_d=\widetilde{x}^{-1} X'_d$.
By construction,
the period map $\Per_{\mathscr{Q}}$
sends $(\mathcal{Q}_Y, 1)$
to
$((P_{\mu})_{A'/I'}, \iota_{x})$.
This concludes the proof of the surjectivity of $\Per_{\mathscr{Q}}$, and hence that of Theorem \ref{Theorem:GM deformation}.
\end{proof}

The following lemma is used in the proof of Theorem \ref{Theorem:GM deformation}.

\begin{lem}\label{Lemma:used in GM deformation}
Let the notation be as in Theorem \ref{Theorem:GM deformation}.
Let $x \in G(A'/I')$ be such that $f(x)=1$ in $G(A/I)$.
Then
there exist $t \in P_\mu(A'/I')$ and $\widetilde{x} \in G(A')$ such that $f(\widetilde{x})=1$ in $G(A)$ and
the image of $\widetilde{x}$ in $G(A'/I')$ is equal to
$x t$.
\end{lem}

\begin{proof}
Since $G$ is smooth and $A'$ is $I'$-adically complete, the map
$G(A') \to G(A'/I')$
is surjective.
Let $\widetilde{x}_1 \in G(A')$ be an element which is mapped to $x \in G(A'/I')$.
By Proposition \ref{Proposition:BB isomorphism},
we have $f(\widetilde{x}_1) \in G_\mu(A, I)$.
By Lemma \ref{Lemma:Gmu surjective nilpotent thickening},
there is an element $\widetilde{x}_2 \in G_\mu(A', I')$ such that $f(\widetilde{x}_2)=f(\widetilde{x}_1)$ in $G_\mu(A, I)$.
The image $t \in G(A'/I')$ of $\widetilde{x}^{-1}_2$ belongs to $P_\mu(A'/I')$.
Then $t$ and
$\widetilde{x}:=\widetilde{x}_1\widetilde{x}^{-1}_2$ have the desired properties.
\end{proof}

\begin{rem}\label{Remark:existence of deformation}
A banal $G$-$\mu$-display
over $(A, I)$
admits a deformation over $(A', I')$ 
since $G(A') \to G(A)$ is surjective.
However, for a $G$-$\mu$-display
$\mathcal{Q}$
over $(A, I)$ which is not necessarily banal, the existence of a deformation of $\mathcal{Q}$ over $(A', I')$ is not clear.
In Proposition \ref{Proposition:existence of deformations for non-banal cases} below,
under a mild additional assumption on
$f \colon (A', I') \to (A, I)$,
we will prove that
a deformation of $\mathcal{Q}$ over $(A', I')$ exists if $\mu$ is 1-bounded.
\end{rem}

\subsection{Normalized period maps}\label{Subsection:Normalized period map}

Let $f \colon (A', I') \to (A, I)$ be a special nilpotent thickening of orientable and bounded $\O_E$-prisms over $\O$.
In this subsection, we assume that $f$ satisfies the following assumption:

\begin{ass}\label{Assumption for special nilpotent thickening}
The kernel $K$ of $A'/I' \to A/I$ satisfies $K^2=0$.
Moreover, there exists an integer $n \geq 1$ such that $\pi^n=0$ in $A'/I'$.
\end{ass}

\begin{rem}
A special nilpotent thickening of Breuil--Kisin type or of perfectoid type (Example \ref{Example:special nilpotent thickening}) satisfies Assumption \ref{Assumption for special nilpotent thickening}.
\end{rem}

Let $\mathcal{Q}$ be a banal $G$-$\mu$-display over $(A, I)$.
The quotient (fppf) sheaf
\[
X_\mu(\mathcal{Q}):=\mathcal{Q}_{A/I}/(P_\mu)_{A/I}
\]
is representable by a scheme over $\Spec A/I$, which is also denoted by $X_\mu(\mathcal{Q})$.
(To see this, it suffices to show that the quotient sheaf $G_{\O/\pi^n}/(P_\mu)_{\O/\pi^n}$ is representable, which follows from \cite[Expose $\mathrm{VI_A}$, Th\'eor\`eme 3.2]{SGA3-1}.
We note that
$\mathcal{Q}_{A/I} \simeq G_{A/I}$ since $\mathcal{Q}$ is banal.)
The closed immersion
$P(\mathcal{Q})_{A/I} \hookrightarrow \mathcal{Q}_{A/I}$
induces a section
\[
s_{\mathcal{Q}} \colon \Spec A/I \hookrightarrow X_\mu(\mathcal{Q}).
\]
The $A/I$-module corresponding to the conormal sheaf of $s_{\mathcal{Q}}$ on $\Spec A/I$ is denoted by $C(\mathcal{Q})$.
We call $C(\mathcal{Q})$ the \textit{conormal module}.

\begin{defn}\label{Definition:universal local ring}
Let $\O_{U^{-}_{\mu}, 1}$
be the local ring of $U^{-}_{\mu}$ at $1 \in U^{-}_{\mu}(k)$ and let
\[
R_{G, \mu}:= \widehat{\O}_{U^{-}_{\mu}, 1}
\]
be the completion of $\O_{U^{-}_{\mu}, 1}$ with respect to the maximal ideal.
We denote by
$
J_{G, \mu} \subset R_{G, \mu}
$
the kernel of the homomorphism
$\epsilon \colon R_{G, \mu} \to \O$
induced by the unit morphism
$\Spec \O \to U^{-}_{\mu}$.
We set
\[
C(R_{G, \mu}):=J_{G, \mu}/J^2_{G, \mu},
\]
which is a free $\O$-module of finite rank.
\end{defn}

\begin{ex}\label{Example:conormal module}
    An isomorphism
$\mathcal{Q} \overset{\sim}{\to} \mathcal{Q}_{X}$
of $G$-$\mu$-displays over $(A, I)$ for an element $X \in G(A)_I$ induces
$
X_\mu(\mathcal{Q}) \overset{\sim}{\to} G_{A/I}/(P_\mu)_{A/I}
$
and
$
C(R_{G, \mu}) \otimes_\O A/I \overset{\sim}{\to} 
    C(\mathcal{Q}).
$
\end{ex}

In the following, for modules $M$ and $N$ over a ring $R$,
the set of $R$-linear homomorphisms $M \to N$ is denoted by $\Hom_R(M, N)$.

\begin{lem}\label{Lemma:lift of Hodge filtration banal case, conormal module}
Let $\mathcal{Q}$ be a banal $G$-$\mu$-display over $(A, I)$, and let $\mathscr{Q}$ be a deformation of $\mathcal{Q}$ over $(A', I')$.
Then there exists a natural bijection
\[
\Lift(P(\mathcal{Q})_{A/I}, \mathscr{Q}_{A'/I'}) \overset{\sim}{\to} \Hom_{A/I}(C(\mathcal{Q}), K).
\]
(Since $K^2=0$, the ideal $K \subset A'/I'$ can be naturally regarded as an $A/I$-module.)
\end{lem}

\begin{proof}
    As above, the Hodge filtration
    $P(\mathscr{Q})_{A'/I'} \to \mathscr{Q}_{A'/I'}$
    induces a section
    \[
    s_\mathscr{Q} \colon \Spec A'/I' \hookrightarrow X_\mu(\mathscr{Q}):=\mathscr{Q}_{A'/I'}/(P_\mu)_{A'/I'}.
    \]
    Let $C(\mathscr{Q})$ be the conormal module of $s_\mathscr{Q}$.
    We have $C(\mathscr{Q})\otimes_{A'/I'} A/I \simeq C(\mathcal{Q})$.
    Let $s \colon \Spec A/I \hookrightarrow X_\mu(\mathscr{Q})$
    be the composition of $s_\mathscr{Q}$ with $\Spec A/I \hookrightarrow \Spec A'/I'$, and let
    $C_s$ be the conormal module of $s$.
    We have a natural homomorphism
    $
    C(\mathscr{Q})\otimes_{A'/I'} A/I \to C_s
    $
    of $A/I$-modules.
    
    Let
    $\iota \colon \mathscr{P} \to \mathscr{Q}_{A'/I'}$
    be a lift of $P(\mathcal{Q})_{A/I}$.
    In the same way as above, this induces a section
    $
    s_{\mathscr{P}} \colon \Spec A'/I' \to X_\mu(\mathscr{Q}).
    $
    Since the composition of $s_{\mathscr{P}}$ with
    $\Spec A/I \hookrightarrow \Spec A'/I'$ is equal to $s$,
    the morphism $s_{\mathscr{P}}$
    induces a homomorphism
    $t_{\mathscr{P}} \colon C_s \to K$
    of $A/I$-modules (defined by $x \mapsto s^*_{\mathscr{P}}(x)$).
    Associating to $\iota \colon \mathscr{P} \to \mathscr{Q}_{A'/I'}$ the composition
    \[
    C(\mathcal{Q}) \simeq C(\mathscr{Q})\otimes_{A'/I'} A/I \to C_s \overset{t_{\mathscr{P}}}{\to} K,
    \]
    we obtain a map
    \[
    \Lift(P(\mathcal{Q})_{A/I}, \mathscr{Q}_{A'/I'}) \to \Hom_{A/I}(C(\mathcal{Q}), K).
    \]

    We shall prove that this map is bijective.
    For this, we may assume that $\mathcal{Q}=\mathcal{Q}_X$ for some
    $X \in G(A)_I$
    and $\mathscr{Q}=(\mathcal{Q}_{X'}, 1)$
    for some $X' \in \mathbf{Def}(X)_{(A', I')}$
    as in the proof of Theorem \ref{Theorem:GM deformation}.
    Then we have
    $P(\mathcal{Q})_{A/I}=(P_\mu)_{A/I}$
    and
    $\mathscr{Q}_{A'/I'}=G_{A'/I'}$, and by the argument in the proof of Theorem \ref{Theorem:GM deformation},
    we can identify
    $\Lift(P(\mathcal{Q})_{A/I}, \mathscr{Q}_{A'/I'})$
    with the set of $P_\mu(A'/I')$-orbits
    $
    x P_\mu(A'/I') \subset G(A'/I')
    $
    of elements $x$ of the kernel of $G(A'/I') \to G(A/I)$.
    We note that for such elements $x$, 
    there are elements $u \in U^-_{\mu}(A'/I')$ and $t \in P_\mu(A'/I')$ such that
    $x=ut$
    since the multiplication map
    \[
    U^{-}_{\mu} \times_{\Spec \O} P_\mu \to G_\O
    \]
    is an open immersion; see also the proof of \cite[Proposition 4.2.8]{Ito-K23}.
    Thus we can identify
    $\Lift(P(\mathcal{Q})_{A/I}, \mathscr{Q}_{A'/I'})$
    with the set of $P_\mu(A'/I')$-orbits
    $
    x P_\mu(A'/I') \subset G(A'/I')
    $
    of elements $x$ of the kernel of $U^{-}_{\mu}(A'/I') \to U^{-}_{\mu}(A/I)$.
    Using this observation and using the fact that
    $U^{-}_{\mu} \times_{\Spec \O} P_\mu \to G_\O$
    is an open immersion again, we can identify
    $\Lift(P(\mathcal{Q})_{A/I}, \mathscr{Q}_{A'/I'})$
    with the kernel of $U^{-}_{\mu}(A'/I') \to U^{-}_{\mu}(A/I)$.
    On the other hand,
    we have
    \[
    \Hom_{A/I}(C(\mathcal{Q}), K) = \Hom_\O(C(R_{G, \mu}), K)
    \]
    by Example \ref{Example:conormal module}.
    The map
    $
    \Lift(P(\mathcal{Q})_{A/I}, \mathscr{Q}_{A'/I'}) \to \Hom_{A/I}(C(\mathcal{Q}), K)
    $
    sends
    an element $x \colon \Spec A'/I' \to U^{-}_{\mu}$ of the kernel of $U^{-}_{\mu}(A'/I') \to U^{-}_{\mu}(A/I)$
    to the homomorphism
    $x^* \colon C(R_{G, \mu}) \to K$.
    Using this description, it is easy to see that this map is bijective.
\end{proof}

The proof of the following result is inspired by that of \cite[Theorem 3.5.11]{Bultel-Pappas}.

\begin{prop}\label{Proposition:existence of deformations for non-banal cases}
Assume that $\mu$ is 1-bounded.
Let $\mathcal{Q}$ be a (not necessarily banal) $G$-$\mu$-display over $(A, I)$.
If $f \colon (A', I') \to (A, I)$ satisfies Assumption \ref{Assumption for special nilpotent thickening}, then 
the set
$\Def(\mathcal{Q})_{(A', I')}$
of isomorphism classes of deformations of $\mathcal{Q}$ over $(A', I')$ has the structure of a torsor under an $A/I$-module.
In particular
$\Def(\mathcal{Q})_{(A', I')}$
is nonempty.
\end{prop}

\begin{proof}
We recall that the functor
$
(A', I')_\et \to (A, I)_\et
$
defined by
$
B' \mapsto B,
$
where $B$ is the $(\pi, I)$-adic completion of $B' \otimes_{A'} A$,
is an equivalence.
Moreover, 
any $\pi$-completely \'etale homomorphism $A/I \to C$ is \'etale since $\pi^n=0$ in $A/I$, and thus the category $(A, I)_\et$ is equivalent to the category
$(A/I)_\et$
of \'etale $A/I$-algebras (see \cite[Lemma 2.5.9]{Ito-K23}).

By Corollary \ref{Corollary:deformations form a sheaf},
the functor
\[
\underline{\Def}(\mathcal{Q}) \colon (A', I')_\et \to \mathrm{Set}, \quad B' \mapsto \underline{\Def}(\mathcal{Q})(B')=\Def(\mathcal{Q}_{(B, IB)})_{(B', I'B')}
\]
forms a sheaf.
Let $B' \in (A', I')_\et$ be such that $\mathcal{Q}_{(B, IB)}$ is banal.
We choose a deformation $\mathscr{Q}$ of $\mathcal{Q}_{(B, IB)}$ over $(B', I'B')$, which exists since $G(B') \to G(B)$ is surjective.
To simplify the notation,
the base change $\mathcal{Q}_{(B, IB)}$ is also denoted by $\mathcal{Q}$.
By Theorem \ref{Theorem:GM deformation}, the period map
\[
\Per_{\mathscr{Q}} \colon \underline{\Def}(\mathcal{Q})(B')\to \Lift(P(\mathcal{Q})_{B/IB}, \mathscr{Q}_{B'/I'B'})
\]
is bijective.
By Lemma \ref{Lemma:lift of Hodge filtration banal case, conormal module},
we have
\[
\Lift(P(\mathcal{Q})_{B/IB}, \mathscr{Q}_{B'/I'B'}) \overset{\sim}{\to} \Hom_{B/IB}(C(\mathcal{Q}), K \otimes_{A/I} B/IB).
\]
We note that $K \otimes_{A/I} B/IB$ is the kernel of $B'/I'B' \to B/IB$.
We endow the set
$\underline{\Def}(\mathcal{Q})(B')$
with
an action of
the $B/IB$-module
$\Hom_{B/IB}(C(\mathcal{Q}), K \otimes_{A/I} B/IB)$ via these identifications.
One can check that this action does not depend on the choice of $\mathscr{Q}$.
By \'etale descent for quasi-coherent sheaves,
the $B/IB$-modules
$\Hom_{B/IB}(C(\mathcal{Q}), K \otimes_{A/I} B/IB)$
give rise to a quasi-coherent sheaf $\mathcal{F}$ on
$(A/I)^{\op}_\et$, and $\underline{\Def}(\mathcal{Q})$ is an $\mathcal{F}$-torsor on $(A/I)^{\op}_\et$.
Since any $\mathcal{F}$-torsor on $(A/I)^{\op}_\et$ is trivial, 
the result follows.
\end{proof}

In the following, the maximal ideal of a local ring $R$ is denoted by $\mathfrak{m}_R$.

\begin{rem}\label{Remark:tangent space}
For a local ring $R$ with a local homomorphism $\O \to R$,
we define
\[
\mathfrak{t}_R:=(\mathfrak{m}_R/(\mathfrak{m}^2_R+\pi R))^\vee:=\Hom_k(\mathfrak{m}_R/(\mathfrak{m}^2_R+\pi R), k),
\]
which is called the \textit{tangent space} of $R$ over $\O$.
When $R=R_{G, \mu}$, we write
$
\mathfrak{m}_{G, \mu}:=\mathfrak{m}_{R_{G, \mu}} 
$
and
$
\mathfrak{t}_{G, \mu}:=\mathfrak{t}_{R_{G, \mu}} 
$
for simplicity.
The natural homomorphism
\[
C(R_{G, \mu}) \otimes_\O k \to \mathfrak{m}_{G, \mu}/(\mathfrak{m}^2_{G, \mu}+\pi R_{G, \mu})=\mathfrak{t}^\vee_{G, \mu}
\]
is an isomorphism.
\end{rem}

\begin{rem}[{Normalized period map}]\label{Remark:normalized period map}
We assume that $K$ is killed by $\pi$, so that $K$ is a $k$-vector space.
Let
$\mathcal{Q}$
be a banal $G$-$\mu$-display over $(A, I)$
and
$\mathscr{Q}$
a deformation of $\mathcal{Q}$ over
$(A', I')$.
We choose
an isomorphism
$\beta \colon \mathcal{Q} \overset{\sim}{\to} \mathcal{Q}_{X}$
of $G$-$\mu$-displays over $(A, I)$ for some $X \in G(A)_{I}$.
The isomorphism $\beta$ induces
\[
\Hom_{A/I}(C(\mathcal{Q}), K) \overset{\sim}{\to} \Hom_\O(C(R_{G, \mu}), K)=\mathfrak{t}_{G, \mu} \otimes_k K.
\]
See Example \ref{Example:conormal module} and Remark \ref{Remark:tangent space}.
By Lemma \ref{Lemma:lift of Hodge filtration banal case, conormal module},
we then obtain the following isomorphism
\[
\Lift(P(\mathcal{Q})_{A/I}, \mathscr{Q}_{A'/I'})  \overset{\sim}{\to} \mathfrak{t}_{G, \mu} \otimes_k K.
\]
We define
\[
\Per^{\beta}_{\mathscr{Q}} \colon \Def(\mathcal{Q})_{(A', I')} \to \mathfrak{t}_{G, \mu} \otimes_k K
\]
as the composition of $\Per_{\mathscr{Q}}$ with the above isomorphism.
We call $\Per^{\beta}_{\mathscr{Q}}$ the \textit{normalized period map} with respect to
the isomorphism
$
\beta \colon \mathcal{Q} \overset{\sim}{\to} \mathcal{Q}_{X}.
$

If $\beta' \colon \mathcal{Q} \overset{\sim}{\to} \mathcal{Q}_{X'}$ is another isomorphism, then there exists an $A/I$-linear isomorphism
$h \colon \mathfrak{t}_{G, \mu} \otimes_k K \overset{\sim}{\to} \mathfrak{t}_{G, \mu} \otimes_k K$
such that $\Per^{\beta'}_{\mathscr{Q}}=h \circ \Per^{\beta}_{\mathscr{Q}}$.
\end{rem}

\section{Universal deformations}\label{Section:Universal deformations}

This section is the main part of this paper.
Throughout this section, we will assume that $\mu$ is 1-bounded.
For a $G$-$\mu$-display $\mathcal{Q}$ over $(\O, (\pi))$, we construct a universal deformation of $\mathcal{Q}$ as
a prismatic $G$-$\mu$-display over $R_{G, \mu}$, where
$R_{G, \mu}$
is the local ring defined in Definition \ref{Definition:universal local ring}.
The precise definition of a universal deformation is explained below.
We also give some characterizations of universal deformations.

\subsection{Definitions and main results}\label{Subsection:Definitions and main results}

Recall that $\mathcal{C}_\O$ is the category of complete regular local rings $R$ over $\O$ with residue field $k$.
Let $R \in \mathcal{C}_\O$.
The $\O_E$-prism
$
(\O, (\pi))
$
with the projection $e \colon R \to k$ is an object of the category $(R)_{\Prism, \O_E}$ defined in Section \ref{Subsection:Prismatic sites}.

\begin{defn}\label{Definition:deformation absolute prismatic crysta}
Let $\mathcal{Q}$ be a $G$-$\mu$-display over $(\O, (\pi))$.
Let $R \in \mathcal{C}_\O$.
A \textit{deformation} of $\mathcal{Q}$ over $R$
    is a pair 
    $(\mathfrak{Q}, \gamma)$ of a prismatic $G$-$\mu$-display $\mathfrak{Q}$ over $R$ (Definition \ref{Definition:G displays over prismatic sites}) and an isomorphism
    $
    \gamma \colon \mathfrak{Q}_{(\O, (\pi))} \overset{\sim}{\to} \mathcal{Q}
    $
    of $G$-$\mu$-displays over $(\O, (\pi))$.
    An isomorphism $(\mathfrak{Q}, \gamma) \overset{\sim}{\to} (\mathfrak{Q}', \gamma')$
    of deformations
    is an isomorphism
    $\mathfrak{Q} \overset{\sim}{\to} \mathfrak{Q}'$
    of prismatic $G$-$\mu$-displays over $R$
    such that the induced isomorphism
    $\mathfrak{Q}_{(\O, (\pi))} \overset{\sim}{\to} \mathfrak{Q}'_{(\O, (\pi))}$
    makes
    the following diagram commute:
    \[
    \xymatrix{
    \mathfrak{Q}_{(\O, (\pi))}   \ar[rr]^-{} \ar[rd]_-{\gamma} & & \mathfrak{Q}'_{(\O, (\pi))}  \ar[ld]^-{\gamma'} \\
    & \mathcal{Q}.  &
    }
    \]
    When there is no danger of confusion, we write $\mathfrak{Q}$ instead of $(\mathfrak{Q}, \gamma)$.
\end{defn}

\begin{rem}\label{Remark:base change of deformations over prismatic site}
Let
$(\mathfrak{Q}, \gamma)$
be a deformation of $\mathcal{Q}$ over $R$.
Let
$h \colon R \to R'$
be a morphism in $\mathcal{C}_\O$.
Then the base change $h^*\mathfrak{Q}$ with the isomorphism
$
\gamma \colon (h^*\mathfrak{Q})_{(\O, (\pi))} = \mathfrak{Q}_{(\O, (\pi))} \overset{\sim}{\to} \mathcal{Q}
$
is a deformation of $\mathcal{Q}$ over $R'$.
\end{rem}

\begin{defn}[Universal deformation]\label{Definition:universal deformation}
Let $\mathcal{Q}$ be a $G$-$\mu$-display over $(\O, (\pi))$.
Let $R \in \mathcal{C}_\O$.
    We say that a deformation $(\mathfrak{Q}, \gamma)$ of $\mathcal{Q}$ over $R$ is a \textit{universal deformation} if it satisfies the following property:
\begin{enumerate}
   \item[$(\ast)$] Let $R' \in \mathcal{C}_\O$ and let $(\mathfrak{Q}', \gamma')$ be a deformation of $\mathcal{Q}$ over $R'$.
    Then there exists a unique 
    local homomorphism
    $
    h \colon R \to R'
    $
    over $\O$ such that 
    $(h^*\mathfrak{Q}, \gamma)$ is isomorphic to
    $(\mathfrak{Q}', \gamma')$
    as a deformation of $\mathcal{Q}$ over $R'$.
\end{enumerate}
\end{defn}

In order to state our main results, we need to introduce some more notation.
We consider a map $f \colon (A', I') \to (A, I)$
of orientable and bounded $\O_E$-prisms over $\O$
such that $f \colon A' \to A$ is surjective.
Let $R \in \mathcal{C}_\O$ and
let
$
\mathfrak{Q}
$
be
a prismatic $G$-$\mu$-display over $R$.
Given a homomorphism
$g \colon R \to A/I$ over $\O$,
we can regard $(A, I)$ as an object of $(R)_{\Prism, \O_E}$, and we have the associated $G$-$\mu$-display $\mathfrak{Q}_{g}=\mathfrak{Q}_{(A, I)}$ over $(A, I)$.
Let
\[
\Hom(R, A'/I')_{g}
\]
be the set of homomorphisms $g' \colon R \to A'/I'$ over $\O$ which are lifts of $g$.
For each $g' \in \Hom(R, A'/I')_{g}$,
the $G$-$\mu$-display $\mathfrak{Q}_{g'}$ over $(A', I')$ is equipped with the isomorphism
$
\gamma_f \colon f^*(\mathfrak{Q}_{g'}) \overset{\sim}{\to} \mathfrak{Q}_{g},
$
namely, it is a deformation of $\mathfrak{Q}_{g}$.
In conclusion, we have a map
\[
\ev_\mathfrak{Q} \colon \Hom(R, A'/I')_{g} \to \Def(\mathfrak{Q}_g)_{(A', I')}, \quad g' \mapsto \mathfrak{Q}_{g'}
\]
of sets.
We call $\ev_\mathfrak{Q}$ the \textit{evaluation map}.

\begin{ex}\label{Example:Kodaira--Spencer map}
Let $\mathcal{Q}$ be a $G$-$\mu$-display over $(\O, (\pi))$.
Let $R \in \mathcal{C}_\O$ and
let
$
\mathfrak{Q}
$
be
a deformation of $\mathcal{Q}$ over $R$.
We consider
the nilpotent thickening
\[
f \colon (\O[[t]]/t^2, (\pi)) \to (\O, (\pi))
\]
of Breuil--Kisin type.
By the above construction, we have a map
\[
\ev_\mathfrak{Q} \colon \Hom(R, k[[t]]/t^2)_{e} \to \Def(\mathcal{Q})_{(\O[[t]]/t^2, (\pi))}, \quad g \mapsto \mathfrak{Q}_{g}.
\]
The set $\Hom(R, k[[t]]/t^2)_{e}$ can be naturally identified with the tangent space $\mathfrak{t}_{R}$ of $R$ over $\O$ (cf.\ Remark \ref{Remark:tangent space}).
We write
\[
\KS(\mathfrak{Q}) \colon \mathfrak{t}_{R} \to \Def(\mathcal{Q})_{(\O[[t]]/t^2, (\pi))}
\]
for the above map $\ev_\mathfrak{Q}$.
\end{ex}

\begin{defn}[Kodaira--Spencer map]\label{Definition:versal deformation over prismatic sites}
The map $\KS(\mathfrak{Q})$
defined in Example \ref{Example:Kodaira--Spencer map}
is called the \textit{Kodaira--Spencer map} of $\mathfrak{Q}$.
We say that $\mathfrak{Q}$ is \textit{versal} if the Kodaira--Spencer map
$\KS(\mathfrak{Q})$ is surjective.
\end{defn}

\begin{defn}\label{Definition:property (Perfd) and (BK)}
    Let
    $R \in \mathcal{C}_\O$
    and
    let
    $
    \mathfrak{Q}
    $
    be
    a prismatic $G$-$\mu$-display over $R$.
    We say that $\mathfrak{Q}$ \textit{has the property $(\mathrm{Perfd})$} if it satisfies the following condition:
    \begin{itemize}
    \item[$(\ast)$] 
    Let $(S, a^\flat)$ be a perfectoid pair over $\O$.
    Let $e \colon R \to S/a$ be the composition $R \to k \to S/a$.
    Then
    the evaluation map
    \[
    \ev_\mathfrak{Q} \colon \Hom(R, S/a^{m})_{e} \to \Def(\mathfrak{Q}_e)_{(W_{\O_E}(S^\flat)/[a^\flat]^{m}, I_S)}
    \]
    is bijective for any $m \geq 1$.
    \end{itemize}
    We say that $\mathfrak{Q}$ \textit{has the property $(\mathrm{BK})$} if it satisfies the following condition:
\begin{itemize}
    \item[$(\ast)$] Let $\widetilde{k}$ be a perfect field containing $k$.
    We set $\widetilde{\O}:=W(\widetilde{k}) \otimes_{W(\F_q)} \O_E$.
    Let $(\mathfrak{S}_{\widetilde{\O}}, (\mathcal{E}))$ be an $\O_E$-prism of Breuil--Kisin type over $\widetilde{\O}$, where $\mathfrak{S}_{\widetilde{\O}}=\widetilde{\O}[[t_1, \dotsc, n]]$ for an integer $n \geq 0$.
    We set $\widetilde{R}:=\mathfrak{S}_{\widetilde{\O}}/\mathcal{E}$
    and $\widetilde{R}_m:=\widetilde{R}/\mathfrak{m}^{m}_{\widetilde{R}}=\mathfrak{S}_{\widetilde{\O}, m}/\mathcal{E}$
    for an integer $m \geq 1$.
    Let $e \colon R \to \widetilde{k}$ be the natural homomorphism.
    Then the evaluation map
    \[
    \ev_\mathfrak{Q} \colon \Hom(R, \widetilde{R}_m)_{e} \to \Def(\mathfrak{Q}_e)_{(\mathfrak{S}_{\widetilde{\O}, m}, (\mathcal{E}))}
    \]
    is bijective for any $m \geq 1$.
\end{itemize}
\end{defn}

The goal of this section is to prove the following two results.
The first one is about the existence of universal deformations and their properties.
Recall that $R_{G, \mu}$ is the completion of $\O_{U^{-}_{\mu}, 1}$ with respect to the maximal ideal; see Definition \ref{Definition:universal local ring}.
There exists an isomorphism
$R_{G, \mu} \simeq \O[[t_1, \dotsc, t_r]]$
over $\O$, so that $R_{G, \mu} \in \mathcal{C}_\O$.

\begin{thm}\label{Theorem:Existence of universal deformation}
Assume that $\mu$ is 1-bounded.
For a $G$-$\mu$-display $\mathcal{Q}$ over $(\O, (\pi))$,
there exists a universal deformation $(\mathfrak{Q}^{\univ}, \gamma)$ of $\mathcal{Q}$ over $R_{G, \mu}$.
Moreover
$\mathfrak{Q}^{\univ}$
has the properties $(\mathrm{Perfd})$ and $(\mathrm{BK})$.
In particular $(\mathfrak{Q}^{\univ}, \gamma)$ is versal.
\end{thm}

The second one gives some characterizations of universal deformations.
Let $\O_C$ be the $\pi$-adic completion of the integral closure of $\O$ in an algebraic closure of $\O[1/\pi]$.
Then $\O_C$ is a $\pi$-adically complete valuation ring of rank $1$ with algebraically closed fraction field.
We note that $\O_C$ is a perfectoid ring.
We denote by $\O_{C^\flat}$ the tilt of $\O_C$.
There is an element
$\pi^\flat \in \O_{C^\flat}$ with $\theta([\pi^\flat])=\pi$, so that $(\O_C, \pi^\flat)$ is a perfectoid pair over $\O$.

\begin{thm}\label{Theorem:characterization of universal family}
Assume that $\mu$ is 1-bounded.
Let $\mathcal{Q}$ be a $G$-$\mu$-display over $(\O, (\pi))$.
Let $R \in \mathcal{C}_\O$
and
let
$(\mathfrak{Q}, \gamma)$
be a deformation of $\mathcal{Q}$ over $R$.
Then the following statements are equivalent:
\begin{enumerate}
    \item $(\mathfrak{Q}, \gamma)$ is a universal deformation.
    \item $\mathfrak{Q}$ has the property $(\mathrm{BK})$.
    \item $\mathfrak{Q}$ has the property $(\mathrm{Perfd})$.
    \item 
    Let $e \colon R \to \O_C/\pi$ be the composition $R \to k \to \O_C/\pi$.
    The evaluation map
    \[
    \ev_\mathfrak{Q} \colon \Hom(R, \O_C/\pi^{2})_{e} \to \Def(\mathfrak{Q}_e)_{(W_{\O_E}(\O_{C^\flat})/[\pi^\flat]^{2}, I_{\O_C})}
    \]
    is bijective.
    \item The equalities $\dim R=\dim R_{G, \mu}$ and $\dim_k \mathfrak{t}_R=\dim_k \mathfrak{t}_{G, \mu}$ hold, and $(\mathfrak{Q}, \gamma)$ is versal.
    (Here $\dim R$ is the dimension of the ring $R$ and $\dim_k \mathfrak{t}_R$ is the dimension of the $k$-vector space $\mathfrak{t}_R$.)
\end{enumerate}
\end{thm}

The rest of this section is devoted to the proofs of Theorem \ref{Theorem:Existence of universal deformation} and Theorem \ref{Theorem:characterization of universal family}.
In Section \ref{Subsection:Preliminaries on maps of prisms}, we present a few technical results on maps of $\O_E$-prisms.
In Section \ref{Subsection:Linearity of period maps}, we investigate the relation between $\mathrm{(Perfd)}$, $\mathrm{(BK)}$, and the property of being a universal deformation.
For this, we discuss the \textit{linearity} of evaluation maps, which plays an important role in our discussion.
Finally, in Section \ref{Subsection:Constructions of universal deformations}, we complete the proofs of Theorem \ref{Theorem:Existence of universal deformation} and Theorem \ref{Theorem:characterization of universal family}.

\subsection{Preliminaries on maps of $\O_E$-prisms}\label{Subsection:Preliminaries on maps of prisms}

Let $R \in \mathcal{C}_\O$.
We fix an $\O_E$-prism
$
(\mathfrak{S}_\O, (\mathcal{E}))
$
of Breuil--Kisin type, where $\mathfrak{S}_\O=\O[[t_1, \dotsc, t_n]]$, with an isomorphism
$R \simeq \mathfrak{S}_\O/\mathcal{E}$
over $\O$.
We note that $R$ is of dimension $n$.
(Here $n$ is allowed to be zero.)

Let $(S, a^\flat)$ be a perfectoid pair over $\O$.
Let $m \geq 1$ be an integer.
The following notion of primitive map will play an important role in our proof of Theorem \ref{Theorem:Existence of universal deformation}.

\begin{defn}[{Primitive map}]\label{Definition:primitive map}
A map
\[
f \colon (\mathfrak{S}_\O, (\mathcal{E})) \to (W_{\O_E}(S^\flat)/[a^\flat]^m, I_S)
\]
of $\O_E$-prisms over $\O$
is called \textit{primitive with respect to $a^\flat$} (or simply, primitive) if for every $1 \leq i \leq n$, the map $f$ sends $t_i$ to $[a^\flat b^\flat_i]$ for some element $b^\flat_i \in S^\flat$.
(Here we denote the image of $[a^\flat b^\flat_i] \in W_{\O_E}(S^\flat)$ in $W_{\O_E}(S^\flat)/[a^\flat]^m$ by the same symbol.)
\end{defn}

We note that a primitive map
$(\mathfrak{S}_\O, (\mathcal{E})) \to (W_{\O_E}(S^\flat)/[a^\flat]^m, I_S)$
factors through a map
$(\mathfrak{S}_{\O, m}, (\mathcal{E})) \to (W_{\O_E}(S^\flat)/[a^\flat]^m, I_S)$.


\begin{lem}\label{Lemma:lifting lemma I}
Assume that the map $S \to S$, $x \mapsto x^p$ is surjective.
Then, for a homomorphism
$g \colon R \to S/a^m$ over $\O$ which is a lift of the composition $R \to k \to S/a$,
there exists a primitive map
$
f \colon (\mathfrak{S}_\O, (\mathcal{E})) \to (W_{\O_E}(S^\flat)/[a^\flat]^m, I_S)
$
which induces $g$.
\end{lem}

\begin{proof}
For every $1 \leq i \leq n$, the image of $t_i$ under the composition
$\mathfrak{S}_\O \to R \to S/a^m$
can be written as $a b_i$ for some element $b_i \in S$.
(Again, we abuse notation and denote the image of $a b_i \in S$ in $S/a^m$ by the same symbol.)
Since the map $S \to S$, $x \mapsto x^p$ is surjective, there exists an element $b^\flat_i \in S^\flat$ with $\theta([b^\flat_i])=b_i$ (cf.\ \cite[Lemma 3.3]{BMS}).
We define a homomorphism
$f \colon \mathfrak{S}_\O \to W_{\O_E}(S^\flat)/[a^\flat]^m$
over $\O$
by $t_i \mapsto [a^\flat b^\flat_i]$.
Since $\phi([a^\flat b^\flat_i])=[a^\flat b^\flat_i]^q$,
it follows that $f$ is $\phi$-equivariant, which in turn implies that $f$ is a homomorphism of $\delta_E$-rings since both the target and the source are $\pi$-torsion free (see Lemma \ref{Lemma:pi torsion free of witt vectors}).
By construction, we see that $f$ is a lift of $g$.
In particular $f$ sends $\mathcal{E}$ into the kernel of $W_{\O_E}(S^\flat)/[a^\flat]^m \to S/a^m$.
Thus we have a primitive map
$
f \colon (\mathfrak{S}_\O, (\mathcal{E})) \to (W_{\O_E}(S^\flat)/[a^\flat]^m, I_S)
$
which induces $g$.
\end{proof}

\begin{rem}\label{Remark:Andre lemma}
Any perfectoid ring $S$ admits a $p$-completely faithfully flat homomorphism
$S \to S'$
of perfectoid rings such that the map $S' \to S'$, $x \mapsto x^p$ is surjective.
This is a consequence of Andr\'e's flatness lemma; see \cite[Theorem 7.14]{BS}.
\end{rem}

\begin{lem}\label{Lemma:lifting lemma II}
Let
$
f \colon (\mathfrak{S}_\O, (\mathcal{E})) \to (W_{\O_E}(S^\flat)/[a^\flat]^m, I_S)
$
be a primitive map and let
$g \colon R \to S/a^m$
be the induced homomorphism.
Suppose that we are given a homomorphism
$g' \colon R \to S/a^{m+1}$ over $\O$ which is a lift of $g$.
Then there exists a map
$
f' \colon (\mathfrak{S}_\O, (\mathcal{E})) \to (W_{\O_E}(S^\flat)/[a^\flat]^{m+1}, I_S)
$
which induces $f$ and $g'$.
\end{lem}

\begin{proof}
Since $f$ is primitive, it maps $t_i$ to $[a^\flat b^\flat_i] \in W_{\O_E}(S^\flat)/[a^\flat]^m$ for some element $b^\flat_i \in S^\flat$ for any $i$.
Let $b_i:=\theta([b^\flat_i]) \in S$.
The image of $t_i$ under the composition
$\mathfrak{S}_\O \to R \to S/a^{m+1}$
can be written as
$ab_i+a^m y_i$
for some element $y_i \in S/a^{m+1}$.
Let $x_i \in W_{\O_E}(S^\flat)/[a^\flat]^{m+1}$ be such that $\theta_{\O_E}(x_i)=y_i$ in $S/a^{m+1}$.
We define a homomorphism
$f' \colon \mathfrak{S}_\O \to W_{\O_E}(S^\flat)/[a^\flat]^{m+1}$
over $\O$
by $t_i \mapsto [a^\flat b^\flat_i]+[a^\flat]^m x_i$.
By construction, it is a lift of $g'$.
We claim that $f'$ is a homomorphism of $\delta_E$-rings.
Indeed, we have
\[
\phi([a^\flat b^\flat_i]+[a^\flat]^m x_i)=[a^\flat b^\flat_i]^q + [a^\flat]^{qm} \phi(x_i)= [a^\flat b^\flat_i]^q
\]
in $W_{\O_E}(S^\flat)/[a^\flat]^{m+1}$
since $qm \geq m+1$.
Similarly, we have
$
([a^\flat b^\flat_i]+[a^\flat]^m x_i)^q=[a^\flat b^\flat_i]^q.
$
It follows that $f'$ is $\phi$-equivariant, which in turn implies that it is a homomorphism of $\delta_E$-rings.
The resulting map
$
f' \colon (\mathfrak{S}_\O, (\mathcal{E})) \to (W_{\O_E}(S^\flat)/[a^\flat]^{m+1}, I_S)
$
has the desired properties.
\end{proof}

\begin{lem}\label{Lemma:map to dimension 1 regular local ring}
    Assume that $k$ is algebraically closed.
    Let $m \geq 1$ be an integer and $x \in \mathfrak{m}^m_R/\mathfrak{m}^{m+1}_R$
    a nonzero element.
    Then there exists a complete regular local ring $R'$ of dimension $1$ with a surjection
    $R \to R'$ such that
    the induced homomorphism
    $
    \mathfrak{m}^m_R/\mathfrak{m}^{m+1}_R \to \mathfrak{m}^m_{R'}/\mathfrak{m}^{m+1}_{R'}
    $
    sends $x$ to a nonzero element.
\end{lem}

\begin{proof}
    Let $x_1, \dotsc, x_n \in \mathfrak{m}_R$ be a regular system of parameters.
    The homomorphism of graded $k$-algebras
    \[
    k[x_1, \dotsc, x_n] \to \bigoplus_{l \geq 0} \mathfrak{m}^l_R/\mathfrak{m}^{l+1}_R
    \]
    defined by $x_i \mapsto x_i \in \mathfrak{m}_R/\mathfrak{m}^{2}_R$ is an isomorphism.
    Since $\mathfrak{m}^m_R/\mathfrak{m}^{m+1}_R$ is nonzero, we have $n \geq 1$.
    If $n = 1$, then the assertion holds trivially.
    We assume that $n \geq 2$.

    We claim that, for a nonzero homogeneous polynomial $f(x_1, \dotsc, x_n)$ of degree $m$,
    there exist two distinct integers $i, j \in \{ 1, \dotsc, n \}$ and an element $b \in k$ such that
    the image of $f(x_1, \dotsc, x_n)$ in $k[x_1, \dotsc, x_n]/(x_i-bx_j)$ is nonzero.
    Indeed, since $k$ is infinite, there is a nonzero element $(b_1, \dotsc, b_n) \in k^n$ such that
    $f(b_1, \dotsc, b_n) \neq 0$.
    Without loss of generality, we may assume that $b_{n-1} \neq 0$.
    We set $b:=b_n/b_{n-1}$.
    Then the polynomial
    $f(x_1, \dotsc, x_{n-1}, bx_{n-1}) \in k[x_1, \dotsc, x_{n-1}]$
    is nonzero since its evaluation at $(b_1, \dotsc, b_{n-1})$ is nonzero.
    
    The claim implies that
    there exist two distinct integers $i, j \in \{ 1, \dotsc, n \}$ and an element $b \in R$ such that, for the quotient $R'':=R/(x_i-bx_j)$,
    the induced homomorphism
    $
    \mathfrak{m}^m_R/\mathfrak{m}^{m+1}_R \to \mathfrak{m}^m_{R''}/\mathfrak{m}^{m+1}_{R''}
    $
    sends $x$ to a nonzero element.
    Since $R''$ is regular, the assertion follows by repeating this procedure.
\end{proof}

\begin{cor}\label{Corollary:map from regular local ring to valuation ring}
Let $m \geq 1$ be an integer and $x \in \mathfrak{m}^m_R/\mathfrak{m}^{m+1}_R$
a nonzero element.
Then
there exists a perfectoid pair
$(V, a^\flat)$
over $\O$ with the following properties:
\begin{enumerate}
    \item $V$ is a $\pi$-adically complete valuation ring of rank $1$ with algebraically closed fraction field.
    \item 
    There exists a primitive map
    \[
    (\mathfrak{S}_\O, (\mathcal{E})) \to (W_{\O_E}(V^\flat)/[a^\flat]^{m+1}, I_V)
    \]
    such that the induced homomorphism
    $
    \mathfrak{m}^m_R/\mathfrak{m}^{m+1}_R \to a^mV/a^{m+1}V
    $
    sends $x$ to a nonzero element.
\end{enumerate}
\end{cor}

\begin{proof}
    We may assume that $k$ is algebraically closed.
    By Lemma \ref{Lemma:map to dimension 1 regular local ring},
    there exists a complete regular local ring $R'$ of dimension $1$ with a surjection
    $R \to R'$ such that
    the induced homomorphism
    $
    \mathfrak{m}^m_R/\mathfrak{m}^{m+1}_R \to \mathfrak{m}^m_{R'}/\mathfrak{m}^{m+1}_{R'}
    $
    sends $x$ to a nonzero element.

    Let $a \in \mathfrak{m}_{R'}$ be a uniformizer.
    Let $V$ be the $\pi$-adic completion of the integral closure of $R'$ in an algebraic closure of $R'[1/a]$.
    Then $V$ is a $\pi$-adically complete valuation ring of rank $1$ with algebraically closed fraction field.
    The natural homomorphism
    $g \colon R \to V/a^{m+1}V$ sends $\mathfrak{m}_R$ into $(a)$, and the induced homomorphism
    $
    \mathfrak{m}^m_R/\mathfrak{m}^{m+1}_R \to a^mV/a^{m+1}V
    $
    sends $x$ to a nonzero element.
    There exists an element
    $a^\flat \in V^\flat$ with $\theta([a^\flat])=a$, so that $(V, a^\flat)$ is a perfectoid pair over $\O$.
    By Lemma \ref{Lemma:lifting lemma I}, there exists a primitive map
    $
    (\mathfrak{S}_\O, (\mathcal{E})) \to (W_{\O_E}(V^\flat)/[a^\flat]^{m+1}, I_V)
    $
    which induces $g$.
    This concludes the proof.
\end{proof}

\subsection{Linearity of evaluation maps}\label{Subsection:Linearity of period maps}

Let $\mathcal{Q}$ be a $G$-$\mu$-display over $(\O, (\pi))$.
For deformations of $\mathcal{Q}$, we investigate the relation between $\mathrm{(Perfd)}$, $\mathrm{(BK)}$, and the property of being a universal deformation.

We first note the following consequence of Theorem \ref{Theorem:main result on G displays over complete regular local rings}.

\begin{rem}\label{Remark:equivalence deformation evaluation}
Let $R \in \mathcal{C}_\O$.
We choose an $\O_E$-prism
$(\mathfrak{S}_\O, (\mathcal{E}))$
of Breuil--Kisin type, where $\mathfrak{S}_\O=\O[[t_1, \dotsc, t_n]]$, with an isomorphism
$R \simeq \mathfrak{S}_\O/\mathcal{E}$
over $\O$.
Let
$
f \colon (\mathfrak{S}_\O, (\mathcal{E})) \to (\O, (\pi))
$
be the morphism defined by $t_i \mapsto 0$.
For a deformation $\mathfrak{Q}$ of $\mathcal{Q}$ over $R$,
the value
$\mathfrak{Q}_{(\mathfrak{S}_\O, (\mathcal{E}))}$
is naturally a deformation of $\mathcal{Q}$ over $(\mathfrak{S}_\O, (\mathcal{E}))$.
By Theorem \ref{Theorem:main result on G displays over complete regular local rings}, this construction induces
an equivalence from the groupoid of deformations of $\mathcal{Q}$ over $R$ to that of deformations of $\mathcal{Q}$ over $(\mathfrak{S}_\O, (\mathcal{E}))$.
\end{rem}

\begin{cor}\label{Corollary:isomorphism of deformatiions over prismatic site}
Let $R \in \mathcal{C}_\O$.
There is at most one isomorphism between two deformations of $\mathcal{Q}$ over $R$.
\end{cor}

\begin{proof}
    This follows from Remark \ref{Remark:equivalence deformation evaluation} and Corollary \ref{Corollary:Uniqueness of isomorphisms between deformations:BK and Perfd case}.
\end{proof}

\begin{cor}\label{Corollary:(BK) implies universality}
Let $R \in \mathcal{C}_\O$
and
let
$\mathfrak{Q}$
be a deformation of $\mathcal{Q}$ over $R$.
If $\mathfrak{Q}$ has the property $\mathrm{(BK)}$,
then $\mathfrak{Q}$ is a universal deformation of $\mathcal{Q}$.
\end{cor}

\begin{proof}
    Let $R' \in \mathcal{C}_\O$ and
    let
    $\mathfrak{Q}'$
    be
    a deformation of $\mathcal{Q}$ over $R'$.
    We want to show that
    there exists a unique 
    local homomorphism
    $
    h \colon R \to R'
    $
    over $\O$ such that 
    $h^*\mathfrak{Q}$ is isomorphic to
    $\mathfrak{Q}'$
    as a deformation of $\mathcal{Q}$.
    We fix
    an $\O_E$-prism
    $(\mathfrak{S}_\O, (\mathcal{E}))$
    of Breuil--Kisin type
    with an isomorphism
    $
    R' \simeq \mathfrak{S}_\O/\mathcal{E}
    $
    over $\O$.
    By Example \ref{Example:banal over finite etale covering} and Corollary \ref{Corollary:Uniqueness of isomorphisms between deformations:BK and Perfd case}, we have
    \[
    \Def(\mathcal{Q})_{(\mathfrak{S}_\O, (\mathcal{E}))} \overset{\sim}{\to}
    {\varprojlim}_m \Def(\mathcal{Q})_{(\mathfrak{S}_{\O, m}, (\mathcal{E}))}.
    \]
    Since $\mathfrak{Q}$ has the property $\mathrm{(BK)}$,
    we see that the evaluation map
    \[
    \ev_\mathfrak{Q} \colon \Hom(R, R')_{e} \to \Def(\mathcal{Q})_{(\mathfrak{S}_\O, (\mathcal{E}))}, \quad h \mapsto \mathfrak{Q}_h
    \]
    is bijective.
    (Here $e \colon R \to k$ is the projection.)
    For a local homomorphism $h \colon R \to R'$ over $\O$,
    it follows from Remark \ref{Remark:equivalence deformation evaluation}
    that $h^*\mathfrak{Q} \simeq \mathfrak{Q}'$
    if and only if
    $\mathfrak{Q}_h \simeq \mathfrak{Q}'_{(\mathfrak{S}_\O, (\mathcal{E}))}$.
    Since the above map $\ev_\mathfrak{Q}$ is bijective, such a homomorphism $h$ exists and is unique.
\end{proof}

\begin{rem}\label{Remark:BK implies versal}
    If
    a deformation
    $
    \mathfrak{Q}
    $
    of
    $\mathcal{Q}$
    over $R$
    for some $R \in \mathcal{C}_\O$
    satisfies the property $(\mathrm{BK})$,
    then the Kodaira--Spencer map $\KS(\mathfrak{Q})$ is bijective, so that $\mathfrak{Q}$ is versal.
\end{rem}

\begin{rem}\label{Remark:versality implies universality}
Assume that there exists a deformation
$\mathfrak{Q}$
of
$\mathcal{Q}$ over $R$ for some $R \in \mathcal{C}_\O$ such that $\mathfrak{Q}$ has the property $\mathrm{(BK)}$.
It follows from Corollary \ref{Corollary:(BK) implies universality} that $\mathfrak{Q}$ is a universal deformation.
Let $\mathfrak{Q}'$ be another deformation of $\mathcal{Q}$ over $R'$ for some $R' \in \mathcal{C}_\O$.
There exists a unique local homomorphism
$
h \colon R \to R'
$
over $\O$ such that
$h^*\mathfrak{Q} \simeq \mathfrak{Q}'$ by the universal property.
For
the homomorphism
$
\mathfrak{t}_{R'} \to \mathfrak{t}_{R}
$
induced from $h$, the following diagram commutes:
    \[
    \xymatrix{
    \mathfrak{t}_{R'}   \ar[rr]^-{} \ar[rd]_-{\KS(\mathfrak{Q}')} & &  \mathfrak{t}_{R}   \ar[ld]^-{\KS(\mathfrak{Q})} \\
    & \Def(\mathcal{Q})_{(\O[[t]]/t^2, (\pi))}.  &
    }
    \]
Since $\KS(\mathfrak{Q})$ is bijective, we see that
$\mathfrak{Q}'$ is versal if and only if $\mathfrak{t}_{R'} \to \mathfrak{t}_{R}$ is surjective.
\end{rem}

To proceed further, we need to introduce two more conditions (Perfd-lin) and (BK-lin).

\begin{rem}\label{Remark:differential map}
We retain the notation from Definition \ref{Definition:property (Perfd) and (BK)}
with $R=R_{G, \mu}$.
Let $m \geq 1$ be an integer.
Let $g \in \Hom(R_{G, \mu}, S/a^{m})_{e}$.
For any $h \in \Hom(R_{G, \mu}, S/a^{m+1})_{g}$, the map
\[
\mathrm{Diff}_{h} \colon \Hom(R_{G, \mu}, S/a^{m+1})_{g} \to \Hom_k(\mathfrak{t}^\vee_{G, \mu}, a^{m}S/a^{m+1}S)=\mathfrak{t}_{G, \mu} \otimes_k (a^{m}S/a^{m+1}S)
\]
defined by $g' \mapsto g'-h$ is bijective.

Similarly, let
$g \in \Hom(R_{G, \mu}, \widetilde{R}_m)_{e}$.
For any $h \in \Hom(R_{G, \mu}, \widetilde{R}_{m+1})_{g}$, we have the following bijection:
\[
\mathrm{Diff}_{h} \colon \Hom(R_{G, \mu}, \widetilde{R}_{m+1})_{g} \overset{\sim}{\to} \mathfrak{t}_{G, \mu} \otimes_k (\mathfrak{m}^{m}_{\widetilde{R}}/\mathfrak{m}^{m+1}_{\widetilde{R}}), \quad g' \mapsto g'-h.
\]
\end{rem}

\begin{defn}\label{Definition:properties (Perfd-lin) and (BK-lin)}
    We retain the notation of Definition \ref{Definition:property (Perfd) and (BK)}.
    Let
    $
    \mathfrak{Q}
    $
    be a prismatic $G$-$\mu$-display over $R_{G, \mu}$.
    We say that $\mathfrak{Q}$ \textit{has the property} (Perfd-lin) if it satisfies the following condition:
    \begin{itemize}
    \item[$(\ast)$] 
    Let $(S, a^\flat)$ be a perfectoid pair over $\O$ such that $\mathfrak{Q}_e$ is a banal $G$-$\mu$-display over $(W_{\O_E}(S^\flat)/[a^\flat], I_S)$.
    Let $m \geq 1$ be an integer
    and let $g \in \Hom(R_{G, \mu}, S/a^{m})_{e}$.
    Then, for any $h \in \Hom(R_{G, \mu}, S/a^{m+1})_{g}$ and any isomorphism
    $\beta \colon \mathfrak{Q}_g \overset{\sim}{\to} \mathcal{Q}_X$, the following composition
    \begin{equation*}
    \begin{split}
    \mathfrak{t}_{G, \mu} \otimes_k &(a^{m}S/a^{m+1}S) \overset{\mathrm{Diff}^{-1}_{h}}{\longrightarrow} \Hom(R_{G, \mu}, S/a^{m+1})_{g} \\
    &\overset{\ev_\mathfrak{Q}}{\longrightarrow} \Def(\mathfrak{Q}_g)_{(W_{\O_E}(S^\flat)/[a^\flat]^{m+1}, I_S)} \overset{\Per^{\beta}_{\mathfrak{Q}_h}}{\longrightarrow} \mathfrak{t}_{G, \mu} \otimes_k (a^{m}S/a^{m+1}S)
    \end{split}
\end{equation*}
    is an \textit{$S/a$-linear isomorphism}.
    See Remark \ref{Remark:normalized period map} for the normalized period map $\Per^{\beta}_{\mathfrak{Q}_h}$, which is bijective by Theorem \ref{Theorem:GM deformation}.
    The above composition is denoted by
    $
    \mathrm{L}^{\beta}_{\mathfrak{Q}, h}.
    $
    \end{itemize}
    We say that $\mathfrak{Q}$ \textit{has the property} (BK-lin) if it satisfies the following condition:
\begin{itemize}
    \item[$(\ast)$] Let $\widetilde{k}$ be a perfect field containing $k$ such that
    $\mathfrak{Q}_{(\widetilde{\O}, (\pi))}$ is banal.
    Let $(\mathfrak{S}_{\widetilde{\O}}, (\mathcal{E}))$ be an $\O_E$-prism of Breuil--Kisin type over $\widetilde{\O}$.
    Let $m \geq 1$ be an integer and
    let $g \in \Hom(R_{G, \mu}, \widetilde{R}_m)_{e}$.
    Then, for any $h \in \Hom(R_{G, \mu}, \widetilde{R}_{m+1})_{g}$ and any isomorphism
    $\beta \colon \mathfrak{Q}_g \overset{\sim}{\to} \mathcal{Q}_X$, the following composition
    \begin{equation*}
    \begin{split}
    \mathfrak{t}_{G, \mu} \otimes_k &(\mathfrak{m}^{m}_{\widetilde{R}}/\mathfrak{m}^{m+1}_{\widetilde{R}}) \overset{\mathrm{Diff}^{-1}_{h}}{\longrightarrow} \Hom(R_{G, \mu}, \widetilde{R}_{m+1})_{g} \\
    &\overset{\ev_\mathfrak{Q}}{\longrightarrow} \Def(\mathfrak{Q}_g)_{(\mathfrak{S}_{\widetilde{\O}, m+1}, (\mathcal{E}))} \overset{\Per^{\beta}_{\mathfrak{Q}_h}}{\longrightarrow} \mathfrak{t}_{G, \mu} \otimes_k (\mathfrak{m}^{m}_{\widetilde{R}}/\mathfrak{m}^{m+1}_{\widetilde{R}})
    \end{split}
\end{equation*}
    is a \textit{$\widetilde{k}$-linear isomorphism}.
    This composition is denoted by
    $
    \mathrm{L}^{\beta}_{\mathfrak{Q}, h}.
    $
\end{itemize}
\end{defn}

\begin{rem}\label{Remark: flat descent for (Perfd-lin)}
    Let the notation be as in Definition \ref{Definition:properties (Perfd-lin) and (BK-lin)}.
    The property (Perfd-lin) can be checked $\pi$-completely flat locally on $S$.
    More precisely,
    let $S \to S'$ be a $\pi$-completely faithfully flat homomorphism of perfectoid rings.
    Let $S''$ be the $\pi$-adic completion of $S' \otimes_S S'$, which is a perfectoid ring (see \cite[Proposition 2.1.11]{CS}).
    If the condition in Definition \ref{Definition:properties (Perfd-lin) and (BK-lin)} is satisfied for $(S', a^\flat)$ and $(S'', a^\flat)$
    (cf.\ Lemma \ref{Lemma:perfectoid ring etale morphism:torsion case} (1)), then it is also satisfied for $(S, a^\flat)$.
    Indeed, since $\pi=0$ in $S/a$, we see that
    $S/a \to S'/a$ is faithfully flat and $S''/a = S'/a \otimes_{S/a} S'/a$.
    Therefore the sequence
    \[
    \mathfrak{t}_{G, \mu} \otimes_k (a^{m}S/a^{m+1}S) \to \mathfrak{t}_{G, \mu} \otimes_k (a^{m}S'/a^{m+1}S') \rightrightarrows \mathfrak{t}_{G, \mu} \otimes_k (a^{m}S''/a^{m+1}S'')
    \]
    is exact by flat descent, from which our claim follows.
    (The same holds if we replace $S''$ by a perfectoid ring which is $\pi$-completely faithfully flat over the $\pi$-adic completion of $S' \otimes_S S'$.)
\end{rem}

\begin{lem}\label{Lemma:(BK-lin) implies (BK), (Perfd-lin) implies (Perfd)}
    Let
    $
    \mathfrak{Q}
    $
    be a prismatic $G$-$\mu$-display over $R_{G, \mu}$.
    \begin{enumerate}
        \item
        Assume that
        $\mathfrak{Q}$
        has the property $\mathrm{(Perfd\mathchar`-lin)}$.
        Then $\mathfrak{Q}$ has the property $\mathrm{(Perfd)}$.
        \item
        Assume that
        $\mathfrak{Q}$
        has the property $\mathrm{(BK\mathchar`-lin)}$.
        Then $\mathfrak{Q}$ has the property $\mathrm{(BK)}$.
    \end{enumerate}
\end{lem}

\begin{proof}
    (1) Let $(S, a^\flat)$ be a perfectoid pair over $\O$.
    We want to prove that the map
    \[
    \ev_\mathfrak{Q} \colon \Hom(R_{G, \mu}, S/a^{m})_{e} \to \Def(\mathfrak{Q}_e)_{(W_{\O_E}(S^\flat)/[a^\flat]^{m}, I_S)}
    \]
    is bijective for any $m \geq 1$.
    It suffices to prove that for any $m \geq 1$ and any $g \in \Hom(R_{G, \mu}, S/a^{m})_{e}$,
    the map
    \[
    \ev_\mathfrak{Q} \colon \Hom(R_{G, \mu}, S/a^{m+1})_{g} \to \Def(\mathfrak{Q}_g)_{(W_{\O_E}(S^\flat)/[a^\flat]^{m+1}, I_S)}
    \]
    is bijective.
    By Corollary \ref{Corollary:deformations form a sheaf},
    we may assume that $\mathfrak{Q}_e$ is banal.
    Then the assertion follows from the property $\mathrm{(Perfd\mathchar`-lin)}$ and Theorem \ref{Theorem:GM deformation}.

    (2) This follows by the same argument as in (1).
\end{proof}

The following proposition is the main reason for introducing (Perfd-lin) and (BK-lin).

\begin{prop}\label{Proposition:(Perfd-lin) implies (BK-lin)}
Let
$
\mathfrak{Q}
$
be a prismatic $G$-$\mu$-display over $R_{G, \mu}$.
Assume that
$\mathfrak{Q}$
has the property $\mathrm{(Perfd\mathchar`-lin)}$.
Then $\mathfrak{Q}$ has the property $\mathrm{(BK\mathchar`-lin)}$.
\end{prop}

\begin{proof}
    With the notation of Definition \ref{Definition:properties (Perfd-lin) and (BK-lin)},
    we want to prove that for any homomorphism $h \in \Hom(R_{G, \mu}, \widetilde{R}_{m+1})_{g}$ and any isomorphism
    $\beta \colon \mathfrak{Q}_g \overset{\sim}{\to} \mathcal{Q}_X$,
    the map
    \[
    \mathrm{L}^{\beta}_{\mathfrak{Q}, h} \colon \mathfrak{t}_{G, \mu} \otimes_k (\mathfrak{m}^{m}_{\widetilde{R}}/\mathfrak{m}^{m+1}_{\widetilde{R}}) \to \mathfrak{t}_{G, \mu} \otimes_k (\mathfrak{m}^{m}_{\widetilde{R}}/\mathfrak{m}^{m+1}_{\widetilde{R}})
    \]
    is a $\widetilde{k}$-linear isomorphism.
    We first prove that $\mathrm{L}^{\beta}_{\mathfrak{Q}, h}$ is a $\widetilde{k}$-linear homomorphism.
    For two elements $g_1, g_2 \in \mathfrak{t}_{G, \mu} \otimes_k (\mathfrak{m}^{m}_{\widetilde{R}}/\mathfrak{m}^{m+1}_{\widetilde{R}})$,
    we claim that
    \[
    \mathrm{L}^{\beta}_{\mathfrak{Q}, h}(g_1+g_2)=\mathrm{L}^{\beta}_{\mathfrak{Q}, h}(g_1)+\mathrm{L}^{\beta}_{\mathfrak{Q}, h}(g_2).
    \]
    Indeed, by virtue of Corollary \ref{Corollary:map from regular local ring to valuation ring},
    it suffices to prove that for any 
    perfectoid pair
    $(V, a^\flat)$ over $\widetilde{\O}$ and any map
    $
    f \colon (\mathfrak{S}_{\widetilde{\O}}, (\mathcal{E})) \to (W_{\O_E}(V^\flat)/[a^\flat]^{m+1}, I_V)
    $
    over $\widetilde{\O}$,
    the induced homomorphism
    \[
    \mathfrak{t}_{G, \mu} \otimes_k (\mathfrak{m}^{m}_{\widetilde{R}}/\mathfrak{m}^{m+1}_{\widetilde{R}}) \to
    \mathfrak{t}_{G, \mu} \otimes_k (a^mV/a^{m+1}V)
    \]
    sends the element
    $
    \mathrm{L}^{\beta}_{\mathfrak{Q}, h}(g_1+g_2)-\mathrm{L}^{\beta}_{\mathfrak{Q}, h}(g_1)-\mathrm{L}^{\beta}_{\mathfrak{Q}, h}(g_2)
    $
    to $0$.
    Let $h' \colon R_{G, \mu} \to V/a^{m+1}$
    be the composition of $h$ with
    the induced homomorphism $\widetilde{R}_{m+1} \to V/a^{m+1}$, and let $\beta'$ be the base change of $\beta$ along
    the map
    $
    (\mathfrak{S}_{\widetilde{\O}, m}, (\mathcal{E})) \to (W_{\O_E}(V^\flat)/[a^\flat]^{m}, I_V)
    $
    induced from $f$.
    The following diagram is commutative:
    \[
    \xymatrix{
    \mathfrak{t}_{G, \mu} \otimes_k (\mathfrak{m}^{m}_{\widetilde{R}}/\mathfrak{m}^{m+1}_{\widetilde{R}}) \ar^-{\mathrm{L}^{\beta}_{\mathfrak{Q}, h}}[r]  \ar[d]_-{}  &  \mathfrak{t}_{G, \mu} \otimes_k (\mathfrak{m}^{m}_{\widetilde{R}}/\mathfrak{m}^{m+1}_{\widetilde{R}})  \ar[d]_-{} \\
    \mathfrak{t}_{G, \mu} \otimes_k (a^mV/a^{m+1}V) \ar[r]^-{\mathrm{L}^{\beta'}_{\mathfrak{Q}, h'}} & \mathfrak{t}_{G, \mu} \otimes_k (a^mV/a^{m+1}V).
    }
    \]
    Since $\mathrm{L}^{\beta'}_{\mathfrak{Q}, h'}$ is a $V/a$-linear homomorphism by our assumption, the claim follows.
    We can prove that
    $
    \mathrm{L}^{\beta}_{\mathfrak{Q}, h}(bg)=b\mathrm{L}^{\beta}_{\mathfrak{Q}, h}(g)
    $
    for any $b \in \widetilde{k}$ and any $g \in \mathfrak{t}_{G, \mu} \otimes_k (\mathfrak{m}^{m}_R/\mathfrak{m}^{m+1}_R)$ in the same way.
    Thus we conclude that $\mathrm{L}^{\beta}_{\mathfrak{Q}, h}$ is a $\widetilde{k}$-linear homomorphism.

    It remains to prove that $\mathrm{L}^{\beta}_{\mathfrak{Q}, h}$ is bijective.
    Since we have shown that $\mathrm{L}^{\beta}_{\mathfrak{Q}, h}$ is a $\widetilde{k}$-linear endomorphism of the finite dimensional $\widetilde{k}$-vector space
    $\mathfrak{t}_{G, \mu} \otimes_k (\mathfrak{m}^{m}_{\widetilde{R}}/\mathfrak{m}^{m+1}_{\widetilde{R}})$, it suffices to prove that its kernel is zero.
    This follows from the same argument as above (since $\mathrm{L}^{\beta'}_{\mathfrak{Q}, h'}$ is injective by our assumption).
\end{proof}

\begin{rem}\label{Remark:implications}
Let $\mathfrak{Q}$ be a deformation of $\mathcal{Q}$ over $R_{G, \mu}$.
In summary, we have the following implications for $\mathfrak{Q}$:
\[
\xymatrix{
\mathrm{(Perfd\mathchar`-lin)} \ar@{=>}[r]  \ar@{=>}[d] &  \mathrm{(BK\mathchar`-lin)}  \ar@{=>}[d] & \\
\mathrm{(Perfd)}  & \mathrm{(BK)} \ar@{=>}[d] \ar@{=>}[r] & \mathrm{(versal)} \\
& \mathrm{(universal \, \, deformation)}. &
}
\]
We will prove in Section \ref{Subsection:Constructions of universal deformations} below that 
there exists a deformation of $\mathcal{Q}$ over $R_{G, \mu}$ with the property $\mathrm{(Perfd\mathchar`-lin)}$.
This in turn implies that all these conditions are equivalent for deformations of $\mathcal{Q}$ over $R_{G, \mu}$.
\end{rem}

\subsection{Constructions of universal deformations}\label{Subsection:Constructions of universal deformations}

Since $\mu$ is 1-bounded, there is an isomorphism
\[
U^{-}_{\mu} \simeq \G^r_a=\Spec \O[t_1, \dotsc, t_r]
\]
of group schemes over $\O$; see \cite[Lemma 4.2.6]{Ito-K23}.
We fix such an isomorphism.
This induces an identification
$
R_{G, \mu} \simeq \O[[t_1, \dotsc, t_r]]
$
over $\O$.
We set
\[
\mathfrak{S}^{\univ}_\O:=\O[[t_1, \dotsc, t_{r+1}]] \quad \text{and} \quad \mathcal{E}:=\pi-t_{r+1},
\]
so that $(\mathfrak{S}^{\univ}_\O, (\mathcal{E}))$ is
an $\O_E$-prism of Breuil--Kisin type over $\O$.
We obtain the following isomorphism defined by $t_{r+1} \mapsto \pi$
\[
\mathfrak{S}^{\univ}_\O/\mathcal{E} \simeq  R_{G, \mu}.
\]
For a $G$-$\mu$-display $\mathcal{Q}$ over $(\O, (\pi))$,
we will construct a deformation of $\mathcal{Q}$ over $(\mathfrak{S}^{\univ}_\O, (\mathcal{E}))$ such that the corresponding deformation 
over $R_{G, \mu}$
has the property $\mathrm{(Perfd\mathchar`-lin)}$.

We need a lemma which allows us to reduce the problem to the case where $\mathcal{Q}$ is banal.
Let $\widetilde{k}$ be a finite Galois extension of $k$ and we set $\widetilde{\O}:=W(\widetilde{k}) \otimes_{W(\F_q)} \O_E$.
Let $\widetilde{\mu} \colon \G_m \to G_{\widetilde{\O}}$ be the base change of $\mu$.
We define
$R_{G, \widetilde{\mu}} \in \mathcal{C}_{\widetilde{\O}}$
in the same way as $R_{G, \mu}$; see Definition \ref{Definition:universal local ring}.
We have
$
R_{G, \mu}\otimes_\O \widetilde{\O} \overset{\sim}{\to} R_{G, \widetilde{\mu}}.
$

\begin{lem}\label{Lemma:reduction to banal case}
    Let $\mathcal{Q}$ be a $G$-$\mu$-display over $(\O, (\pi))$.
    Assume that the base change $\widetilde{\mathcal{Q}}$ of $\mathcal{Q}$ to $(\widetilde{\O}, (\pi))$ is banal.
    If there exists a deformation
    $\widetilde{\mathfrak{Q}}$ of
    $\widetilde{\mathcal{Q}}$ over $R_{G, \widetilde{\mu}}$ having the property $\mathrm{(Perfd\mathchar`-lin)}$,
    then
    the same holds for $\mathcal{Q}$, i.e.\
    there exists a deformation
    of $\mathcal{Q}$
    over $R_{G, \mu}$
    having the property $\mathrm{(Perfd\mathchar`-lin)}$.
\end{lem}

\begin{proof}
    The proof is divided into four parts.
    
    \textit{Step 1.}
    Let $\Gal(\widetilde{k}/k)$ be the Galois group of $\widetilde{k}$ over $k$.
    For an $\O$-algebra $A$,
    each $s \in \Gal(\widetilde{k}/k)$ induces an automorphism of
    $A \otimes_{\O} \widetilde{\O}$ over $\O$, which we denote by the same symbol $s$.
    In particular, since $R_{G, \widetilde{\mu}}=R_{G, \mu}\otimes_\O \widetilde{\O}$,
    we have an automorphism
    $
    s \colon R_{G, \widetilde{\mu}} \overset{\sim}{\to} R_{G, \widetilde{\mu}}
    $
    over $\O$.
    Since $s^*(\widetilde{\mathcal{Q}})=\widetilde{\mathcal{Q}}$,
    the base change
    $s^*(\widetilde{\mathfrak{Q}})$ is a deformation of
    $\widetilde{\mathcal{Q}}$ over $R_{G, \widetilde{\mu}}$.
    By Remark \ref{Remark:implications}, we see that
    $\widetilde{\mathfrak{Q}}$, and hence $s^*(\widetilde{\mathfrak{Q}})$, is a universal deformation of $\widetilde{\mathcal{Q}}$.
    Thus there exists a unique isomorphism
    $
    h_s \colon R_{G, \widetilde{\mu}} \overset{\sim}{\to} R_{G, \widetilde{\mu}}
    $
    over $\widetilde{\O}$ such that
    $
    h^*_s(\widetilde{\mathfrak{Q}}) \simeq s^*(\widetilde{\mathfrak{Q}})
    $
    as deformations of $\widetilde{\mathcal{Q}}$.
    We have
    \begin{equation}\label{equation:1-cocycle}
        h_{ss'}=s(h_{s'}) \circ h_s
    \end{equation}
    for all $s, s' \in \Gal(\widetilde{k}/k)$.
    Here $s(h_{s'}) \colon R_{G, \widetilde{\mu}} \to R_{G, \widetilde{\mu}}$ is the base change of $h_{s'}$ along $s$.

    \textit{Step 2.}
    Since
    \[
    R_{G, \mu}/(\mathfrak{m}^2_{G, \mu}+\pi R_{G, \mu}) = k[t_1, \dotsc, t_r]/(t_1, \dotsc, t_r)^2,
    \]
    the automorphisms $h_s$ induce automorphisms
    \[
    h'_s \colon \widetilde{k}[t_1, \dotsc, t_r]/(t_1, \dotsc, t_r)^2 \overset{\sim}{\to} \widetilde{k}[t_1, \dotsc, t_r]/(t_1, \dotsc, t_r)^2
    \]
    over $\widetilde{k}$ satisfying the same relation as (\ref{equation:1-cocycle}).
    We note that the set of automorphisms of $\widetilde{k}[t_1, \dotsc, t_r]/(t_1, \dotsc, t_r)^2$ over $\widetilde{k}$ can be identified with 
    $\GL_r(\widetilde{k})$.
    Since the Galois cohomology $H^1(\Gal(\widetilde{k}/k), \GL_r(\widetilde{k}))$ has only one element,
    there exists an automorphism $g'$ of 
    $\widetilde{k}[t_1, \dotsc, t_r]/(t_1, \dotsc, t_r)^2$ over $\widetilde{k}$
    such that
    $h'_s=s(g'^{-1}) \circ g'$
    for any
    $s \in \Gal(\widetilde{k}/k)$.
    Let $g$ be an automorphism of $R_{G, \widetilde{\mu}}$ over $\widetilde{\O}$ lifting $g'$.
    By replacing $\widetilde{\mathfrak{Q}}$ by $g^*(\widetilde{\mathfrak{Q}})$,
    we may assume that
    the automorphisms $h_s$ induce the identity on
    $\widetilde{k}[t_1, \dotsc, t_r]/(t_1, \dotsc, t_r)^2$.

    \textit{Step 3.}
    The $\O_E$-prism
    $
    (\widetilde{\O}[[t_1, \dotsc, t_r]]/(t_1, \dotsc, t_r)^2, (\pi))
    $
    is naturally an object of the category $(R_{G, \widetilde{\mu}})_{\Prism, \O_E}$.
    Since $h_s$ is a lift of the identity of
    $\widetilde{k}[t_1, \dotsc, t_r]/(t_1, \dotsc, t_r)^2$,
    we have
    \[
    (h^*_s(\widetilde{\mathfrak{Q}}))_{(\widetilde{\O}[[t_1, \dotsc, t_r]]/(t_1, \dotsc, t_r)^2, (\pi))}=\widetilde{\mathfrak{Q}}_{(\widetilde{\O}[[t_1, \dotsc, t_r]]/(t_1, \dotsc, t_r)^2, (\pi))};
    \]
    see also Remark \ref{Remark:alternative definition of G displays over prismatic sites}.
    On the other hand, we have
    \[
    (s^*(\widetilde{\mathfrak{Q}}))_{(\widetilde{\O}[[t_1, \dotsc, t_r]]/(t_1, \dotsc, t_r)^2, (\pi))}=s^*(\widetilde{\mathfrak{Q}}_{(\widetilde{\O}[[t_1, \dotsc, t_r]]/(t_1, \dotsc, t_r)^2, (\pi))}).
    \]
    Thus the isomorphism
    $h^*_s(\widetilde{\mathfrak{Q}}) \simeq s^*(\widetilde{\mathfrak{Q}})$
    induces
    an isomorphism
    \[
    \widetilde{\mathfrak{Q}}_{(\widetilde{\O}[[t_1, \dotsc, t_r]]/(t_1, \dotsc, t_r)^2, (\pi))} \simeq s^*(\widetilde{\mathfrak{Q}}_{(\widetilde{\O}[[t_1, \dotsc, t_r]]/(t_1, \dotsc, t_r)^2, (\pi))})
    \]
    of deformations of $\widetilde{\mathcal{Q}}$.
    Then, by Galois descent (see Corollary \ref{Corollary:deformations form a sheaf}),
    we see that
    \[
    \widetilde{\mathfrak{Q}}_{(\widetilde{\O}[[t_1, \dotsc, t_r]]/(t_1, \dotsc, t_r)^2, (\pi))}
    \]
    arises from a deformation $\mathscr{Q}'$ of $\mathcal{Q}$ over $(\O[[t_1, \dotsc, t_r]]/(t_1, \dotsc, t_r)^2, (\pi))$.
    
    The homomorphism
    $\mathfrak{S}^{\univ}_\O \to \O[[t_1, \dotsc, t_r]]/(t_1, \dotsc, t_r)^2$, defined by $t_i \mapsto t_i$ ($1 \leq i \leq r$) and $t_{r+1} \mapsto 0$,
    induces a morphism
    \[
    (\mathfrak{S}^{\univ}_\O, (\mathcal{E})) \to (\O[[t_1, \dotsc, t_r]]/(t_1, \dotsc, t_r)^2, (\pi))
    \]
    in $(R_{G, \mu})_{\Prism, \O_E}$.
    By Proposition \ref{Proposition:limit of G displays} and Proposition \ref{Proposition:existence of deformations for non-banal cases},
    there exists a deformation $\mathscr{Q}$ of
    $\mathscr{Q}'$
    over $(\mathfrak{S}^{\univ}_\O, (\mathcal{E}))$.

    \textit{Step 4.}
    The deformation
    $\mathfrak{Q}$
    of
    $\mathcal{Q}$
    over $R_{G, \mu}$ corresponding to $\mathscr{Q}$ has the property $\mathrm{(Perfd\mathchar`-lin)}$.
    Indeed, let $\iota \colon R_{G, \mu} \to R_{G, \widetilde{\mu}}$ be the inclusion.
    It is enough to show that $\iota^*\mathfrak{Q}$ has the property $\mathrm{(Perfd\mathchar`-lin)}$.
    (To check that the condition in Definition \ref{Definition:properties (Perfd-lin) and (BK-lin)} is satisfied for a perfectoid pair $(S, a^\flat)$ over $\O$, it suffices to check it for the induced perfectoid pair $(S \otimes_{\O} \widetilde{\O}, a^\flat)$; see Remark \ref{Remark: flat descent for (Perfd-lin)}.)
    Let $h \colon R_{G, \widetilde{\mu}} \to R_{G, \widetilde{\mu}}$
    be the unique local homomorphism over $\widetilde{\O}$ such that $h^*(\widetilde{\mathfrak{Q}}) \simeq \iota^*\mathfrak{Q}$
    as deformations of $\widetilde{\mathcal{Q}}$.
    By construction, we have
    \[
    (\iota^*\mathfrak{Q})_{(\widetilde{\O}[[t_1, \dotsc, t_r]]/(t_1, \dotsc, t_r)^2, (\pi))} \simeq \widetilde{\mathfrak{Q}}_{(\widetilde{\O}[[t_1, \dotsc, t_r]]/(t_1, \dotsc, t_r)^2, (\pi))}.
    \]
    Since $\widetilde{\mathfrak{Q}}$ has the property 
    $\mathrm{(BK\mathchar`-lin)}$ by Proposition \ref{Proposition:(Perfd-lin) implies (BK-lin)},
    it follows that $h$ induces the identity on $\widetilde{k}[t_1, \dotsc, t_r]/(t_1, \dotsc, t_r)^2$.
    In particular $h$ is an isomorphism, which in turn implies that $\iota^*\mathfrak{Q}$ has the property $\mathrm{(Perfd\mathchar`-lin)}$.
    This concludes the proof.
\end{proof}

\begin{thm}\label{Theorem:existence of deformations with (Perfd-lin)}
    Let $\mathcal{Q}$ be a $G$-$\mu$-display over $(\O, (\pi))$.
    There exists a
    deformation $\mathfrak{Q}^{\univ}$ of $\mathcal{Q}$ over $R_{G, \mu}$
    having the property $\mathrm{(Perfd\mathchar`-lin)}$.
    In particular $\mathfrak{Q}^{\univ}$ is a universal deformation of $\mathcal{Q}$.
\end{thm}

\begin{proof}
    By Lemma \ref{Lemma:reduction to banal case},
    we may assume that
    $\mathcal{Q}=\mathcal{Q}_X$
    for some $X \in G(\O)_{(\pi)}$.
    Let
    $
    U \in G(R_{G, \mu})
    $
    be the $R_{G, \mu}$-valued point
    corresponding to the composition
    $
    \Spec R_{G, \mu} \to U^{-}_{\mu} \to G_\O.
    $
    We define
    \[
    U^{\univ} \in G(\mathfrak{S}^{\univ}_\O)
    \]
    as the image of $U \in G(R_{G, \mu})$ under the homomorphism
    $G(R_{G, \mu}) \to G(\mathfrak{S}^{\univ}_\O)$ induced by $t_i \mapsto t_i$ ($1 \leq i \leq r$).
    Let $X^{\univ} \in G(\mathfrak{S}^{\univ}_\O)_{(\mathcal{E})}$ be the element such that
    \[
    (X^{\univ})_{\mathcal{E}}=(U^{\univ})^{-1}X_\pi \in G(\mathfrak{S}^{\univ}_\O).
    \]
    (Here the image of $X_\pi \in G(\O)$ in $G(\mathfrak{S}^{\univ}_\O)$ is denoted by the same symbol.
    See also (\ref{equation:d-component}) for $(X^{\univ})_{\mathcal{E}}$ and $X_\pi$.)
    Then the $G$-$\mu$-display $\mathcal{Q}_{X^{\univ}}$ over
    $(\mathfrak{S}^{\univ}_\O, (\mathcal{E}))$
    is naturally a deformation of $\mathcal{Q}_X$.
    Let
    $
    \mathfrak{Q}^{\univ}
    $
    be the deformation of $\mathcal{Q}_X$ over $R_{G, \mu}$
    such that
    \[
    (\mathfrak{Q}^{\univ})_{(\mathfrak{S}^{\univ}_\O, (\mathcal{E}))} \simeq \mathcal{Q}_{X^{\univ}}.
    \]

    We shall prove that $\mathfrak{Q}^{\univ}$ has the property $\mathrm{(Perfd\mathchar`-lin)}$.
    In other words, with the notation of Definition \ref{Definition:properties (Perfd-lin) and (BK-lin)},
    we want to show that
    for any $\widetilde{g} \in \Hom(R_{G, \mu}, S/a^{m+1})_{g}$ and any isomorphism
    $\beta \colon \mathfrak{Q}^{\univ}_g \overset{\sim}{\to} \mathcal{Q}_Y$,
    the map
    \[
    \mathrm{L}^{\beta}_{\mathfrak{Q}^{\univ}, \widetilde{g}} \colon \mathfrak{t}_{G, \mu} \otimes_k (a^{m}S/a^{m+1}S) \to \mathfrak{t}_{G, \mu} \otimes_k (a^{m}S/a^{m+1}S)
    \]
    is an $S/a$-linear isomorphism.
    By Remark \ref{Remark:Andre lemma} and Remark \ref{Remark: flat descent for (Perfd-lin)}, 
    we are reduced to the case where the map $S \to S$, $x \mapsto x^p$ is surjective.
    By Lemma \ref{Lemma:lifting lemma I}, the homomorphism $\widetilde{g}$
    then lifts to a
    primitive map
    \[
    \widetilde{f} \colon (\mathfrak{S}^{\univ}_\O, (\mathcal{E})) \to (W_{\O_E}(S^\flat)/[a^\flat]^{m+1}, I_S).
    \]
    Let $f \colon (\mathfrak{S}^{\univ}_\O, (\mathcal{E})) \to (W_{\O_E}(S^\flat)/[a^\flat]^{m}, I_S)$ be the map induced by $\widetilde{f}$.
    Without loss of generality, we may assume that
    $Y=f(X^{\univ})$
    in
    $G(W_{\O_E}(S^\flat)/[a^\flat]^{m})_{I_S}$
    and
    $\beta$ is the following composition
    \[
    \mathfrak{Q}^{\univ}_{g} \overset{\gamma^{-1}_{f}}{\to} {f}^*(\mathcal{Q}_{X^{\univ}}) \simeq \mathcal{Q}_{f(X^{\univ})};
    \]
    see the last paragraph of Remark \ref{Remark:normalized period map}.
    In this case, we claim that
    $\mathrm{L}^{\beta}_{\mathfrak{Q}^{\univ}, \widetilde{g}}$ is the identity map.
    Let $g' \in \Hom(R_{G, \mu}, S/a^{m+1})_{g}$.
    By Lemma \ref{Lemma:lifting lemma II}, there exists a map
    \[
    f' \colon (\mathfrak{S}^{\univ}_\O, (\mathcal{E})) \to (W_{\O_E}(S^\flat)/[a^\flat]^{m+1}, I_S)
    \]
    of $\O_E$-prisms which induces $g'$ and $f$.
    It suffices to show that
    the period map
    \[
    \Per_{\mathcal{Q}_{\widetilde{f}(X^{\univ})}} \colon \Def(\mathcal{Q}_{f(X^{\univ})})_{(W_{\O_E}(S^\flat)/[a^\flat]^{m+1}, I_S)} \to \mathfrak{t}_{G, \mu} \otimes_k (a^{m}S/a^{m+1}S)
    \]
    sends the deformation $\mathcal{Q}_{f'(X^{\univ})}$ to $\mathrm{Diff}_{\widetilde{g}}(g')$.
    Here, as in Remark \ref{Remark:normalized period map},
    we identify
    $\Lift(P(\mathcal{Q}_{f(X^{\univ})})_{S/a^m}, (\mathcal{Q}_{\widetilde{f}(X^{\univ})})_{S/a^{m+1}})$
    with $\mathfrak{t}_{G, \mu} \otimes_k (a^{m}S/a^{m+1}S)$.
    We set $d:=\widetilde{f}(\mathcal{E})$.
    We have $f'(\mathcal{E})=u d$ for a unique $u \in (W_{\O_E}(S^\flat)/[a^\flat]^{m+1})^\times$.
    Since the image of $u$ in $W_{\O_E}(S^\flat)/[a^\flat]^{m}$ is equal to $1$,
    the equality
    $\phi(\mu(u))=1$ holds in $G(W_{\O_E}(S^\flat)/[a^\flat]^{m+1})$.
    Then we have
    \[
    \widetilde{f}(X^{\univ})_d=\widetilde{f}(U^{\univ})^{-1} X_\pi \quad \text{and} \quad f'(X^{\univ})_d=f'(U^{\univ})^{-1} X_\pi
    \]
    in $G(W_{\O_E}(S^\flat)/[a^\flat]^{m+1})$.
    By the proof of Proposition \ref{Proposition:canonical isomorphism of deformations of phi G torsors}, the unique isomorphism
    \[
    (\mathcal{Q}_{f'(X^{\univ})})_\phi \overset{\sim}{\to} (\mathcal{Q}_{\widetilde{f}(X^{\univ})})_\phi
    \]
    of deformations of the underlying $G$-$\phi$-module $(\mathcal{Q}_{f(X^{\univ})})_\phi$
    is given by
    $\widetilde{f}(U^{\univ})^{-1}f'(U^{\univ})$.
    Thus the lift
    $\Per_{\mathcal{Q}_{\widetilde{f}(X^{\univ})}}(\mathcal{Q}_{f'(X^{\univ})})$
    of the Hodge filtration
    can be identified with the $P_\mu(S/a^{m+1})$-orbit
    \[
    \widetilde{g}(U)^{-1}g'(U)P_\mu(S/a^{m+1}) \subset G(S/a^{m+1}),
    \]
    and this corresponds to
    $\mathrm{Diff}_{\widetilde{g}}(g') \in \mathfrak{t}_{G, \mu} \otimes_k (a^{m}S/a^{m+1}S)$ as desired.
    This concludes the proof of Theorem \ref{Theorem:existence of deformations with (Perfd-lin)}.
\end{proof}

We can now complete the proofs of Theorem \ref{Theorem:Existence of universal deformation} and Theorem \ref{Theorem:characterization of universal family}.

\begin{proof}[Proof of Theorem \ref{Theorem:Existence of universal deformation}]
This follows from Theorem \ref{Theorem:existence of deformations with (Perfd-lin)} and Remark \ref{Remark:implications}.
\end{proof}

\begin{proof}[Proof of Theorem \ref{Theorem:characterization of universal family}]
We have already proved that (1) implies (2), (3), (4), and (5); see Theorem \ref{Theorem:Existence of universal deformation}.
By Corollary \ref{Corollary:(BK) implies universality}, we have
$(2) \Rightarrow (1)$.

Let $\mathfrak{Q}^{\univ}$ be a universal deformation of $\mathcal{Q}$ over $R_{G, \mu}$, which has the properties $\mathrm{(Perfd)}$ and $\mathrm{(BK)}$ by Theorem \ref{Theorem:Existence of universal deformation}.
Let
$h \colon R_{G, \mu} \to R$ be the unique local homomorphism over $\O$ such that $h^*(\mathfrak{Q}^{\univ}) \simeq \mathfrak{Q}$.

We assume that (5) holds.
By Remark \ref{Remark:versality implies universality},
the induced homomorphism
$\mathfrak{t}_R \to \mathfrak{t}_{G, \mu}$
is surjective.
Since $\dim_k \mathfrak{t}_R = \dim_k \mathfrak{t}_{G, \mu}$, we then have
$\mathfrak{t}_R \overset{\sim}{\to} \mathfrak{t}_{G, \mu}$.
It follows that $h \colon R_{G, \mu} \to R$ is surjective.
Since $\dim R=\dim R_{G, \mu}$, we conclude that $h$ is an isomorphism, and hence $\mathfrak{Q}$ is a universal deformation. 
Thus we have
$(5) \Rightarrow (1)$.

It is obvious that (3) implies (4).
It remains to show that (4) implies (1).
We assume that 
the evaluation map
    \[
    \ev_\mathfrak{Q} \colon \Hom(R, \O_C/\pi^{2})_{e} \to \Def(\mathcal{Q})_{(W_{\O_E}(\O_{C^\flat})/[\pi^\flat]^{2}, I_{\O_C})}
    \]
    is bijective, where the base change of $\mathcal{Q}$ to $(W_{\O_E}(\O_{C^\flat})/[\pi^\flat], I_{\O_C})$ is denoted by the same symbol.
Since $\Def(\mathcal{Q})_{(W_{\O_E}(\O_{C^\flat})/[\pi^\flat]^{2}, I_{\O_C})}$ is nonempty (by Remark \ref{Remark:existence of deformation} or Proposition \ref{Proposition:existence of deformations for non-banal cases}), there exists a homomorphism
$g \colon R \to \O_C/\pi^{2}$ which is a lift of $e$.
In particular, it follows that
$\pi$ is not contained in $\mathfrak{m}^2_R \subset R$.
The following diagram commutes:
\[
    \xymatrix{
    \mathfrak{t}_{R} \otimes_k (\pi \O_C/\pi^2 \O_C) \ar^-{\mathrm{Diff}^{-1}_g}[r]  \ar[d]_-{}  &  \Hom(R, \O_C/\pi^{2})_{e}  \ar[d]_-{} \ar[r]^-{\ev_\mathfrak{Q}} & \Def(\mathcal{Q})_{(W_{\O_E}(\O_{C^\flat})/[\pi^\flat]^{2}, I_{\O_C})}\\
    \mathfrak{t}_{G, \mu} \otimes_k (\pi \O_C/\pi^2 \O_C) \ar[r]^-{\mathrm{Diff}^{-1}_{g \circ h}} & \Hom(R_{G, \mu}, \O_C/\pi^{2})_{e} \ar[ur]_-{\ev_{\mathfrak{Q}^{\univ}}} &
    }
    \]
where the vertical arrows are induced from $h$.
It follows that
\[
\mathfrak{t}_{R} \otimes_k (\pi \O_C/\pi^2 \O_C) \to \mathfrak{t}_{G, \mu} \otimes_k (\pi \O_C/\pi^2 \O_C),
\]
and hence $\mathfrak{t}_R \to \mathfrak{t}_{G, \mu}$,
is an isomorphism.
This implies that $h \colon R_{G, \mu} \to R$ is surjective.
Since $\pi$ is not contained in $\mathfrak{m}^2_R$, we see that $R/\pi$ is regular, or equivalently $\dim R/\pi = \dim_k \mathfrak{t}_R$.
We then obtain
\[
\dim R = \dim_k \mathfrak{t}_R +1 = \dim_k \mathfrak{t}_{G, \mu} +1 = \dim R_{G, \mu}.
\]
It follows that $h \colon R_{G, \mu} \to R$ is an isomorphism.

The proof of Theorem \ref{Theorem:characterization of universal family} is now complete.
\end{proof}

\section{Integral local Shimura varieties with hyperspecial level structure}\label{Section:Integral models of local Shimura varieties}

In this section, as an application of our deformation theory,
we prove the local representability and the formal smoothness of integral local Shimura varieties
with hyperspecial level structure; see Theorem \ref{Theorem:conjecture of Pappas-Rapoport} for the precise statement.
Throughout this section,
we assume that $G$ is a connected reductive group scheme over $\Spec \O_E$.
We assume that $k$ is an algebraic closure of $\F_q$.

In Section \ref{Subsection:$G$-shtukas}, we recall the definition of $G$-shtukas (with one leg) over perfectoid spaces over $k$ introduced by Scholze, and discuss their relation to $G$-$\mu$-displays.
In Section \ref{Subsection:quasi-isogenies}, we study the notion of quasi-isogeny for $G$-shtukas.
Integral local Shimura varieties are defined as moduli spaces of $G$-shtukas together with some quasi-isogeny.
In Section \ref{Subsection:Moduli spaces of $G$-shtukas}, we prove the main result of this section.

We assume that the reader is familiar with the theory of perfectoid spaces, $v$-sheaves, and ``the curve'' $\mathcal{Y}^{[0, \infty)}_S$.
Our basic references are \cite{Scholzediamond}, \cite{Scholze-Weinstein}, and \cite{FS}.

\subsection{$G$-shtukas}\label{Subsection:$G$-shtukas}

Let
$
\Perf_k
$
be the category of perfectoid spaces over $k$.
We endow $\Perf_k$ with the $v$-topology introduced in \cite[Definition 8.1]{Scholzediamond}.
A sheaf on $\Perf_k$ with respect to the $v$-topology is called a \textit{$v$-sheaf}.

\begin{rem}[The notation in this section]\label{Remark:Perf_k, untilt, Spd O_E}
    In this section, we let $S$ stand for a perfectoid space over $k$ (rather than a perfectoid ring).
    We write
\[
\O_{\breve{E}}:=W(k) \otimes_{W(\F_q)} \O_E \quad \text{and} \quad \breve{E}:=\O_{\breve{E}}[1/\pi].
\]
    Let
    $
    \Spd \O_{\breve{E}}
    $
    be the $v$-sheaf on $\Perf_k$ 
    which sends a perfectoid space $S$ over $k$
    to the set of isomorphism classes of untilts $S^\sharp$ of $S$ over $\O_{\breve{E}}$; see \cite[Section 18.1]{Scholze-Weinstein}.
    If $S=\Spa(R, R^+)$ is an affinoid perfectoid space over $k$, then an untilt $S^\sharp$ is also an affinoid perfectoid space, and 
    we use the following notation:
    \[
    S^\sharp=\Spa(R^\sharp, R^{\sharp+}).
    \]
    The ring $R^{\sharp+}$ is a perfectoid ring (in the sense of \cite{BMS}) with $(R^{\sharp+})^\flat = R^+$.
    Let $\varpi^\flat \in R^+$ be a pseudo-uniformizer (i.e.\ it is a topologically nilpotent unit in $R$).
    Then
    $\varpi:=\theta([\varpi^\flat]) \in R^{\sharp+}$
    is also a pseudo-uniformizer.
    After replacing $\varpi^\flat$ by $(\varpi^\flat)^{1/p^n}$ for a large enough $n$, we may assume that
    $\pi \in (\varpi)$ in $R^{\sharp+}$.
    The pair $(R^{\sharp+}, \varpi^\flat)$ is a perfectoid pair in the sense of Definition \ref{Definition:perfectoid pair}.
\end{rem}


\begin{ex}[Product of points]\label{Example:product of points}
Let $\{ (C_i, C^{+}_i) \}_{i \in I}$
be a set of analytic affinoid fields over $k$ such that $C_i$ is algebraically closed for any $i\in I$.
We choose a pseudo-uniformizer $\varpi^\flat_i \in C^{+}_i$ for each $i \in I$.
Let $\varpi^\flat:= (\varpi^\flat_i)_{i \in I} \in \prod_{i \in I} C^{+}_i$.
We endow
$\prod_{i \in I} C^{+}_i$
with the $\varpi^\flat$-adic topology.
Then
\[
S:=\Spa((\prod_{i \in I} C^{+}_i)[1/\varpi^\flat], \prod_{i \in I} C^{+}_i)
\]
is an affinoid perfectoid space over $k$.
Following \cite{Gleason},
we call such a perfectoid space a \textit{product of points}.
Products of points over $k$ form a basis of $\Perf_k$ with respect to the $v$-topology; see \cite[Example 1.1]{Gleason}.
Let $S^\sharp$ be an untilt of $S$ over $\O_{\breve{E}}$.
By \cite[Corollary 3.20]{Scholzediamond}, we have the corresponding untilt
$\Spa(C^{\sharp}_i, C^{\sharp+}_i)$ of $\Spa(C_i, C^{+}_i)$ over $S^\sharp$.
Let $\varpi_i:=\theta([\varpi^\flat_i]) \in C^{\sharp+}_i$ and let $\varpi:=(\varpi_i)_{i \in I} \in \prod_{i \in I} C^{\sharp+}_i$.
Then the untilt $S^\sharp$ is isomorphic to
$
\Spa((\prod_{i \in I} C^{\sharp+}_i)[1/\varpi], \prod_{i \in I} C^{\sharp+}_i).
$
\end{ex}

We shall recall the definition of $G$-shtukas,
following \cite{Scholze-Weinstein}.
In the rest of this subsection, we assume that
$S=\Spa(R, R^+)$ is an affinoid perfectoid space over $k$ for simplicity.
Let
$S^\sharp=\Spa(R^\sharp, R^{\sharp+})$
be an untilt
of $S$ over $\O_{\breve{E}}$.

We consider the following adic space
\[
\mathcal{Y}^{[0, \infty)}_S:=\Spa(W_{\O_E}(R^+))\backslash V([\varpi^\flat]),
\]
where $\varpi^\flat \in R^+$ is a pseudo-uniformizer and $W_{\O_E}(R^+)$ is equipped with the
$(\pi, [\varpi^\flat])$-adic topology.
The adic space $\mathcal{Y}^{[0, \infty)}_S$
is called \textit{the curve}
and denoted by
$\mathcal{Y}_S$ in \cite[Section II.1.1]{FS}.
The curve $\mathcal{Y}^{[0, \infty)}_S$ is a \textit{sousperfectoid} adic space over $\O_{\breve{E}}$
in the sense of \cite[Definition 6.3.1, Appendix to Lecture 19]{Scholze-Weinstein}; see the proof of \cite[Proposition II.1.1]{FS}.
The Frobenius of $W_{\O_E}(R^+)$ induces an isomorphism
$
\Frob \colon \mathcal{Y}^{[0, \infty)}_S \to \mathcal{Y}^{[0, \infty)}_S
$
over $\O_{E}$.
The untilt $S^\sharp$ can be regraded as a (closed) Cartier divisor of $\mathcal{Y}^{[0, \infty)}_S$; see \cite[Proposition II.1.4]{FS}.
This Cartier divisor is induced by the ideal $I_{R^{\sharp+}} \subset W_{\O_E}(R^+)$.

\begin{rem}
For a sousperfectoid adic space $X$ over $\O_E$,
let $\Vect(X)$ be the category of vector bundles on $X$.
A \textit{$G$-torsor} over $X$ is an exact tensor functor
$
\Rep_{\O_E}(G) \to \Vect(X),
$
where $\Rep_{\O_E}(G)$ is the category of algebraic representations of $G$ on free $\O_E$-modules of finite rank.
We refer to \cite[Theorem 19.5.2]{Scholze-Weinstein} and \cite[Definition/Proposition III.1.1]{FS} for equivalent definitions of $G$-torsors.
\end{rem}

\begin{defn}\label{Definition:G-shtuka}
A \textit{$G$-shtuka} over $S$ with one leg at $S^\sharp$ is a $G$-torsor $\mathscr{P}$ over $\mathcal{Y}^{[0, \infty)}_S$ with an isomorphism
\[
\phi_\mathscr{P} \colon (\Frob^*\mathscr{P})_{ \vert \mathcal{Y}^{[0, \infty)}_S \backslash S^{\sharp}} \overset{\sim}{\to} \mathscr{P}_{ \vert \mathcal{Y}^{[0, \infty)}_S \backslash S^{\sharp}} 
\]
of $G$-torsors over the open subspace $\mathcal{Y}^{[0, \infty)}_S \backslash S^{\sharp} \subset \mathcal{Y}^{[0, \infty)}_S$ which is meromorphic along
the Cartier divisor
$S^{\sharp} \hookrightarrow \mathcal{Y}^{[0, \infty)}_S$.
The meromorphic condition on $\phi_\mathscr{P}$ is defined as in \cite[Definition 5.3.5]{Scholze-Weinstein} (via the Tannakian formalism).
\end{defn}

\begin{rem}\label{Remark:base change of vector bundles}
Let
$\Vect(W_{\O_E}(R^+))$
be the category of finite projective $W_{\O_E}(R^+)$-modules.
We have a natural morphism
$
\mathcal{Y}^{[0, \infty)}_S \to \Spec(W_{\O_E}(R^+))
$
of locally ringed spaces.
This induces a functor
\[
\Vect(W_{\O_E}(R^+))  \to \Vect(\mathcal{Y}^{[0, \infty)}_S),
\]
which in turn induces a functor from the category of $G$-torsors over $\Spec W_{\O_E}(R^+)$ to the category of $G$-torsors over $\mathcal{Y}^{[0, \infty)}_S$.
\end{rem}

In the following,
a $G$-Breuil--Kisin module
over $(W_{\O_E}(R^+), I_{R^{\sharp+}})$
in the sense of Definition \ref{Definition:G-BK module of type mu}
is also called a $G$-Breuil--Kisin module for $R^{\sharp+}$.

\begin{ex}\label{Example:G shtuka assotieted with phi G torsor}
Let
$\mathcal{P}_{\sht}$
be the $G$-torsor over $\mathcal{Y}^{[0, \infty)}_S$
associated with a $G$-Breuil--Kisin module $\mathcal{P}$ for $R^{\sharp+}$.
The isomorphism
$
\phi_{\mathcal{P}_{\sht}} \colon (\Frob^*\mathcal{P}_{\sht})_{ \vert \mathcal{Y}^{[0, \infty)}_S \backslash S^{\sharp}} \overset{\sim}{\to} (\mathcal{P}_{\sht})_{ \vert \mathcal{Y}^{[0, \infty)}_S \backslash S^{\sharp}}
$
induced by the Frobenius $F_\mathcal{P}$
is meromorphic along $S^{\sharp} \hookrightarrow \mathcal{Y}^{[0, \infty)}_S$.
With this isomorphism,
we view $\mathcal{P}_{\sht}$
as a $G$-shtuka over $S$ with one leg at $S^\sharp$.
\end{ex}

Following \cite{Scholze-Weinstein} and \cite{PappasRapoport21},
we introduce some boundedness condition on $G$-shtukas.
First, we fix some notation:

\begin{rem}\label{Remark:completed local ring of the curve}
    We assume that $S$ is a geometric point of rank 1, that is, $S= \Spa(C, \O_C)$ for an algebraically closed nonarchimedean field $C$ over $k$ with ring of integers $\O_C$.
    Let $S^\sharp=\Spa(C^\sharp, \O_{C^\sharp})$ be an untilt of $S$ over $\O_{\breve{E}}$.
    The completed local ring of $\mathcal{Y}^{[0, \infty)}_S$ at the closed point $\Spa(C^\sharp, \O_{C^\sharp}) \hookrightarrow \mathcal{Y}^{[0, \infty)}_S$ is naturally isomorphic to
    \[
    B^+_{\dR}(C^\sharp):={\varprojlim}_{m} W_{\O_E}(\O_C)[1/[\varpi^\flat]]/\xi^m
    \]
    where $\xi \in I_{\O_{C^\sharp}}$ is a generator.
    (See also \cite[Example 15.2.10]{Scholze-Weinstein} and the discussion before it.)
    We recall that $B^+_{\dR}(C^\sharp)$ is a complete discrete valuation ring.
    The maximal ideal of $B^+_{\dR}(C^\sharp)$ is generated by $\xi$, and the homomorphism $\theta$ induces $B^+_{\dR}(C^\sharp)/\xi \simeq C^\sharp$.
    We set
    $
    B_{\dR}(C^\sharp):=B^+_{\dR}(C^\sharp)[1/\xi].
    $

    Let $\mathscr{P}$ be a $G$-shtuka over $S$ with one leg at $S^\sharp$.
    We write $\mathscr{P}_{\vert B^+_{\dR}(C^\sharp)}$ for the restriction of $\mathscr{P}$ to 
    $B^+_{\dR}(C^\sharp)$, which we regard as a $G$-torsor over $\Spec B^+_{\dR}(C^\sharp)$.
    Since $\phi_\mathscr{P}$ is meromorphic, it induces an isomorphism
    \[
    (\Frob^*\mathscr{P})_{ \vert B^+_{\dR}(C^\sharp)} \times_{\Spec B^+_{\dR}(C^\sharp)} \Spec B_{\dR}(C^\sharp) \overset{\sim}{\to} \mathscr{P}_{ \vert B^+_{\dR}(C^\sharp)} \times_{\Spec B^+_{\dR}(C^\sharp)} \Spec B_{\dR}(C^\sharp)
    \]
    of $G$-torsors over $\Spec B_{\dR}(C^\sharp)$, which is denoted by $(\phi_{\mathscr{P}})_{ \vert B_{\dR}(C^\sharp)}$.
    We note that, since $C^\sharp$ is algebraically closed, the $G$-torsor $\mathscr{P}_{\vert B^+_{\dR}(C^\sharp)}$ is trivial.
\end{rem}

Let
$\mu \colon \G_m \to G_{\O_{\breve{E}}}$
be a minuscule cocharacter.

\begin{defn}\label{Definition:G shtuka bounded by mu}
Let 
$\mathscr{P}$
be a $G$-shtuka over $S$ with one leg at $S^\sharp$.

\begin{enumerate}
    \item  We assume that $S = \Spa(C, \O_C)$ is a geometric point of rank 1 with $S^\sharp=\Spa(C^\sharp, \O_{C^\sharp})$.
    We say that $\mathscr{P}$ is \textit{bounded by $\mu$} if for some (and hence all) trivializations $\mathscr{P}_{ \vert B^+_{\dR}(C^\sharp)} \simeq G_{B^+_{\dR}(C^\sharp)}$ and $(\Frob^*\mathscr{P})_{ \vert B^+_{\dR}(C^\sharp)} \simeq G_{B^+_{\dR}(C^\sharp)}$, the isomorphism $(\phi_{\mathscr{P}})_{ \vert B_{\dR}(C^\sharp)}$ (see Remark \ref{Remark:completed local ring of the curve}) is given by
        $g \mapsto Yg$ for some element $Y$
        in the double coset
        \[
        G(B^+_{\dR}(C^\sharp))\mu(\xi)^{-1}G(B^+_{\dR}(C^\sharp)) \subset G(B_{\dR}(C^\sharp)).
        \]
    \item In general, we say that $\mathscr{P}$ is \textit{bounded by $\mu$} if, for any morphism
    $\Spa(C, \O_C) \to S$
    with $\Spa(C, \O_C)$ a geometric point of rank 1, the pull-back of $\mathscr{P}$ to 
    $\mathcal{Y}^{[0, \infty)}_{\Spa(C, \O_C)}$ is bounded by $\mu$ in the sense of (1).
\end{enumerate}
Notice that we have chosen $\mu(\xi)^{-1}$ rather than $\mu(\xi)$; our convention agrees with that of \cite[Definition 2.4.3]{PappasRapoport21}.
\end{defn}

\begin{ex}\label{Example:G shtuka assotieted with phi G torsor of type mu}
Let
$\mu^{-1}$ be the inverse of $\mu$.
If $\mathcal{P}$ is a $G$-Breuil--Kisin module for $R^{\sharp+}$ of type $\mu^{-1}$ in the sense of Definition \ref{Definition:G-BK module of type mu},
then the associated $G$-shtuka
$\mathcal{P}_{\sht}$
over $S$ with one leg at $S^\sharp$ is bounded by $\mu$.
\end{ex}

The following result can be viewed as a partial converse to Example \ref{Example:G shtuka assotieted with phi G torsor of type mu}.

\begin{prop}\label{Proposition:bounded by mu and of type mu}
    Assume that
    $S=\Spa(R, R^+)$
    is a product of points over $k$.
    Let
    $\mathcal{P}$
    be a $G$-Breuil--Kisin module for $R^{\sharp+}$.
    Then
    $\mathcal{P}$
    is of type $\mu^{-1}$
    if and only if
    the associated $G$-shtuka 
    $\mathcal{P}_{\sht}$
    over $S$ with one leg at $S^\sharp$ is bounded by $\mu$.
\end{prop}

\begin{proof}
We write
$
(R, R^+)=((\prod_{i \in I} C^{+}_i)[1/\varpi^\flat], \prod_{i \in I} C^{+}_i).
$
Then
$(R^{\sharp}, R^{\sharp+})$
is of the form
$((\prod_{i \in I} C^{\sharp+}_i)[1/\varpi], \prod_{i \in I} C^{\sharp+}_i)$
as in Example \ref{Example:product of points}.
It suffices to prove the ``if'' direction.
By the proof of \cite[Proposition 5.6.11]{Ito-K23},
it is enough to show that
the base change of $\mathcal{P}$ to
$(W_{\O_E}(C^{+}_i), I_{C^{\sharp+}_i})$
is of type $\mu^{-1}$ for each $i \in I$.
Thus we may assume that the set $I$ is a singleton.
We write $(R, R^+)=(C, C^+)$
and
$(R^{\sharp}, R^{\sharp+}) =(C^{\sharp}, C^{\sharp+})$.

We note that $\mathcal{P}$ is trivial as a $G$-torsor since $W_{\O_E}(C^{+})$ is strictly henselian.
We fix a trivialization
$\mathcal{P} \simeq G_{W_{\O_E}(C^{+})}$,
and 
we set $Y:=F_{\mathcal{P}}(1) \in G(W_{\O_E}(C^{+})[1/\xi])$, where $\xi \in I_{C^{\sharp+}}$ is a generator.
Since $\mathcal{P}_{\sht}$ is bounded by $\mu$,
the image of
$Y$
under the homomorphism
$
G(W_{\O_E}(C^{+})[1/\xi]) \to G(B_{\dR}(C^\sharp))
$
lies in
$G(B^+_{\dR}(C^\sharp))\mu(\xi)^{-1}G(B^+_{\dR}(C^\sharp))$.
We want to prove that
$Y$ is contained in
\[
G(W_{\O_E}(C^{+}))\mu(\xi)^{-1}G(W_{\O_E}(C^{+})) \subset G(W_{\O_E}(C^{+})[1/\xi]).
\]

To simplify the notation, we write $A:=W_{\O_E}(C^{+})$
and $B^+:=B^+_{\dR}(C^\sharp)$.
The homomorphism
$A/\xi \to B^+/\xi$
can be identified with $C^{\sharp+} \to C^\sharp$, 
and hence it is injective.
It follows that $A \to B^+$ is injective, and we have
$A=A[1/\xi] \cap B^+$.
This implies that
the natural map
$
G(A[1/\xi])/G(A) \to G(B^+[1/\xi])/G(B^+)
$
is injective.
We claim that this map induces the following bijection:
\[
G(A)\mu(\xi)^{-1}G(A)/G(A) \overset{\sim}{\to} G(B^+)\mu(\xi)^{-1}G(B^+)/G(B^+).
\]
This claim implies that $Y \in G(A)\mu(\xi)^{-1}G(A)$, as desired.

We shall prove the claim.
The map
$
G(A)\mu(\xi)^{-1}G(A)/G(A) \to G(A)/G_\mu(A, (\xi))
$
defined by $g\mu(\xi)^{-1}G(A) \mapsto gG_\mu(A, (\xi))$ is bijective.
By Proposition \ref{Proposition:BB isomorphism}, we have
$G(A)/G_\mu(A, (\xi)) \overset{\sim}{\to} G(C^{\sharp+})/P_\mu(C^{\sharp+})$.
Similarly, we have
\[
G(B^+)\mu(\xi)^{-1}G(B^+)/G(B^+) \overset{\sim}{\to} G(B^+)/G_\mu(B^+, (\xi)) \overset{\sim}{\to} G(C^{\sharp})/P_\mu(C^{\sharp}).
\]
(We note that the results of Section \ref{Subsection:Display groups} apply to the pair $(B^+, (\xi))$.)
It then suffices to prove that
the map
\[
G(C^{\sharp+})/P_\mu(C^{\sharp+}) \to
G(C^{\sharp})/P_\mu(C^{\sharp})
\]
is bijective.
This follows from the valuative criterion and the properness of the scheme $G_{\O_{\breve{E}}}/P_\mu$ over $\Spec \O_{\breve{E}}$.
(We also use that the field $C^\sharp$ is algebraically closed and that the valuation ring $C^{\sharp+}$ is strictly henselian.)
\end{proof}

We define an open subset
\[
\mathcal{Y}^{[0, \infty]}_S:=\Spa(W_{\O_E}(R^+))\backslash V(\pi, [\varpi^\flat]) \subset \Spa(W_{\O_E}(R^+)).
\]
We collect some results on $\mathcal{Y}^{[0, \infty]}_S$ that will be needed in the sequel.

\begin{rem}\label{Remark:punctured curve is sousperfectoid}
    As explained in \cite[Proposition 13.1.1]{Scholze-Weinstein} and \cite{Kedlaya} (see also \cite[Section 2.1]{PappasRapoport21}),
    we can regard $\mathcal{Y}^{[0, \infty]}_S$ as a sousperfectoid adic space.
    More precisely, we consider the following rational open subsets of $\Spa(W_{\O_E}(R^+))$:
    \begin{align*}
    U&:= \{ \, x \in \Spa(W_{\O_E}(R^+)) \, \vert \, \vert [\varpi^\flat](x) \vert \leq \vert \pi(x) \vert \neq 0 \, \},\\
    V&:= \{ \, x \in \Spa(W_{\O_E}(R^+)) \, \vert \, \vert \pi(x) \vert \leq \vert [\varpi^\flat](x) \vert \neq 0 \, \}.
    \end{align*}
    By the proof of \cite[Proposition 3.6]{Kedlaya},
    both $U$ and $V$ are sousperfectoid adic spaces.
    We can glue them along $U\cap V$ and endow $\mathcal{Y}^{[0, \infty]}_S$ with the structure of a sousperfectoid adic space.
    Strictly speaking, only the case where $\O_E=\Z_p$ is considered in \cite{Kedlaya}.
    However, it is easy to see that the same argument works for general $\O_E$.
\end{rem}

\begin{rem}\label{Remark:punctured curve, vector bundles}
Let
$
j^* \colon \Vect(W_{\O_E}(R^+))  \to \Vect(\mathcal{Y}^{[0, \infty]}_S)
$
be the functor induced by the natural morphism
    $
    j \colon \mathcal{Y}^{[0, \infty]}_S \to \Spec(W_{\O_E}(R^+))
    $
    of locally ringed spaces.
    We claim that $j^*$ is fully faithful.
    Indeed, the morphism $j$ factors through the open subscheme
    \[
    \mathcal{Y}^{[0, \infty], \alg}_S:=\Spec(W_{\O_E}(R^+)) \backslash V(\pi, [\varpi^\flat]) \subset \Spec(W_{\O_E}(R^+)). 
    \]
    Thus $j^*$ can be written as the composition
    \[
    \Vect(W_{\O_E}(R^+)) \to \Vect(\mathcal{Y}^{[0, \infty], \alg}_S) \to \Vect(\mathcal{Y}^{[0, \infty]}_S),
    \]
    where
    $\Vect(\mathcal{Y}^{[0, \infty], \alg}_S)$ is the category of vector bundles on the scheme $\mathcal{Y}^{[0, \infty], \alg}_S$.
    Since we have
    $W_{\O_E}(R^+) = \O_{\mathcal{Y}^{[0, \infty], \alg}_S}(\mathcal{Y}^{[0, \infty], \alg}_S)$, the first functor is fully faithful.
    The second functor is an equivalence by \cite[Theorem 3.8]{Kedlaya} (again, the result holds for general $\O_E$ by the same argument), and hence the claim follows.
    The second equivalence is also an exact equivalence of exact categories (cf.\ \cite[Remark 2.3]{Gleasongeometric}).
    
    Let
    $\xi \in I_{R^{\sharp+}}$ be a generator.
    For finite projective $W_{\O_E}(R^+)$-modules $M$ and $N$,
    the functor $j^*$ induces a bijection between
    the set of isomorphisms
    $M[1/\xi] \overset{\sim}{\to} N[1/\xi]$ and the set of isomorphisms
    $
    (j^*M)_{\vert \mathcal{Y}^{[0, \infty]}_S \backslash S^{\sharp}} \overset{\sim}{\to} (j^*N)_{\vert \mathcal{Y}^{[0, \infty]}_S \backslash S^{\sharp}}
    $
    which are meromorphic.
\end{rem}

The following result is obtained in \cite[Section 2.1.2]{Gleasonthesis} (in the more general case where $G$ is parahoric).

\begin{prop}\label{Proposition:G-torsor extension}
    Assume that
    $S=\Spa(R, R^+)$
    is a product of points over $k$.
    Then the functor
    $j^* \colon \Vect(W_{\O_E}(R^+))  \to \Vect(\mathcal{Y}^{[0, \infty]}_S)$
    induces an equivalence from the category of $G$-torsors over $\Spec W_{\O_E}(R^+)$ to that of $G$-torsors over $\mathcal{Y}^{[0, \infty]}_S$.
\end{prop}

\begin{proof}
We briefly recall the argument for the convenience of the reader.
By Remark \ref{Remark:punctured curve, vector bundles}, it suffices to prove that
the restriction functor
from the category of $G$-torsors over $\Spec W_{\O_E}(R^+)$ to that of $G$-torsors over $\mathcal{Y}^{[0, \infty], \alg}_S$ is an equivalence.

By the proof of \cite[Proposition 2.1.17]{Gleasonthesis},
we see that the functor
$
\Vect(W_{\O_E}(R^+))  \to \Vect(\mathcal{Y}^{[0, \infty], \alg}_S)
$
is an equivalence.
(This is proved in \cite[Theorem 2.7]{Kedlaya} if $S=\Spa(C, C^+)$ for an algebraically closed nonarchimedean field $C$ over $k$, and the general case can be deduced from this special case.)
Then the assertion follows from \cite[Proposition 8.5]{Anschutz} since we assume that $G$ is reductive here.
\end{proof}

The Frobenius of $W_{\O_E}(R^+)$ induces an isomorphism
$
\Frob \colon \mathcal{Y}^{[0, \infty]}_S \to \mathcal{Y}^{[0, \infty]}_S
$
over $\O_{E}$.
It will be convenient to make the following definition:

\begin{defn}\label{Definition:extended G shtuka}
An \textit{extended $G$-shtuka} over $S$ with one leg at $S^{\sharp}$ is a $G$-torsor $\widetilde{\mathscr{P}}$ over
$\mathcal{Y}^{[0, \infty]}_S$ 
with
an isomorphism
    \[
\phi_{\widetilde{\mathscr{P}}} \colon (\Frob^*\widetilde{\mathscr{P}})_{ \vert \mathcal{Y}^{[0, \infty]}_S \backslash S^{\sharp}} \overset{\sim}{\to} \widetilde{\mathscr{P}}_{ \vert \mathcal{Y}^{[0, \infty]}_S \backslash S^{\sharp}}
\]
which is meromorphic along $S^{\sharp} \hookrightarrow \mathcal{Y}^{[0, \infty]}_S$.
\end{defn}

As in Example \ref{Example:G shtuka assotieted with phi G torsor},
to a $G$-Breuil--Kisin module
$\mathcal{P}$ for $R^{\sharp+}$,
we can attach an extended $G$-shtuka $\widetilde{\mathcal{P}}_{\sht}$ over $S$ with one leg at $S^{\sharp}$.

\begin{cor}\label{Corollary:equivalence between G BK module and extended G shtuks}
    The functor $\mathcal{P} \mapsto \widetilde{\mathcal{P}}_{\sht}$ from
    the groupoid of $G$-Breuil--Kisin modules for $R^{\sharp+}$ to that of extended $G$-shtukas over $S$ with one leg at $S^{\sharp}$ is fully faithful.
    If $S$
    is a product of points over $k$, then the functor is an equivalence.
\end{cor}

\begin{proof}
The first assertion follows from Remark \ref{Remark:punctured curve, vector bundles}.
The second assertion follows from Remark \ref{Remark:punctured curve, vector bundles} and Proposition \ref{Proposition:G-torsor extension}.
\end{proof}

\begin{prop}[{cf.\ \cite[Proposition 2.4.6]{PappasRapoport21}}]\label{Proposition:equivalence between G shutukas and G-BK modules}
Assume that $S=\Spa(C, \O_C)$ is a geometric point of rank $1$ with $S^{\sharp}=\Spa(C^{\sharp}, \O_{C^\sharp})$.
Then
the construction $\mathcal{P} \mapsto \mathcal{P}_{\sht}$
induces an equivalence of groupoids:
\[
\left\{
\begin{tabular}{c}
     $G$-Breuil--Kisin modules for $\O_{C^\sharp}$
\end{tabular}
\right\}
\overset{\sim}{\to}
\left\{
\begin{tabular}{c}
     $G$-shtukas over $\Spa(C, \O_C)$ \\
     with one leg at $\Spa(C^{\sharp}, \O_{C^\sharp})$
\end{tabular}
\right\}.
\]
Moreover, this equivalence induces
\[
\left\{
\begin{tabular}{c}
     $G$-$\mu^{-1}$-displays \\ over $(W_{\O_E}(\O_C), I_{\O_{C^\sharp}})$
\end{tabular}
\right\}
\overset{\sim}{\to}
\left\{
\begin{tabular}{c}
     $G$-shtukas over $\Spa(C, \O_C)$ \\
     with one leg at $\Spa(C^{\sharp}, \O_{C^\sharp})$ \\
     which are bounded by $\mu$
\end{tabular}
\right\}.
\]
\end{prop}

\begin{proof}
Although this result is not used in the rest of this paper, we include a proof for completeness.
By Proposition \ref{Proposition:G-displays to G-BK}
    and
    Proposition \ref{Proposition:bounded by mu and of type mu}, it suffices to prove the first assertion.
    By \cite[Proposition 2.2.7]{PappasRapoport21} (which holds for general $\O_E$ by the same argument), the functor is fully faithful.
    As in \cite[Proposition 2.4.6]{PappasRapoport21},
    we can extend a $G$-shtuka $\mathscr{P}$ over $S$
    with one leg at $S^\sharp$ to
    an extended $G$-shtuka $\widetilde{\mathscr{P}}$, by using Fargues' classification \cite{Fargues20} of $G$-torsors over the Fargues--Fontaine curve (see also \cite{Anschutz19}).
    By Corollary \ref{Corollary:equivalence between G BK module and extended G shtuks},
    it then follows that the functor is essentially surjective.
\end{proof}

\subsection{Quasi-isogenies}\label{Subsection:quasi-isogenies}

Let $S=\Spa(R, R^+)$ be an affinoid perfectoid space over $k$,
and let
$S^\sharp=\Spa(R^\sharp, R^{\sharp+})$ be an untilt of $S$ over $\O_{\breve{E}}$.
We fix a pseudo-uniformizer $\varpi^\flat \in R^+$.

For an integer $n \geq 1$, we define a rational open subset
\[
\mathcal{Y}^{[n, \infty]}_S:= \{ \, x \in \Spa(W_{\O_E}(R^+)) \, \vert \, \vert [\varpi^\flat](x) \vert \leq \vert \pi(x) \vert^{n} \neq 0 \, \} \subset \Spa(W_{\O_E}(R^+)).
\]
We set
$
\mathcal{Y}^{[n, \infty)}_S:=\mathcal{Y}^{[0, \infty)}_S \cap \mathcal{Y}^{[n, \infty]}_S.
$
For the isomorphism
$\Frob \colon \mathcal{Y}^{[0, \infty]}_S \to \mathcal{Y}^{[0, \infty]}_S$,
the equality
$
\Frob(\mathcal{Y}^{[n, \infty]}_S)=\mathcal{Y}^{[qn, \infty]}_S$
holds.
Similarly, we have
$\Frob(\mathcal{Y}^{[n, \infty)}_S)=\mathcal{Y}^{[qn, \infty)}_S$.

\begin{defn}\label{Definition:quasi isogeny for G-shtuka}
Let $\mathcal{P}_0$ be a $G$-Breuil--Kisin module over $(\O_{\breve{E}}, (\pi))$ and $\mathscr{P}$ a $G$-shtuka over $S$ with one leg at $S^\sharp$.
A \textit{quasi-isogeny} from $\mathscr{P}$ to $\mathcal{P}_0$
is an element of the set
\[
{\varinjlim}_n\{ \, \text{Frobenius equivariant isomorphisms} \ \mathscr{P}_{\vert \mathcal{Y}^{[n, \infty)}_S} \overset{\sim}{\to} (\mathcal{P}_0)_{\vert \mathcal{Y}^{[n, \infty)}_S}  \ \text{over} \ \mathcal{Y}^{[n, \infty)}_S \, \}
\]
where $n$ runs over all integers $n \geq 1$ such that $\mathcal{Y}^{[n, \infty)}_S$ does not intersect with $S^\sharp$.
Here $(\mathcal{P}_0)_{\vert \mathcal{Y}^{[n, \infty)}_S}$ is the pull-back of $\mathcal{P}_0$ along
the morphism $\mathcal{Y}^{[n, \infty)}_S \to \Spec \O_{\breve{E}}$ of locally ringed spaces.
Similarly, for
an extended $G$-shtuka $\widetilde{\mathscr{P}}$ over $S$ with one leg at $S^{\sharp}$,
we define a \textit{quasi-isogeny} from
$\widetilde{\mathscr{P}}$ to $\mathcal{P}_0$ as an element of the set
\[
{\varinjlim}_n\{ \, \text{Frobenius equivariant isomorphisms} \ \widetilde{\mathscr{P}}_{\vert \mathcal{Y}^{[n, \infty]}_S} \overset{\sim}{\to} (\mathcal{P}_0)_{\vert \mathcal{Y}^{[n, \infty]}_S}  \ \text{over} \ \mathcal{Y}^{[n, \infty]}_S \, \}
\]
where $n$ runs over all integers $n \geq 1$ such that $\mathcal{Y}^{[n, \infty]}_S$ does not intersect with $S^\sharp$.
\end{defn}


\begin{rem}\label{Remark:extending along quasi-isogenies}
Let $\mathscr{P}$ be a $G$-shtuka over $S$ with one leg at $S^\sharp$.
Using a quasi-isogeny from $\mathscr{P}$ to $\mathcal{P}_0$, we can extend
$\mathscr{P}$
to an extended $G$-shtuka over $S$ with one leg at $S^{\sharp}$.
It follows that the restriction functor
from the groupoid of extended $G$-shtukas $\widetilde{\mathscr{P}}$ over $S$ with one leg at $S^\sharp$ together with a quasi-isogeny from
$\widetilde{\mathscr{P}}$ to $\mathcal{P}_0$
to the groupoid of $G$-shtukas $\mathscr{P}$ over $S$ with one leg at $S^\sharp$ together with a quasi-isogeny from
$\mathscr{P}$ to $\mathcal{P}_0$ is an equivalence.
\end{rem}

We want to relate quasi-isogenies of (extended) $G$-shtukas to quasi-isogenies of $G$-Breuil--Kisin modules, which we define as follows:

\begin{defn}\label{Definition:quasi-isogeny}
    Let $(A, (\pi))$ be an $\O_E$-prism.
    Let $\mathcal{P}$ and $\mathcal{P}'$ be $G$-Breuil--Kisin modules over $(A, (\pi))$.
    A \textit{quasi-isogeny} from $\mathcal{P}$ to $\mathcal{P}'$ is a Frobenius equivariant isomorphism
    \[
    \eta  \colon \mathcal{P}[1/\pi] \overset{\sim}{\to} \mathcal{P}'[1/\pi]
    \]
    of $G$-torsors over $\Spec A[1/\pi]$.
    If $\eta$ arises from a (unique) isomorphism
    $\mathcal{P} \overset{\sim}{\to} \mathcal{P}'$
    of $G$-Breuil--Kisin modules over $(A, (\pi))$,
    we say that $\eta$ is \textit{effective}.
    In this case, the isomorphism $\mathcal{P} \overset{\sim}{\to} \mathcal{P}'$ is denoted by the same symbol $\eta$.
\end{defn}

We assume that $\pi=0$ in $R^{\sharp+}/\varpi$,
or equivalently $\pi \in I_{R^{\sharp+}}$, where $\varpi:=\theta([\varpi^\flat]) \in R^{\sharp+}$.
By \cite[Lemma 2.3.3]{Ito-K23},
we have
$I_{R^{\sharp+}}=(\pi)$ in $W_{\O_E}(R^+)/[\varpi^\flat]$.
We use the following notation.
For an integer $m \geq 0$, we set
\[
(A_m, (\pi)):=(W_{\O_E}(R^+)/[\varpi^\flat]^{1/q^m}, (\pi)),
\]
where $[\varpi^\flat]^{1/q^m}:=[(\varpi^\flat)^{1/q^m}]$.
We write
$A=A_0$.
Let
\[
(A_{\red}, (\pi)):=(W_{\O_E}(R^+_{\red}), (\pi))
\]
where
$
R^+_{\red}:= \varinjlim_{m} R^+/(\varpi^\flat)^{1/q^m}.
$
For a
$G$-Breuil--Kisin module
$\mathcal{P}$ over $(A, (\pi))$,
the base change of $\mathcal{P}$ to $(A_m, (\pi))$ (resp.\ $(A_{\red}, (\pi))$) is denoted by $\mathcal{P}_m$ (resp.\ $\mathcal{P}_{\red}$).
We use the same notation for
quasi-isogenies (or isomorphisms).

\begin{lem}\label{Lemma:quasi-isogeny reduction equivalence}
    Let $\mathcal{P}$ and $\mathcal{P}'$ be $G$-Breuil--Kisin modules over $(A, (\pi))$.
    The map
    \[
    \{ \, \text{quasi-isogenies from}\ \mathcal{P} \ \text{to} \  \mathcal{P}' \,\} \to \{ \, \text{quasi-isogenies from}\ \mathcal{P}_m \ \text{to} \  \mathcal{P}'_m \,\}, \quad \eta \mapsto \eta_m
    \]
    is bijective for every $m \geq 1$.
\end{lem}

\begin{proof}
    By the Tannakian formalism, we are reduced to proving the following claim (by applying it to internal homs): Let $M$ be a finite projective $A$-module with
    an isomorphism
$F_M \colon (\phi^*M)[1/\pi] \overset{\sim}{\to} M[1/\pi]$.
Then the natural homomorphism
\[
M[1/\pi]^{F_M=\id} \to (M \otimes_A A_m[1/\pi])^{F_{M}=\id}
\]
is bijective.
This is easy to see since the kernel of $A \to A_m$ is killed by $\phi^m$.
\end{proof}

\begin{lem}\label{Lemma:quasi-isogeny reduction effective}
    Let $\mathcal{P}$ and $\mathcal{P}'$ be $G$-Breuil--Kisin modules over $(A, (\pi))$ such that $\mathcal{P}$ and $\mathcal{P}'$ are trivial as $G$-torsors over $\Spec A$.
    Let $\eta \colon \mathcal{P}[1/\pi] \overset{\sim}{\to} \mathcal{P}'[1/\pi]$ be a quasi-isogeny.
    If 
    the base change
    $\eta_{\red} \colon \mathcal{P}_{\red}[1/\pi] \overset{\sim}{\to} \mathcal{P}'_{\red}[1/\pi]$
    is effective,
    then there exists an integer $m \geq 1$
    such that
    $
    \eta_m \colon \mathcal{P}_{m}[1/\pi] \overset{\sim}{\to} \mathcal{P}'_{m}[1/\pi]
    $
    is effective.
\end{lem}

\begin{proof}
We fix trivializations $\mathcal{P} \simeq G_A$ and $\mathcal{P}' \simeq G_A$.
Then $\eta$ is identified with an element $g \in G(A[1/\pi])$ whose image in $G(A_{\red}[1/\pi])$ lies in $G(A_{\red})$.
We want to show that
the image of $g$ in $G(A_m[1/\pi])$ belongs to $G(A_m)$ for some $m \geq 1$.
Since $G$ is an affine scheme of finite type over $\O_E$, it suffices to prove that for an element $x \in A$ whose image $x_{\red} \in A_{\red}$ is divisible by $\pi$,
there exists an integer $m \geq 1$ such that the image
$x_{m} \in A_m$ is also divisible by $\pi$.
This is clear since $R^+_{\red}= \varinjlim_{m} R^+/(\varpi^\flat)^{1/q^m}$.
\end{proof}

We define
$B^{[n, \infty]}_S:=\O_{\mathcal{Y}^{[n, \infty]}_S}(\mathcal{Y}^{[n, \infty]}_S)$,
which is a $W_{\O_E}(R^+)[1/\pi]$-algebra.
The universal property of the rational open subset $\mathcal{Y}^{[n, \infty]}_S$
(by applying it to the affinoid ring $(A[1/\pi], A[1/\pi]^{\circ})$)
induces a natural homomorphism
\begin{equation}\label{equation:reduction homomorphism}
    B^{[n, \infty]}_S \to A[1/\pi]=(W_{\O_E}(R^+)/[\varpi^\flat])[1/\pi]
\end{equation}
such that the composition
$
W_{\O_E}(R^+)[1/\pi] \to B^{[n, \infty]}_S \to (W_{\O_E}(R^+)/[\varpi^\flat])[1/\pi]
$
coincides with the quotient map.

Let $\mathcal{P}_0$ be a $G$-Breuil--Kisin module over $(\O_{\breve{E}}, (\pi))$
and
let $\mathcal{P}$ be a $G$-Breuil--Kisin module over $(W_{\O_E}(R^+), I_{R^{\sharp+}})$.
We assume that $\mathcal{P}$ is trivial as a $G$-torsor.
We denote by $\mathcal{P}_A$
(resp.\ $(\mathcal{P}_0)_A$)
the base change of $\mathcal{P}$ (resp.\ $\mathcal{P}_0$) to $A$.
Let
$\widetilde{\mathcal{P}}_{\sht}$
be
the extended $G$-shtuka over $S$ with one leg at $S^\sharp$ associated with $\mathcal{P}$.
A quasi-isogeny
\[
\iota \colon (\widetilde{\mathcal{P}}_{\sht})_{\vert \mathcal{Y}^{[n, \infty]}_S} \overset{\sim}{\to} (\mathcal{P}_0)_{\vert \mathcal{Y}^{[n, \infty]}_S}
\]
from 
$\widetilde{\mathcal{P}}_{\sht}$
to $\mathcal{P}_0$ induces a quasi-isogeny
$
\eta_\iota \colon \mathcal{P}_A[1/\pi] \overset{\sim}{\to} (\mathcal{P}_0)_A[1/\pi]
$
by base change along $(\ref{equation:reduction homomorphism})$.
The following proposition is implicitly proved in \cite{Gleasongeometric}.

\begin{prop}[\cite{Gleasongeometric}]\label{Proposition:comparison of quasi-isogenies}
    The above construction $\iota \mapsto \eta_\iota$ induces a bijection between the set of
    quasi-isogenies $\iota$
    from 
    $\widetilde{\mathcal{P}}_{\sht}$
    to $\mathcal{P}_0$
    such that $(\eta_\iota)_{\red}$ is effective and the set of
    quasi-isogenies $\eta$
    from $\mathcal{P}_A$ to $(\mathcal{P}_0)_A$
    such that $\eta_{\red}$ is effective.
\end{prop}

\begin{proof}
We fix trivializations
$\mathcal{P} \simeq G_{W_{\O_E}(R^+)}$ and
$\mathcal{P}_0 \simeq G_{\O_{\breve{E}}}$.
We first prove the injectivity of the map.
Let $\iota$ and $\iota'$ be quasi-isogenies from 
$\widetilde{\mathcal{P}}_{\sht}$
to $\mathcal{P}_0$
defined over $\mathcal{Y}^{[n, \infty]}_S$
such that $(\eta_\iota)_{\red}$ and $(\eta_{\iota'})_{\red}$ are effective.
We assume that
$\eta_{\iota}=\eta_{\iota'}$.
We identify $\iota$ and $\iota'$ with elements $g, g' \in G(B^{[n, \infty]}_S)$, respectively.
Arguing as in \cite[Lemma 2.14]{Gleasongeometric},
we can find elements $h, h' \in G(W_{\O_E}(R^+))$ such that for some $r \geq n$ and $m \geq 1$, we have
$h=g$ and $h'=g'$ in
$G(B^{[r, \infty]}_S/[\varpi^\flat]^{1/q^m})$.
Since $\eta_{\iota}=\eta_{\iota'}$,
we see that
$h=h'$ in $G(W_{\O_E}(R^+)/[\varpi^\flat]^{1/q^{m}})$.
This implies that
$g=g'$ in $G(B^{[r, \infty]}_S/[\varpi^\flat]^{1/q^{m}})$.
It now follows from \cite[Lemma 2.15]{Gleasongeometric} (which also holds for general $\O_E$) that $g=g'$ in $G(B^{[r, \infty]}_S)$, whence $\iota=\iota'$.

Let $\eta \colon \mathcal{P}_A[1/\pi] \overset{\sim}{\to} (\mathcal{P}_0)_A[1/\pi]$
be a quasi-isogeny
such that $\eta_{\red}$ is effective.
By applying \cite[Lemma 2.15]{Gleasongeometric} to the base change of $\eta$
along
$A[1/\pi] \to B^{[n, \infty]}_S/[\varpi^\flat]$,
we can find a
quasi-isogeny
$\iota$
from
$\widetilde{\mathcal{P}}_{\sht}$
to $\mathcal{P}_0$
with $\eta_\iota=\eta$.
This proves that the map is surjective.
\end{proof}

\subsection{Formal completions of integral local Shimura varieties}\label{Subsection:Moduli spaces of $G$-shtukas}

Let
$\mu \colon \G_m \to G_{\O_{\breve{E}}}$
be a minuscule cocharacter.
Let
$B(G)$ be the set of $\phi$-conjugacy classes of $G(\breve{E})$ and let
$B(G, \mu^{-1}) \subset B(G)$ be the subset of 
neutral acceptable elements for $\mu^{-1}$; see \cite[Section 2.4.1]{PappasRapoport21}.
Let $b \in G(\breve{E})$ be an element such that its $\phi$-conjugacy class belongs to $B(G, \mu^{-1})$.
Let
$\mathcal{E}_b=G_{\O_{\breve{E}}}$ be the $G$-Breuil--Kisin module over
$(\O_{\breve{E}}, (\pi))$
with the Frobenius given by $g \mapsto bg$.

\begin{defn}[{\cite[Definition 25.1.1]{Scholze-Weinstein}, \cite[Definition 3.2.1]{PappasRapoport21}}]\label{Definition:integral model of local Shimura varieties}
    The \textit{integral moduli space of local shtukas}
    \[
    \mathcal{M}^{\int}_{G, b, \mu}
    \]
    associated with the triple
    $(G, b, \mu)$
    is the $v$-sheaf on $\Perf_k$
    which sends an affinoid perfectoid space
    $S=\Spa(R, R^+)$ over $k$ to the set of isomorphism classes of tuples
    \[
    (S^{\sharp}, \mathscr{P}, \iota)
    \]
    consisting of an untilt
    $S^{\sharp}=\Spa(R^{\sharp}, R^{\sharp+})$ of $S$ over $\O_{\breve{E}}$,
    a $G$-shtuka
    $\mathscr{P}$
    over $S$ with one leg at $S^\sharp$ which is bounded by $\mu$, and
    a quasi-isogeny $\iota$ from $\mathscr{P}$ to $\mathcal{E}_b$.
    We also call $\mathcal{M}^{\int}_{G, b, \mu}$
    the \textit{integral local Shimura variety} (with hyperspecial level structure).
    To see that $\mathcal{M}^{\int}_{G, b, \mu}$ is a $v$-sheaf, we can use \cite[Proposition 19.5.3]{Scholze-Weinstein}.
    (Although only the case of $\O_E=\Z_p$ is considered in \cite[Proposition 19.5.3]{Scholze-Weinstein}, the same argument works for general $\O_E$.)
\end{defn}

We have a morphism
$\mathcal{M}^{\int}_{G, b, \mu} \to \Spd \O_{\breve{E}}$
defined by $(S^{\sharp}, \mathscr{P}, \iota) \mapsto S^{\sharp}$.
Namely, $\mathcal{M}^{\int}_{G, b, \mu}$ is a $v$-sheaf over $\Spd \O_{\breve{E}}$.
The generic fiber of $\mathcal{M}^{\int}_{G, b, \mu}$ is (the $v$-sheaf associated with)
the local Shimura variety
$\mathcal{M}_{G_E, b, \mu, K}$ over $\Spa (\breve{E}, \O_{\breve{E}})$, where $K:=G(\O_E)$.

\begin{rem}\label{Remark:product of points cases}
    Let
    $S=\Spa(R, R^+)$
    be a product of points over $k$
    and let
    $(S^{\sharp}, \mathscr{P}, \iota) \in \mathcal{M}^{\int}_{G, b, \mu}(S)$.
    By Remark \ref{Remark:extending along quasi-isogenies}, we can uniquely extend $\mathscr{P}$
    to
    an extended $G$-shtuka
    $\widetilde{\mathscr{P}}$ over $S$ with one leg at $S^\sharp$
    and regard $\iota$
    as
    a quasi-isogeny
    from
    $\widetilde{\mathscr{P}}$ to $\mathcal{E}_b$.
    By Corollary \ref{Corollary:equivalence between G BK module and extended G shtuks},
    Proposition \ref{Proposition:bounded by mu and of type mu},
    and Proposition \ref{Proposition:G-displays to G-BK}, there is a unique (up to isomorphism) $G$-$\mu^{-1}$-display $\mathcal{Q}$ over
    $(W_{\O_E}(R^+), I_{R^{\sharp+}})$ such that
    $
    \widetilde{(\mathcal{Q}_{\mathrm{BK}})}_{\sht} \simeq \widetilde{\mathscr{P}}.
    $
    We note $\mathcal{Q}$ is banal by Proposition \ref{Proposition:G display with trivial Hodge filtration is banal}.
    (One can show that the Hodge filtration $P(\mathcal{Q})_{R^{\sharp+}}$ is trivial as a $P_\mu$-torsor over $\Spec R^{\sharp+}$ arguing as in the proof of \cite[Proposition 5.6.11]{Ito-K23}.)
    Let $\mathcal{Q}_{\red}$ be the base change of $\mathcal{Q}$ along
    $(W_{\O_E}(R^+), I_{R^{\sharp+}}) \to (W_{\O_E}(R^+_{\red}), (\pi))$.
    Then $\iota$ induces a quasi-isogeny
    \[
    (\eta_\iota)_{\red} \colon (\mathcal{Q}_{\red})_{\mathrm{BK}}[1/\pi] \overset{\sim}{\to} \mathcal{E}_b[1/\pi]
    \]
    as in Proposition \ref{Proposition:comparison of quasi-isogenies}.
    Here the base change of $\mathcal{E}_b$ to $(W_{\O_E}(R^+_{\red}), (\pi))$ is denoted by the same symbol.
\end{rem}

Let $\mathcal{Q}_0$ be a $G$-$\mu^{-1}$-display over $(\O_{\breve{E}}, (\pi))$ together with
a quasi-isogeny
\[
\iota_0 \colon (\mathcal{Q}_0)_{\mathrm{BK}}[1/\pi] \overset{\sim}{\to} \mathcal{E}_b[1/\pi].
\]
We note that $\mathcal{Q}_0$ is banal by Proposition \ref{Proposition:G display with trivial Hodge filtration is banal}.
In particular
$(\mathcal{Q}_0)_{\mathrm{BK}}$ is trivial as a $G$-torsor over $\Spec \O_{\breve{E}}$.
We fix
a trivialization
$(\mathcal{Q}_0)_{\mathrm{BK}} \simeq G_{\O_{\breve{E}}}$.
Then
$\iota_0$
can be regarded as an element $g \in G(\breve{E})$
such that
$g^{-1}b\phi(g) \in G(\O_{\breve{E}})\mu(\pi)^{-1} G(\O_{\breve{E}})$.
The image of $g$ in $G(\breve{E})/G(\O_{\breve{E}})$
does not depend on the choice of $(\mathcal{Q}_0)_{\mathrm{BK}} \simeq G_{\O_{\breve{E}}}$.
In this way, we obtain a point
\[
x \in X_{G}(b, \mu^{-1}):= \{ \, g \in G(\breve{E})/G(\O_{\breve{E}})   \, \vert \,  g^{-1}b\phi(g) \in G(\O_{\breve{E}})\mu(\pi)^{-1} G(\O_{\breve{E}})  \, \}
\]
of the affine Deligne--Lusztig variety $X_{G}(b, \mu^{-1})$.

\begin{defn}[{\cite{Gleason}}]\label{Definition:formal completion}
The \textit{formal completion}
\[
\widehat{\mathcal{M}^{\int}_{G, b, \mu}}_{/ x}
\]
of $\mathcal{M}^{\int}_{G, b, \mu}$ at $x$ is the subsheaf of $\mathcal{M}^{\int}_{G, b, \mu}$ which sends
a product of points $S=\Spa(R, R^+)$ over $k$ to the set of isomorphism classes of tuples
$
(S^{\sharp}, \mathscr{P}, \iota) \in \mathcal{M}^{\int}_{G, b, \mu}(S)
$
such that, for the associated $G$-$\mu^{-1}$-display $\mathcal{Q}$ over
$(W_{\O_E}(R^+), I_{R^{\sharp+}})$, there exists an isomorphism
$\mathcal{Q}_{\red} \overset{\sim}{\to} \mathcal{Q}_{0}$
of $G$-$\mu^{-1}$-displays over $(W_{\O_E}(R^+_{\red}), (\pi))$ which makes the following diagram commute:
 \[
\xymatrix{
(\mathcal{Q}_{\red})_{\mathrm{BK}}[1/\pi]   \ar[rr]^-{} \ar[rd]_-{(\eta_\iota)_{\red}} & &    \ar[ld]^-{\iota_0} (\mathcal{Q}_{0})_{\mathrm{BK}}[1/\pi] \\
    &  \mathcal{E}_b[1/\pi]. &
    }
\]
See Remark \ref{Remark:product of points cases} for the notation used here.
As in Remark \ref{Remark:product of points cases}, the base change of $(\mathcal{Q}_{0})_{\mathrm{BK}}$ to $(W_{\O_E}(R^+_{\red}), (\pi))$ is denoted by the same symbol.
\end{defn}

\begin{rem}\label{Remark:definition of formal completion makes sense}
It is easy to see that
the subsheaf $\widehat{\mathcal{M}^{\int}_{G, b, \mu}}_{/ x}$ exists and is uniquely determined by the above condition (since products of points over $k$ form a basis of $\Perf_k$ with respect to the $v$-topology).
In \cite{Gleason},
Gleason gives a more conceptual definition of $\widehat{\mathcal{M}^{\int}_{G, b, \mu}}_{/ x}$.
With the notation of \cite{Gleason},
$\widehat{\mathcal{M}^{\int}_{G, b, \mu}}_{/ x}$
is the formal neighborhood 
of $x \in (\mathcal{M}^{\int}_{G, b, \mu})^{\mathrm{red}}$
on the prekimberlite
$\mathcal{M}^{\int}_{G, b, \mu}$; see \cite[Definition 4.18]{Gleason}.
We note that $(\mathcal{M}^{\int}_{G, b, \mu})^{\mathrm{red}}$ can be identified with the affine Deligne--Lusztig variety $X_{G}(b, \mu^{-1})$ by \cite[Proposition 2.30]{Gleasongeometric}.
See also \cite[Lemma 2.31]{Gleasongeometric}
for the fact that $\mathcal{M}^{\int}_{G, b, \mu}$ is a prekimberlite (in the sense of \cite[Definition 4.15]{Gleason}).
\end{rem}

We consider
the local ring
$R_{G, \mu^{-1}}$
over $\O_{\breve{E}}$
defined as in Definition \ref{Definition:universal local ring}.
There is an isomorphism
$
R_{G, \mu^{-1}} \simeq \O_{\breve{E}}[[t_1, \dotsc, t_r]]
$
over $\O_{\breve{E}}$.
We endow $R_{G, \mu^{-1}}$ with the $\mathfrak{m}_{G, \mu^{-1}}$-adic topology.
Let
$
\Spd R_{G, \mu^{-1}}
$
be the $v$-sheaf (over $\Spd \O_{\breve{E}}$) on $\Perf_k$ which sends a perfectoid space $S$ over $k$ to the set of isomorphism classes of
untilts
$S^\sharp$ of $S$ over $\O_{\breve{E}}$
equipped with a morphism of adic spaces $S^\sharp \to \Spa R_{G, \mu^{-1}}$
over $\O_{\breve{E}}$; see \cite[Section 18.4]{Scholze-Weinstein}.

The main result of this section is the following theorem.

\begin{thm}\label{Theorem:conjecture of Pappas-Rapoport}
    There exists an isomorphism
    \[
    \widehat{\mathcal{M}^{\int}_{G, b, \mu}}_{/ x} \simeq \Spd R_{G, \mu^{-1}}
    \]
    of $v$-sheaves over $\Spd \O_{\breve{E}}$.
\end{thm}

\begin{proof}
Let
$S=\Spa(R, R^+)$
be a product of points over $k$
and let
$S^{\sharp}=\Spa(R^{\sharp}, R^{\sharp+})$
be
an untilt of $S$ over $\O_{\breve{E}}$.
Then the inverse image of
$S^{\sharp} \in (\Spd \O_{\breve{E}})(S)$
under the map
$(\Spd R_{G, \mu^{-1}})(S) \to (\Spd \O_{\breve{E}})(S)$
can be identified with the set of continuous homomorphisms $R_{G, \mu^{-1}} \to R^{\sharp+}$ over $\O_{\breve{E}}$.

Let $\varpi^{\flat} \in R^+$ be a pseudo-uniformizer such that $\pi \in (\varpi)$ in $R^{\sharp+}$, where $\varpi:=\theta([\varpi^\flat]) \in R^{\sharp+}$.
The inverse image of
$S^{\sharp} \in (\Spd \O_{\breve{E}})(S)$
under the map
$\widehat{\mathcal{M}^{\int}_{G, b, \mu}}_{/ x}(S) \to (\Spd \O_{\breve{E}})(S)$
can be identified with the set of isomorphism classes of pairs
$
(\mathcal{Q}, \eta)
$
consisting of
a banal $G$-$\mu^{-1}$-display $\mathcal{Q}$ over $(W_{\O_E}(R^+), I_{R^{\sharp+}})$
and an element $\eta$ of the set
\[
{\varinjlim}_m \{ \, \text{isomorphisms} \ \mathcal{Q}_m \overset{\sim}{\to} (\mathcal{Q}_0)_m \ \text{of} \ G\mathchar`-\mu^{-1}\mathchar`-\text{displays over} \ (A_m, (\pi)) \, \},
\]
where
$(A_m, (\pi))=(W_{\O_E}(R^+)/[\varpi^\flat]^{1/q^m}, (\pi))$,
and $\mathcal{Q}_m$ (resp.\ $(\mathcal{Q}_0)_m$) is the base change of $\mathcal{Q}$ (resp.\ $\mathcal{Q}_0$) to $(A_m, (\pi))$.
This follows by combining
Proposition \ref{Proposition:G-displays to G-BK},
Proposition \ref{Proposition:bounded by mu and of type mu},
Corollary \ref{Corollary:equivalence between G BK module and extended G shtuks},
Lemma \ref{Lemma:quasi-isogeny reduction equivalence},
Lemma \ref{Lemma:quasi-isogeny reduction effective}, and Proposition \ref{Proposition:comparison of quasi-isogenies}.

By Theorem \ref{Theorem:Existence of universal deformation},
there exists a universal deformation $\mathfrak{Q}^{\univ}$ of $\mathcal{Q}_0$ over $R_{G, \mu^{-1}}$.
Moreover $\mathfrak{Q}^{\univ}$ has the property $(\mathrm{Perfd})$.
Let
$g \colon R_{G, \mu^{-1}} \to R^{\sharp+}$
be a continuous homomorphism over $\O_{\breve{E}}$.
We regard $(W_{\O_E}(R^+), I_{R^{\sharp+}})$ as an object of $(R_{G, \mu^{-1}})_{\Prism, \O_E}$ via the homomorphism $g$, and let
$\mathfrak{Q}^{\univ}_g$ be the associated $G$-$\mu^{-1}$-display over $(W_{\O_E}(R^+), I_{R^{\sharp+}})$.
For some $m \geq 1$,
we have
$g(\mathfrak{m}_{G, \mu^{-1}}) \subset (\varpi^{1/q^m})$
in $R^{\sharp+}$, where $\varpi^{1/q^m}:=\theta([\varpi^\flat]^{1/q^m}) \in R^{\sharp+}$.
Hence there is a natural isomorphism
$
\eta_g \colon (\mathfrak{Q}^{\univ}_g)_m \overset{\sim}{\to} (\mathcal{Q}_0)_m
$
over $(A_m, (\pi))$.
We define a map
\[
(\Spd R_{G, \mu^{-1}})(S) \times_{(\Spd \O_{\breve{E}})(S)} \{ S^\sharp \} \to \widehat{\mathcal{M}^{\int}_{G, b, \mu}}_{/ x}(S) \times_{(\Spd \O_{\breve{E}})(S)} \{ S^\sharp \}
\]
by $g \mapsto (\mathfrak{Q}^{\univ}_g, \eta_g)$.
Using the property $(\mathrm{Perfd})$ and using Proposition \ref{Proposition:limit of G displays},
we see that this map is bijective.
Then we obtain a bijection
$
(\Spd R_{G, \mu^{-1}})(S) \overset{\sim}{\to} \widehat{\mathcal{M}^{\int}_{G, b, \mu}}_{/ x}(S)
$
over $(\Spd \O_{\breve{E}})(S)$,
which is independent of the choice of $\varpi^\flat$ and is functorial in $S$.
Since products of points over $k$ form a basis of $\Perf_k$,
we finally get the desired isomorphism
$
\Spd R_{G, \mu^{-1}} \overset{\sim}{\to} \widehat{\mathcal{M}^{\int}_{G, b, \mu}}_{/ x}
$
of $v$-sheaves over $\Spd \O_{\breve{E}}$.
\end{proof}


In the following, we assume that $\O_E=\Z_p$ for simplicity.

\begin{rem}\label{Remark:comparison with Bartling}
Theorem \ref{Theorem:conjecture of Pappas-Rapoport} implies that a conjecture of Pappas--Rapoport \cite[Conjecture 3.3.5]{PappasRapoport21}
in the hyperspecial case, which was originally proposed by Gleason \cite[Conjecture 1]{Gleasonthesis}, holds true.
This result was already known in the following cases.
\begin{enumerate}
\item In \cite{PappasRapoport22}, Pappas--Rapoport proved Theorem \ref{Theorem:conjecture of Pappas-Rapoport} under the assumption that $p \geq 3$ and the pair $(G_{\Q_p}, \mu)$ is of abelian type in the sense of \cite[Definition 2.1.3]{PappasRapoport22}.
In fact, the result is formulated and proved in the more general case where $G$ is parahoric.
If $p=2$, the same result is obtained under the additional assumption that $G_{\Q_p}$ is of type $A$ or $C$.

\item In \cite{Bartling},
Bartling proved Theorem \ref{Theorem:conjecture of Pappas-Rapoport} under the assumption that the element $b$ satisfies the \textit{adjoint nilpotent condition} introduced in \cite[Definition 3.4.2]{Bultel-Pappas}
and $p \geq 3$.
He studied the relation between  $\widehat{\mathcal{M}^{\int}_{G, b, \mu}}_{/ x}$
and the deformation space constructed by B\"ultel--Pappas \cite[3.5.9]{Bultel-Pappas}.
\end{enumerate}
\end{rem}

\begin{rem}\label{Remark:global representability}
Scholze conjectured that there exists a (unique) normal formal scheme $\mathscr{M}_{G, b, \mu}$ which is flat and locally formally of finite type over $\Spf W(k)$ such that the $v$-sheaf associated with $\mathscr{M}_{G, b, \mu}$ is isomorphic to $\mathcal{M}^{\int}_{G, b, \mu}$.
(This conjecture is stated in the more general case where $G$ is parahoric.)
If $(G, b, \mu)$ is of EL type or PEL type, then this is proved in \cite[Lecture 25]{Scholze-Weinstein}.
Moreover, in the case (1) of Remark \ref{Remark:comparison with Bartling}, this is proved in \cite{PappasRapoport22} using Theorem \ref{Theorem:conjecture of Pappas-Rapoport}.
Since we have proved Theorem \ref{Theorem:conjecture of Pappas-Rapoport} for any $p$ (including $p=2$), we can prove that
if $(G_{\Q_p}, \mu)$ is of abelian type,
then
$\mathscr{M}_{G, b, \mu}$
exists
for any $p$ in the same way as in \cite{PappasRapoport22}.

If $\mathscr{M}_{G, b, \mu}$ exists, then it is formally smooth over $\Spf W(k)$
by Theorem \ref{Theorem:conjecture of Pappas-Rapoport}.
\end{rem}

\section{Comparison with universal deformations of $p$-divisible groups}\label{Section:Comparison with universal deformations of $p$-divisible groups}

In this section, we assume that $\O_E=\Z_p$.
Let $k$ be a perfect field of characteristic $p$.
Let $\mathcal{G}$ be a $p$-divisible group over $\Spec k$
of height $N$ and of dimension $d$.
Let
$\mu \colon \G_m \to \GL_N$
be the cocharacter over $W(k)$
defined by
\[
t \mapsto \diag{(\underbrace{t, \dotsc, t}_{N-d}, \underbrace{1, \dotsc, 1}_{d})}.
\]
Let
$\mathcal{Q}$
be the $\GL_N$-$\mu$-display over the prism $(W(k), (p))$ associated with
the Dieudonn\'e module of $\mathcal{G}$.
For a universal deformation $\mathcal{G}^{\univ}$ of $\mathcal{G}$,
we prove that the prismatic Dieudonn\'e crystal of $\mathcal{G}^{\univ}$, as introduced in \cite{Anschutz-LeBras}, induces a universal deformation $\mathfrak{Q}^{\univ}$ of $\mathcal{Q}$; see Theorem \ref{Theorem:universal family of p-divisible groups gives a universal family of displays}.
Then, by relating the properties $\mathrm{(Perfd)}$ and $\mathrm{(BK)}$ of the universal deformation $\mathfrak{Q}^{\univ}$ to the universal property of $\mathcal{G}^{\univ}$, we establish some classification results of $p$-divisible groups; see Section \ref{Subsection:Classifications of p-divisible groups} for details.

\subsection{Prismatic Dieudonn\'e crystals of $p$-divisible groups}\label{Subsection:Prismatic Diedonn\'e modules of p-divisible groups}

In this section and Section \ref{Section:Consequences on prismatic F-gauges} below, we need
the notion of \textit{quasisyntomic rings}
introduced in \cite[Definition 4.10]{BMS2}.
We refer to \cite[Section 4]{BMS2} for basic properties of quasisyntomic rings.

\begin{ex}\label{Example:quasisyntomic rings}
\ 
\begin{enumerate}
    \item Let
    $R \in \mathcal{C}_{W(k)}$
    be a complete regular local ring over $W(k)$ with residue field $k$.
    Then $R$ is a quasisyntomic ring.
    Moreover $R/\mathfrak{m}^m_R$
    is a quasisyntomic ring if $\dim R \leq 1$.
    See \cite[Example 3.17]{Anschutz-LeBras}.
    \item Let $(S, a^\flat)$ be a perfectoid pair (Definition \ref{Definition:perfectoid pair}).
    Then $S$ and $S/a^m$ are quasisyntomic rings.
    More precisely, $S$ and $S/a^m$ are \textit{quasiregular semiperfectoid rings} in the sense of \cite[Definition 4.20]{BMS2}; see \cite[Example 3.20]{Anschutz-LeBras}.
\end{enumerate}
\end{ex}

Let $R$ be a $p$-adically complete ring.
Let $\mathscr{G}$ be a $p$-divisible group over $\Spec R$.
Following \cite{Anschutz-LeBras},
we consider the sheaf
\[
\mathcal{E}xt^1_{(R)_\Prism}(\underline{\mathscr{G}}, \O_\Prism)
\]
on the site $(R)^{\op}_\Prism$, where
$\underline{\mathscr{G}}$ is the sheaf defined by
$(A, I) \mapsto \mathscr{G}(A/I)$
and $\O_\Prism$ is defined by $(A, I) \mapsto A$.
In \cite[Proposition 4.69]{Anschutz-LeBras}, it is proved that 
$\mathcal{E}xt^1_{(R)_\Prism}(\underline{\mathscr{G}}, \O_\Prism)(A, I)$
is a finite projective $A$-module
for any $(A, I) \in (R)_\Prism$
and its formation commutes with base change along any morphism $(A, I) \to (A', I')$ in $(R)_\Prism$.
By \cite[Remark 7.3.1]{Ito-K23},
there is a canonical isomorphism
\[
\mathrm{Ext}^1_{(A, I)_\Prism}(\underline{\mathscr{G}}, \O_\Prism) \overset{\sim}{\to} \mathcal{E}xt^1_{(R)_\Prism}(\underline{\mathscr{G}}, \O_\Prism)(A, I)
\]
for any $(A, I) \in (R)_\Prism$, where we view the site 
$(A, I)^{\op}_\Prism$
as the localization of $(R)^{\op}_\Prism$ at $(A, I)$, and the restriction of $\underline{\mathscr{G}}$ to $(A, I)^{\op}_{\Prism}$ is denoted by the same symbol.
We set
\[
M_{\Prism}(\mathscr{G})(A, I):=\mathcal{E}xt^1_{(R)_\Prism}(\underline{\mathscr{G}}, \O_\Prism)(A, I).
\]

In \cite[Theorem 4.71]{Anschutz-LeBras}, it is proved that if $R$ is a quasisyntomic ring, then $\mathcal{E}xt^1_{(R)_\Prism}(\underline{\mathscr{G}}, \O_\Prism)$
is a \textit{prismatic Dieudonn\'e crystal} on $(R)^{\op}_\Prism$ in the sense that 
for any $(A, I) \in (R)_\Prism$,
the $A$-module
$M_{\Prism}(\mathscr{G})(A, I)=\mathcal{E}xt^1_{(R)_\Prism}(\underline{\mathscr{G}}, \O_\Prism)(A, I)$
with
the $A$-linear homomorphism
\[
\phi^*M_{\Prism}(\mathscr{G})(A, I) \to M_{\Prism}(\mathscr{G})(A, I)
\]
induced by the Frobenius $\phi \colon \O_\Prism \to \O_\Prism$
is a minuscule Breuil--Kisin module over $(A, I)$, and its formation commutes with base change along any morphism $(A, I) \to (A', I')$ in $(R)_\Prism$.
(See also \cite[Section 7.3]{Ito-K23}.)

\begin{rem}\label{Remark:definition of prismatic Dieudonne crystal in ALB}
We assume that $R$ is a quasisyntomic ring.
By \cite[Proposition 4.4]{Anschutz-LeBras} (see also \cite[Proposition 2.7]{BS2}),
the category of prismatic Dieudonn\'e crystals on $(R)^{\op}_\Prism$ is equivalent to that of prismatic Dieudonn\'e crystals over $R$ in the sense of \cite[Definition 4.5]{Anschutz-LeBras}, which are defined as sheaves on the quasisyntomic site of $R$.
Let $\mathrm{DM}(R)$ denote the category of prismatic Dieudonn\'e crystals on $(R)^{\op}_\Prism$.
We have an identification
\[
\mathrm{DM}(R) = {2-\varprojlim}_{(A, I) \in (R)_{\Prism}} \mathrm{BK}_{\mathrm{min}}(A, I)
\]
where
$
\mathrm{BK}_{\mathrm{min}}(A, I)
$
is the category of minuscule Breuil--Kisin modules over $(A, I)$.
\end{rem}

\begin{ex}\label{Example:comparison with crystalline cohomology}
    Let $\mathfrak{S}:=W(k)[[t_1, \dotsc, t_n]]$ and $
\mathfrak{S}_m:=W(k)[[t_1, \dotsc, t_n]]/(t_1, \dotsc, t_n)^{m}
$.
We assume that $n \leq 1$ (so that $\mathfrak{S}_m/p$ is a quasisyntomic ring by Example \ref{Example:quasisyntomic rings}).
Let $\mathscr{G}$ be a $p$-divisible group over $\Spec \mathfrak{S}_m/p$.
    Then there exists a natural isomorphism
    \[
        \phi^*M_\Prism(\mathscr{G})(\mathfrak{S}_m, (p)) \simeq \D(\mathscr{G})(\mathfrak{S}_m \to \mathfrak{S}_m/p)
    \]
    of minuscule Breuil--Kisin modules over $(\mathfrak{S}_m, (p))$, where
    $\D(\mathscr{G})(\mathfrak{S}_m \to \mathfrak{S}_m/p)$
    is the evaluation on the divided power extension
    $\mathfrak{S}_m \to \mathfrak{S}_m/p$
    of the contravariant Dieudonn\'e crystal $\mathbb{D}(\mathscr{G})$ defined in \cite[D\'efinition 3.3.6]{BBM}.
    This is a special case of \cite[Lemma 4.45]{Anschutz-LeBras}.
\end{ex}

\subsection{Universal deformations of $p$-divisible groups}\label{Subsection:Universal deformations of $p$-divisible groups}

Let $\mathcal{G}$ be a $p$-divisible group over $\Spec k$.
Let $N$ be the height of $\mathcal{G}$ and let $d$ be the dimension of $\mathcal{G}$.

Let $\Art_{W(k)}$ be the category of artinian local rings $R$ with a local homomorphism $W(k) \to R$ which induces an isomorphism on the residue fields.
By \cite[Corollaire 4.8]{Illusie} (or \cite[Proposition 3.11]{Lau14}),
the functor $\Art_{W(k)} \to \mathrm{Set}$
sending $R$ to the set of isomorphism classes of deformations of $\mathcal{G}$ over $\Spec R$ is pro-representable by
\[
R^{\univ}:=W(k)[[t_1, \dotsc, t_{d(N-d)}]].
\]
Let $\mathcal{G}^{\univ}$ be a universal deformation of $\mathcal{G}$ over $\Spec R^{\univ}$.

Let
$\mu \colon \G_m \to \GL_N$
be the cocharacter over $W(k)$
defined as in the beginning of this section.
The minuscule Breuil--Kisin module
$M_\Prism(\mathcal{G}^{\univ})(A, I)$ is of type $\mu$ for every $(A, I) \in (R^{\univ})_\Prism$.
Via the equivalence
$
\mathrm{BK}_\mu(A, I)^{\simeq} \overset{\sim}{\to} \GL_N\mathchar`-\mathrm{Disp}_\mu(A, I)
$
in Example \ref{Example:GLn displays},
the prismatic Dieudonn\'e crystal
$\mathcal{E}xt^1_{(R^{\univ})_\Prism}(\underline{\mathcal{G}}^{\univ}, \O_\Prism)$
induces a
prismatic $\GL_N$-$\mu$-display
$
\mathfrak{Q}(\mathcal{G}^{\univ})
$
over $R^{\univ}$.
We set $\mathcal{Q}:=\mathfrak{Q}(\mathcal{G}^{\univ})_{(W(k), (p))}$, which corresponds to $M_\Prism(\mathcal{G})(W(k), (p))$.

\begin{thm}\label{Theorem:universal family of p-divisible groups gives a universal family of displays}
    The prismatic $\GL_N$-$\mu$-display
    $
    \mathfrak{Q}(\mathcal{G}^{\univ})
    $
    over $R^{\univ}$ is a universal deformation of $\mathcal{Q}$.
    Moreover
    $
    \mathfrak{Q}(\mathcal{G}^{\univ})
    $
    has the properties $\mathrm{(Perfd)}$ and $\mathrm{(BK)}$.
\end{thm}

\begin{proof}
Since
$R_{\GL_N, \mu} \simeq R^{\univ}$,
it suffices to prove that $\mathfrak{Q}(\mathcal{G}^{\univ})$ is versal by Theorem \ref{Theorem:characterization of universal family}.
    We consider the period map
    \[
    \Per_{s^*\mathcal{Q}} \colon \Def(\mathcal{Q})_{(W(k)[[t]]/t^2, (p))} \to
    \Lift(P(\mathcal{Q})_{k}, (s^*\mathcal{Q})_{k[[t]]/t^2})
    \]
    where
    $s^*\mathcal{Q}$
    is the base change of $\mathcal{Q}$ along the natural map
    \[
    s \colon (W(k), (p)) \to (W(k)[[t]]/t^2, (p)).
    \]
    The period map $\Per_{s^*\mathcal{Q}}$ is bijective by Theorem \ref{Theorem:GM deformation}.
    Let $\Def(\mathcal{G})_{k[[t]]/t^2}$ be the set of isomorphism classes of deformations of $\mathcal{G}$ over $\Spec k[[t]]/t^2$.
    By the universal property of $\mathcal{G}^{\univ}$, 
    we may identify $\Def(\mathcal{G})_{k[[t]]/t^2}$ with $\mathfrak{t}_{R^{\univ}}$.
    It thus suffices to show that the following composition is bijective:
    \begin{equation}\label{equation:kodaira-spencer period map composition}
        \Def(\mathcal{G})_{k[[t]]/t^2} \to \Def(\mathcal{Q})_{(W(k)[[t]]/t^2, (p))} \overset{\sim}{\to} \Lift(P(\mathcal{Q})_{k}, (s^*\mathcal{Q})_{k[[t]]/t^2}).
    \end{equation}

    Let $\mathscr{G} \in \Def(\mathcal{G})_{k[[t]]/t^2}$.
    By Example \ref{Example:comparison with crystalline cohomology},
    we have
    \[
    \phi^*M_{\Prism}(\mathscr{G})(W(k)[[t]]/t^2, (p)) \simeq \D(\mathscr{G})(W(k)[[t]]/t^2 \to k[[t]]/t^2)
    \]
    and
    \[
    \phi^*M_{\Prism}(\mathcal{G})(W(k), (p)) \simeq \D(\mathcal{G})(W(k)):=\D(\mathcal{G})(W(k) \to k).
    \]
    It follows from Example \ref{Example:phi GLn torsor and Hodge filtration}
    and
    Proposition \ref{Proposition:canonical isomorphism of deformations of phi G torsors} that
    there exists a unique Frobenius equivariant isomorphism
    \[
    \psi \colon \D(\mathscr{G})(W(k)[[t]]/t^2 \to k[[t]]/t^2) \overset{\sim}{\to} \D(\mathcal{G})(W(k)) \otimes_{W(k)} W(k)[[t]]/t^2
    \]
    which is a lift of the identity of $\D(\mathcal{G})(W(k))$.
    We shall give a different description of $\psi$.
    The kernel of $W(k)[[t]]/t^2 \to k$ has a unique divided power structure which is compatible with
    the usual divided power structure on $(p)$ and the trivial divided power structure on $(t)$.
    Then we have a chain of canonical isomorphisms
    \begin{align*}
 \D(\mathscr{G})(W(k)[[t]]/t^2 \to k[[t]]/t^2) &\simeq \D(\mathscr{G})(W(k)[[t]]/t^2 \to k)  \\
         & \simeq \D(\mathcal{G})(W(k)[[t]]/t^2 \to k) \\
         & \simeq \D(\mathcal{G})(W(k)) \otimes_{W(k)} W(k)[[t]]/t^2,
         \end{align*}
    where the first and third isomorphisms are induced by the crystals $\D(\mathscr{G})$ and $\D(\mathcal{G})$ and the second one is induced by base change.
    The composition, denoted by $\psi'$, is Frobenius equivariant and is a lift of the identity of $\D(\mathcal{G})(W(k))$.
    Then the uniqueness part of Proposition \ref{Proposition:canonical isomorphism of deformations of phi G torsors} ensures that $\psi=\psi'$.
    In particular we see that the reduction modulo $p$ of $\psi$
    coincides with the composition of canonical isomorphisms
    \[
    \D(\mathscr{G})(k[[t]]/t^2 \to k) \simeq \D(\mathcal{G})(k[[t]]/t^2 \to k) \simeq \D(\mathcal{G})(k) \otimes_k k[[t]]/t^2,
    \]
    where $\D(\mathcal{G})(k):=\D(\mathcal{G})(k \to k)$.
    
    Let $P^1 \subset \D(\mathcal{G})(k)$ be the Hodge filtration.
    We may identify
    $\Lift(P(\mathcal{Q})_{k}, (s^*\mathcal{Q})_{k[[t]]/t^2})$ with the set of lifts $\mathscr{P} \subset \D(\mathcal{G})(k) \otimes_k k[[t]]/t^2$ of the Hodge filtration $P^1$.
    The above argument shows that, under this identification,
    the composition (\ref{equation:kodaira-spencer period map composition}) coincides with the usual period map
    in the Grothendieck--Messing deformation theory for $p$-divisible groups.
    In particular (\ref{equation:kodaira-spencer period map composition}) is bijective; see \cite[Chapter V, Theorem 1.6]{Messing}.

    The proof of Theorem \ref{Theorem:universal family of p-divisible groups gives a universal family of displays} is complete.
\end{proof}

\subsection{Classifications of $p$-divisible groups}\label{Subsection:Classifications of p-divisible groups}

For a ring $R$, let $\mathrm{BT}(R)$ be the category of $p$-divisible groups over $\Spec R$.
Recall that
$
\mathrm{BK}_{\mathrm{min}}(A, I)
$
is the category of minuscule Breuil--Kisin modules over a bounded prism $(A, I)$.

Let
$R \in \mathcal{C}_{W(k)}$
and let
$
(\mathfrak{S}, (\mathcal{E})):=(W(k)[[t_1, \dotsc, t_n]], (\mathcal{E}))
$
be a prism of Breuil--Kisin type
equipped with an isomorphism
$R \simeq \mathfrak{S}/\mathcal{E}$
over $W(k)$.
We set
$
\mathfrak{S}_m:=W(k)[[t_1, \dotsc, t_n]]/(t_1, \dotsc, t_n)^{m}
$
and
$R_m:=R/\mathfrak{m}^m_R$.
We remark that
$R_m$ is not a quasisyntomic ring if $m \geq 2$ and $\dim R \geq 2$.
However, we still have the following result.

\begin{lem}\label{Lemma: minuscule BK module for quotients of regular local rings}
    Let
    $R \in \mathcal{C}_{W(k)}$
    and let $\mathscr{G}$ be a $p$-divisible group over $\Spec R_m$ for some $m \geq 1$.
    Then $M_{\Prism}(\mathscr{G})(\mathfrak{S}_m, (\mathcal{E}))$
    is a minuscule Breuil--Kisin module over $(\mathfrak{S}_m, (\mathcal{E}))$.
\end{lem}

\begin{proof}
    There is a $p$-divisible group $\mathscr{G}'$ over $\Spec R$ such that $\mathscr{G}' \times_{\Spec R} \Spec R_m \simeq \mathscr{G}$.
    (See \cite[Corollaire 4.8]{Illusie} for example.)
    Then
    $M_{\Prism}(\mathscr{G}')(\mathfrak{S}, (\mathcal{E}))$
    is a minuscule Breuil--Kisin module
    over $(\mathfrak{S}, (\mathcal{E}))$ since $R$ is quasisyntomic and the base change of $M_{\Prism}(\mathscr{G}')(\mathfrak{S}, (\mathcal{E}))$
    along $\mathfrak{S} \to \mathfrak{S}_m$
    is isomorphic to
    $M_{\Prism}(\mathscr{G})(\mathfrak{S}_m, (\mathcal{E}))$
    as explained in Section \ref{Subsection:Prismatic Diedonn\'e modules of p-divisible groups}.
    The assertion follows from this fact.
\end{proof}

We can deduce the following result from Theorem \ref{Theorem:universal family of p-divisible groups gives a universal family of displays}.

\begin{thm}\label{Theorem:classification of p-divisible group:torsion regular local ring}
For every integer $m \geq 1$, the contravariant functor
\[
\mathrm{BT}(R_m) \to \mathrm{BK}_{\mathrm{min}}(\mathfrak{S}_m, (\mathcal{E})), \quad \mathscr{G} \mapsto M_\Prism(\mathscr{G})(\mathfrak{S}_m, (\mathcal{E}))
\]
is an anti-equivalence of categories.
\end{thm}

\begin{proof}
If $m=1$, then the assertion follows from the classical Dieudonn\'e theory (see also Example \ref{Example:comparison with crystalline cohomology}).
For a general $m \geq 1$, it is enough to show that
the functor
$
\mathscr{G} \mapsto M_\Prism(\mathscr{G})(\mathfrak{S}_m, (\mathcal{E}))
$
induces an equivalence of \textit{groupoids}; see for example the proof of \cite[Theorem 3.2]{Lau10}.
Since the assertion holds for $m=1$, we are reduced to proving that
for any $p$-divisible group $\mathcal{G}$ over $\Spec k$ and
$M:=M_\Prism(\mathcal{G})(W(k), (p))$,
the functor 
$
\mathscr{G} \mapsto M_\Prism(\mathscr{G})(\mathfrak{S}_m, (\mathcal{E}))
$
induces an equivalence from the groupoid
of deformations of $\mathcal{G}$ over $\Spec R_m$
to that of deformations of $M$ over $(\mathfrak{S}_m, (\mathcal{E}))$.
By the fact that any deformation of $\mathcal{G}$ over $\Spec R_m$
has no nontrivial automorphisms and by Corollary \ref{Corollary:Uniqueness of isomorphisms between deformations:BK and Perfd case},
it is enough to check that
the construction
$
\mathscr{G} \mapsto M_\Prism(\mathscr{G})(\mathfrak{S}_m, (\mathcal{E}))
$
induces a bijection between the sets of isomorphism classes of objects.
Let $\mathcal{G}^{\univ}$ be a universal deformation of $\mathcal{G}$.
Since
$\mathfrak{Q}(\mathcal{G}^{\univ})$ has the property (BK) by Theorem \ref{Theorem:universal family of p-divisible groups gives a universal family of displays}, the assertion follows.
\end{proof}

As a consequence, we obtain an alternative proof of the following result of Ansch\"utz--Le Bras.
A detailed comparison will be discussed at the end of this subsection.

\begin{cor}[{\cite[Theorem 5.12]{Anschutz-LeBras}}]\label{Corollary:classification of p-divisible group:regular local ring}
The contravariant functor
\[
\mathrm{BT}(R) \to \mathrm{BK}_{\mathrm{min}}(\mathfrak{S}, (\mathcal{E})), \quad \mathscr{G} \mapsto M_\Prism(\mathscr{G})(\mathfrak{S}, (\mathcal{E}))
\]
is an anti-equivalence of categories.
\end{cor}

\begin{proof}
    The assertion follows from Theorem \ref{Theorem:classification of p-divisible group:torsion regular local ring}
    since
    $\mathrm{BT}(R) \overset{\sim}{\to} {2-\varprojlim}_{m} \mathrm{BT}(R_m)$
    and
    $
    \mathrm{BK}_{\mathrm{min}}(\mathfrak{S}, (\mathcal{E})) \overset{\sim}{\to} {2-\varprojlim}_{m} \mathrm{BK}_{\mathrm{min}}(\mathfrak{S}_m, (\mathcal{E})).
    $
\end{proof}

\begin{rem}\label{Remark:previous studies regular local ring DVR case}
Assume that $\mathfrak{S}=W(k)[[t]]$ and $\mathcal{E}$ is an Eisenstein polynomial.
In this case, a classification result as in Corollary \ref{Corollary:classification of p-divisible group:regular local ring}
was initiated by Breuil and Kisin.
More precisely, Breuil conjectured that
there exists an equivalence between the two categories in Corollary \ref{Corollary:classification of p-divisible group:regular local ring}, and Kisin proved this conjecture when $p \geq 3$; see \cite{Kisin06} and \cite{Kisin09}.

For general $R$ and
$(\mathfrak{S}, (\mathcal{E}))=(W(k)[[t_1, \dotsc, t_n]], (\mathcal{E}))$
as above,
an equivalence
between the two categories in Theorem \ref{Theorem:classification of p-divisible group:torsion regular local ring}
was previously obtained by Lau, using crystalline Dieudonn\'e theory and Dieudonn\'e displays; see \cite{Lau10} for $p \geq 3$ and \cite[Corollary 5.4, Theorem 6.6]{Lau14} for any $p$.
This result in particular implies that Breuil's conjecture holds true for any $p$; see also \cite{Kim12} and \cite{LiuTong}.
\end{rem}

Let $\O_C$ be a $p$-adically complete valuation ring 
of rank $1$ with algebraically closed fraction field $C$.
The ring $\O_C$ is a perfectoid ring.
We write $\O_{C^\flat}$ for the tilt of $\O_C$.
Let $\varpi^\flat \in \O_{C^\flat}$ be a pseudo-uniformizer such that $p=0$ in $\O_C/\varpi$, where $\varpi:=\theta([\varpi^\flat])$.
Assume that $k$ is the residue field of $\O_C$.
There exists a local homomorphism $s \colon W(k) \to \O_C$ which induces the identity $\id_k \colon k \to k$; this can be proved by the same argument as in the proof of \cite[Theorem 13.19]{BMS}.
We fix such a local homomorphism $s$.

\begin{prop}[{\cite[Th\'eor\`eme 11.1.7]{FarguesFontaine}}]\label{Proposition:isotrivial result for Breuil-Kisin module}
Let $M$ be a minuscule Breuil--Kisin module over $(W(\O_{C^\flat})/[\varpi^\flat], (p))$.
Let $M_{\red}$ be the base change of $M$ along the map
$(W(\O_{C^\flat})/[\varpi^\flat], (p)) \to (W(k), (p))$.
Then there exist an integer $n \geq 1$ and an isomorphism
\[
M_{\red} \otimes_{W(k)} W(\O_{C^\flat})/[\varpi^\flat]^{1/p^{n}} \simeq M \otimes_{W(\O_{C^\flat})/[\varpi^\flat]} W(\O_{C^\flat})/[\varpi^\flat]^{1/p^{n}}
\]
of minuscule Breuil--Kisin modules over
$(W(\O_{C^\flat})/[\varpi^\flat]^{1/p^{n}}, (p))$
which is a lift of the identity of $M_{\red}$.
\end{prop}

\begin{proof}
    We set $\overline{B}:= (\varinjlim_n W(\O_{C^\flat})/[\varpi^\flat]^{1/p^{n}})[1/p]$, which agrees with the ring defined in \cite[D\'efinition 1.10.14]{FarguesFontaine} (for $\O_F=\O_{C^\flat}$ and $E=\Q_p$).
    By \cite[Th\'eor\`eme 11.1.7]{FarguesFontaine}, there exists a finite dimensional $W(k)[1/p]$-vector space
    $N$ together with an isomorphism $\phi^*N \overset{\sim}{\to} N$ and a Frobenius equivariant isomorphism
    \[
    \eta \colon N \otimes_{W(k)[1/p]} \overline{B} \overset{\sim}{\to} M \otimes_{W(\O_{C^\flat})/[\varpi^\flat]} \overline{B}.
    \]
    Let
    $\eta_{\red} \colon  N \overset{\sim}{\to} M_{\red}[1/p]$
    be the base change of $\eta$ along $\overline{B} \to W(k)[1/p]$.
    By composing the base change of $\eta^{-1}_{\red}$ along $W(k)[1/p] \to \overline{B}$ with $\eta$,
    we then obtain a Frobenius equivariant isomorphism
    \[
    M_{\red} \otimes_{W(k)} \overline{B} \overset{\sim}{\to} M \otimes_{W(\O_{C^\flat})/[\varpi^\flat]} \overline{B}
    \]
    which is a lift of the identity of $M_{\red}[1/p]$.
    Since $\overline{B}= (\varinjlim_n W(\O_{C^\flat})/[\varpi^\flat]^{1/p^{n}})[1/p]$,
    there is a Frobenius equivariant isomorphism
    \[
    M_{\red} \otimes_{W(k)} (W(\O_{C^\flat})/[\varpi^\flat]^{1/p^{n}})[1/p] \overset{\sim}{\to} M \otimes_{W(\O_{C^\flat})/[\varpi^\flat]} (W(\O_{C^\flat})/[\varpi^\flat]^{1/p^{n}})[1/p]
    \]
    for some integer $n \geq 1$ which is a lift of the identity of $M_{\red}[1/p]$.
    Now the assertion follows from Lemma \ref{Lemma:quasi-isogeny reduction effective}.
\end{proof}

\begin{thm}[{\cite[Theorem 5.7]{Lau18}}]\label{Theorem:Lau's classification}
The contravariant functor
\[
\mathrm{BT}(\O_C/\varpi) \to \mathrm{BK}_{\mathrm{min}}(W(\O_{C^\flat})/[\varpi^\flat], I_{\O_C}), \quad \mathscr{G} \mapsto M_\Prism(\mathscr{G})(W(\O_{C^\flat})/[\varpi^\flat], I_{\O_C})
\]
is an anti-equivalence of categories.
\end{thm}

\begin{proof}
This is a special case of \cite[Theorem 5.7]{Lau18}.
Indeed, the contravariant functor $\Phi_A$ associated with $A=W(\O_{C^\flat})/[\varpi^\flat]$ in \cite[(5.2)]{Lau18} can be identified with $\mathscr{G} \mapsto M_\Prism(\mathscr{G})(W(\O_{C^\flat})/[\varpi^\flat], I_{\O_C})$
by \cite[Lemma 4.45]{Anschutz-LeBras} and \cite[Lemma 2.1.16]{Cais-Lau}.
(The proof of \cite[Theorem 5.7]{Lau18} relies on Gabber's classification of $p$-divisible groups over perfect rings of characteristic $p$, which is also proved in \cite[Theorem D]{Lau13}.
In fact, Berthelot's classification \cite[Corollaire 3.4.3]{Berthelot}
of $p$-divisible groups over perfect valuation rings of characteristic $p$ is enough for the proof of Theorem \ref{Theorem:Lau's classification}.)
\end{proof}

By combining Theorem \ref{Theorem:universal family of p-divisible groups gives a universal family of displays}, Proposition \ref{Proposition:isotrivial result for Breuil-Kisin module}, and
Theorem \ref{Theorem:Lau's classification}, we obtain the following result:

\begin{thm}\label{Theorem:classification of p-divisible group:torsion valuation ring}
For every integer $m \geq 1$, the contravariant functor
\[
\mathrm{BT}(\O_C/\varpi^m) \to \mathrm{BK}_{\mathrm{min}}(W(\O_{C^\flat})/[\varpi^\flat]^{m}, I_{\O_C}), \quad \mathscr{G} \mapsto M_\Prism(\mathscr{G})(W(\O_{C^\flat})/[\varpi^\flat]^{m}, I_{\O_C})
\]
is an anti-equivalence of categories.
\end{thm}

\begin{proof}
    If $m=1$, then this is Theorem \ref{Theorem:Lau's classification}.
    For a general $m \geq 1$, arguing as in the proof of Theorem \ref{Theorem:classification of p-divisible group:torsion regular local ring},
    we are reduced to proving that
    for any $p$-divisible group $\mathcal{G}$ over $\Spec \O_C/\varpi$ and
    \[
    M:=M_\Prism(\mathcal{G})(W(\O_{C^\flat})/[\varpi^\flat], I_{\O_C}),
    \]
    the construction
    $
    \mathscr{G} \mapsto M_\Prism(\mathscr{G})(W(\O_{C^\flat})/[\varpi^\flat]^m, I_{\O_C})
    $
    induces a bijection between the set
    $\Def(\mathcal{G})_{\O_C/\varpi^m}$
    of isomorphism classes of deformations of 
    $\mathcal{G}$
    over $\Spec \O_C/\varpi^m$
    and the set of isomorphism classes of deformations of 
    $M$
    over $(W(\O_{C^\flat})/[\varpi^\flat]^{m}, I_{\O_C})$.

    By Proposition \ref{Proposition:isotrivial result for Breuil-Kisin module}, Theorem \ref{Theorem:Lau's classification}, and the classical Dieudonn\'e theory,
    we may assume that
    $\mathcal{G}$
    is the base change of a $p$-divisible group
    $\mathcal{G}_0$ over $\Spec k$ along the section $k \to \O_C/\varpi$
    after replacing $\varpi^\flat$ by $(\varpi^\flat)^{1/p^n}$ for some $n$.
    Let $\mathcal{G}^{\univ}_0$ be a universal deformation of $\mathcal{G}_0$ over $\Spec R$ for some $R \in \mathcal{C}_{W(k)}$.
    Let $e \colon R \to \O_C/\varpi$ be the composition $R \to k \to \O_C/\varpi$.
    By \cite[Remark 3.12]{Lau14}, the following map is bijective:
    \[
    \Hom(R, \O_C/\varpi^m)_{e} \to \Def(\mathcal{G})_{\O_C/\varpi^m}, \quad g \mapsto g^*\mathcal{G}^{\univ}_0.
    \]
    Then the assertion follows
    since $\mathfrak{Q}(\mathcal{G}^{\univ}_0)$ has the property (Perfd) by Theorem \ref{Theorem:universal family of p-divisible groups gives a universal family of displays}.
\end{proof}

To the best of our knowledge, Theorem \ref{Theorem:classification of p-divisible group:torsion valuation ring} is a new result.
As a consequence, we can give an alternative proof of (the first part of) the following result.

\begin{cor}[{\cite{Berthelot}, \cite{Lau18}, \cite{Scholze-Weinstein}, \cite{Anschutz-LeBras}}]\label{Corollary:classification of p-divisible group:valuation ring}
\ 
\begin{enumerate}
    \item The contravariant functor
    \[
\mathrm{BT}(\O_C) \to \mathrm{BK}_{\mathrm{min}}(W(\O_{C^\flat}), I_{\O_C}), \quad \mathscr{G} \mapsto M_\Prism(\mathscr{G})(W(\O_{C^\flat}), I_{\O_C})
\]
is an anti-equivalence of categories.
\item Let $S$ be a perfectoid ring.
The contravariant functor
\[
\mathrm{BT}(S) \to \mathrm{BK}_{\mathrm{min}}(W(S^\flat), I_{S}), \quad \mathscr{G} \mapsto M_\Prism(\mathscr{G})(W(S^\flat), I_{S})
\]
is an anti-equivalence of categories.
\end{enumerate}
\end{cor}

\begin{proof}
    (1) As in the proof of Corollary \ref{Corollary:classification of p-divisible group:regular local ring}, 
    the assertion follows from Theorem \ref{Theorem:classification of p-divisible group:torsion valuation ring}.

    (2) As explained in \cite[Theorem 17.5.2]{Scholze-Weinstein} and \cite[Section 5]{Ito-K21},
    the assertion follows from (1) (and the classical Dieudonn\'e theory) by using $p$-complete $\arc$-descent for finite projective modules over perfectoid rings $S$ and over $W(S^\flat)$.
\end{proof}

We compare our results with previous studies in more detail.

\begin{rem}\label{Remark:previsou studies more detail perfectoid}
Let the notation be as in Corollary \ref{Corollary:classification of p-divisible group:valuation ring}.
An equivalence
\[
\mathrm{BT}(\O_C) \overset{\sim}{\to} \mathrm{BK}_{\mathrm{min}}(W(\O_{C^\flat}), I_{\O_C}) \quad (\text{resp.\ } \mathrm{BT}(S) \overset{\sim}{\to} \mathrm{BK}_{\mathrm{min}}(W(S^\flat), I_{S}))
\]
was originally constructed by Scholze--Weinstein
in \cite[Theorem 14.4.1]{Scholze-Weinstein}
(resp.\ \cite[Theorem 17.5.2]{Scholze-Weinstein}, see also \cite[Section 5]{Ito-K21}).
If $p \geq 3$, then such equivalences were also obtained by Lau \cite[Theorem 9.8]{Lau18} independently.
Corollary \ref{Corollary:classification of p-divisible group:valuation ring} itself was proved in \cite[Corollary 4.49]{Anschutz-LeBras} by showing that the dual of
    the contravariant functor
    $
    \mathscr{G} \mapsto \phi^*M_\Prism(\mathscr{G})(W(S^\flat), I_{S})
    $
    coincides with the equivalence constructed in \cite{Scholze-Weinstein}.
\end{rem}

\begin{rem}\label{Remark:previous studies more details valuation ring}
Assume that $C$ is of characteristic $0$.
We consider the following categories:
    \begin{itemize}
    \item $\mathcal{C}_1:=\mathrm{BT}(\O_C)$.
    \item $\mathcal{C}_2:=\mathrm{BK}_{\mathrm{min}}(W(\O_{C^\flat}), I_{\O_C})$.
    \item The category $\mathcal{C}_3$ of free $\Z_p$-modules $T$ of finite rank together with a $C$-subspace of $T \otimes_{\Z_p} C$.
    \end{itemize}
    It is a theorem of Fargues that
    $\mathcal{C}_2$
    is equivalent to $\mathcal{C}_3$; see \cite[Theorem 14.1.1]{Scholze-Weinstein}.
    Scholze--Weinstein proved that
    $\mathcal{C}_1$
    is equivalent to $\mathcal{C}_3$ in \cite[Theorem B]{ScholzeWeinsteinModuli}.
    Then, by using Fargues' theorem, they obtained that $\mathcal{C}_1$
    is equivalent to $\mathcal{C}_2$; see \cite[Theorem 14.4.1]{Scholze-Weinstein}.
    Our proof of Corollary \ref{Corollary:classification of p-divisible group:valuation ring} (1) does not use \cite[Theorem B]{ScholzeWeinsteinModuli}.
    Consequently, together with Fargues' theorem, we obtain an alternative proof of \cite[Theorem B]{ScholzeWeinsteinModuli}.
    If $C$ is of characteristic $p$, then Corollary \ref{Corollary:classification of p-divisible group:valuation ring} (1)
    follows from \cite[Corollaire 3.4.3]{Berthelot} (and \cite[Lemma 4.45]{Anschutz-LeBras}), and our proof does not give any new information.
\end{rem}

\begin{rem}\label{Remark:Anschutz-LeBras main result}
    Let $S$ be a quasisyntomic ring.
    Recall the category $\mathrm{DM}(S)$ of prismatic Dieudonn\'e crystals on $(S)^{\op}_\Prism$ from Remark \ref{Remark:definition of prismatic Dieudonne crystal in ALB}.
    Let
    \[
        \mathrm{DM}^{\mathrm{adm}}(S) \subset \mathrm{DM}(S)
    \]
    be the full subcategory of \textit{admissible} prismatic Dieudonn\'e crystals over $S$ in the sense of \cite[Definition 4.5]{Anschutz-LeBras}.
    By \cite[Theorem 4.74]{Anschutz-LeBras},
    we have the following anti-equivalence:
    \begin{equation}\label{equation:prismatic Dieudonne main theorem}
        \mathrm{BT}(S) \overset{\sim}{\to} \mathrm{DM}^{\mathrm{adm}}(S), \quad \mathscr{G} \mapsto \mathcal{E}xt^1_{(S)_\Prism}(\underline{\mathscr{G}}, \O_\Prism).
    \end{equation}
    Its proof uses \cite[Theorem 17.5.2]{Scholze-Weinstein} mentioned in Remark \ref{Remark:previsou studies more detail perfectoid}.
\end{rem}

    In \cite[Section 5.2]{Anschutz-LeBras} and \cite[Proposition 7.1.1]{Ito-K23},
    the following equivalences are obtained (without using $p$-divisible groups):
    \[
    \mathrm{DM}^{\mathrm{adm}}(R) \overset{\sim}{\to} \mathrm{DM}(R) \overset{\sim}{\to} \mathrm{BK}_{\mathrm{min}}(\mathfrak{S}, (\mathcal{E}))
    \]
    where
    the second functor is defined by
    $\mathcal{M} \mapsto \mathcal{M}(\mathfrak{S}, (\mathcal{E}))$.
    In \cite[Theorem 5.12]{Anschutz-LeBras},
    Ansch\"utz--Le Bras proved Corollary \ref{Corollary:classification of p-divisible group:regular local ring}
    by combining this result and $(\ref{equation:prismatic Dieudonne main theorem})$.
    
On the other hand, by using the classification results of $p$-divisible groups obtained in \cite{Anschutz-LeBras} and in this paper, one can show the following:
   
\begin{cor}\label{Corollary:admissible prismatic dieudonne crystal evaluation}
Let $m \geq 1$ be an integer.
We assume that $\dim R =1$.
We set $S:=R_m$ (resp.\ $S:=\O_C/\varpi^m$)
and
$(A, I):=(\mathfrak{S}_m, (\mathcal{E}))$
(resp.\ $(A, I):=(W(\O_{C^\flat})/[\varpi^\flat]^{m}, I_{\O_C})$).
Then the functor
\[
\mathrm{DM}^{\mathrm{adm}}(S) \to \mathrm{BK}_{\mathrm{min}}(A, I), \quad 
\mathcal{M} \mapsto \mathcal{M}(A, I)
\]
is an equivalence.
\end{cor}

\begin{proof}
    We note that $S$ is a quasisyntomic ring by Example \ref{Example:quasisyntomic rings}.
    The assertion then follows from $(\ref{equation:prismatic Dieudonne main theorem})$, Theorem \ref{Theorem:classification of p-divisible group:torsion regular local ring}, and
    Theorem \ref{Theorem:classification of p-divisible group:torsion valuation ring}.
\end{proof}

It would be interesting to give a direct proof of Corollary \ref{Corollary:admissible prismatic dieudonne crystal evaluation}, without using $p$-divisible groups.
This will be achieved in \cite{Gardner-Madapusi}.
See also Conjecture \ref{Conjecture:evaluation prismatic G-F-gauge} below.

\section{Consequences on prismatic $F$-gauges}\label{Section:Consequences on prismatic F-gauges}

In this section, we assume that $\O_E=\Z_p$.
Let $G$ be a smooth affine group scheme over $\Z_p$
and
let
$
\mu \colon \G_m \to G_{W(k)}
$
be a cocharacter where $k$ is a perfect field of characteristic $p$.
Let $R$ be a quasisyntomic ring over $W(k)$.
In \cite[Section 8.2]{Ito-K23},
we introduced the groupoid
\[
G\mathchar`-F\mathchar`-\mathrm{Gauge}_\mu(R)
\]
of \textit{prismatic $G$-$F$-gauges of type $\mu$} over $R$,
following the theory of prismatic $F$-gauges introduced by Drinfeld and Bhatt--Lurie (cf.\ \cite{Drinfeld22}, \cite{BL}, \cite{BL2}, \cite{BhattGauge}) and the work of Guo--Li \cite{Guo-Li}.
We refer to \cite[Section 8.2]{Ito-K23} for details.

Here we discuss some consequences of our deformation theory on prismatic $G$-$F$-gauges of type $\mu$, and make some conjectures.
These conjectures should follow from the results in \cite{Gardner-Madapusi}; see also Remark \ref{Remark:GMM}.

\subsection{Prismatic $G$-$F$-gauges of type $\mu$}\label{Subsection:Prismatic G-F-gauges of type mu}

Let $R$ be a quasisyntomic ring over $W(k)$.
The relation between prismatic $G$-$F$-gauges of type $\mu$ over $R$ and prismatic $G$-$\mu$-displays over $R$ can be described as follows:

\begin{prop}\label{Proposition:G-F-gauges to G-displays}
    There exists a fully faithful functor
    \begin{equation}\label{equation:functor from G-gauge to G-display}
        G\mathchar`-F\mathchar`-\mathrm{Gauge}_\mu(R) \to G\mathchar`-\mathrm{Disp}_{\mu}((R)_{\Prism}).
    \end{equation}
    This functor is an equivalence if $R$ is a perfectoid ring over $W(k)$ or a complete regular local ring over $W(k)$ with residue field $k$.
    In particular, we have
    \[
    G\mathchar`-F\mathchar`-\mathrm{Gauge}_\mu(k) \overset{\sim}{\to} G\mathchar`-\mathrm{Disp}_{\mu}((k)_{\Prism}) \overset{\sim}{\to} G\mathchar`-\mathrm{Disp}_\mu(W(k), (p)).
    \]
\end{prop}

\begin{proof}
    See \cite[Proposition 8.2.11, Corollary 8.2.12]{Ito-K23}.
\end{proof}

Let
$R \in \mathcal{C}_{W(k)}$
be a complete regular local ring over $W(k)$ with residue field $k$.
Let
$
(\mathfrak{S}, (\mathcal{E})):=(W(k)[[t_1, \dotsc, t_n]], (\mathcal{E}))
$
be a prism of Breuil--Kisin type with an isomorphism
$R \simeq \mathfrak{S}/\mathcal{E}$
over $W(k)$.
By Theorem \ref{Theorem:main result on G displays over complete regular local rings} and Proposition \ref{Proposition:G-F-gauges to G-displays},
we have the following equivalences
    \[
    G\mathchar`-F\mathchar`-\mathrm{Gauge}_\mu(R) \overset{\sim}{\to} G\mathchar`-\mathrm{Disp}_{\mu}((R)_{\Prism}) \overset{\sim}{\to}  G\mathchar`-\mathrm{Disp}_\mu(\mathfrak{S}, (\mathcal{E}))
    \]
when $\mu$ is 1-bounded.
Similarly to what we have seen in Section \ref{Section:Universal deformations},
these equivalences enable us to study the groupoid
$G\mathchar`-F\mathchar`-\mathrm{Gauge}_\mu(R)$
using $G\mathchar`-\mathrm{Disp}_\mu(\mathfrak{S}, (\mathcal{E}))$.
The former is conceptually important, while the latter is useful from a practical viewpoint.
The following analogous statement should hold true.

\begin{conj}\label{Conjecture:evaluation prismatic G-F-gauge}
We assume that $\mu$ is 1-bounded.
\begin{enumerate}
    \item
    We assume that $\dim R =1$.
For every integer $m \geq 1$, the natural functor
\[
    G\mathchar`-F\mathchar`-\mathrm{Gauge}_\mu(R/\mathfrak{m}^{m}_R) \to  G\mathchar`-\mathrm{Disp}_\mu(\mathfrak{S}_m, (\mathcal{E}))
\]
is an equivalence.
    \item
    Let $(S, a^\flat)$ be a perfectoid pair over $W(k)$.
    For every integer $m \geq 1$, the natural functor
\[
    G\mathchar`-F\mathchar`-\mathrm{Gauge}_\mu(S/a^m) \to  G\mathchar`-\mathrm{Disp}_\mu(W(S^\flat)/[a^\flat]^m, I_S)
\]
is an equivalence.
\end{enumerate}
\end{conj}

We assume that $G=\GL_{N}$.
Let $R$ be a quasisyntomic ring over $W(k)$.
A prismatic 
$\GL_{N}$-$F$-gauge of type $\mu$
over $R$ can be viewed as a prismatic $F$-gauge in vector bundles over $R$ (with some additional condition) in the sense of Drinfeld and Bhatt--Lurie.
By \cite[Theorem 2.54]{Guo-Li},
the category
$\mathrm{DM}^{\mathrm{adm}}(R)$
of admissible prismatic Dieudonn\'e crystals over $R$ (see Remark \ref{Remark:Anschutz-LeBras main result})
is equivalent to that of prismatic $F$-gauges in vector bundles ``of weight $[0, 1]$'' over $R$.
In particular, we have the following conclusion:

\begin{ex}\label{Example:prismatic GL_N-F-gauges}
Let
$\mu \colon \G_m \to \GL_{N}$
be a cocharacter
defined as in Example \ref{Example:GLn displays}.
We define
\[
\mathrm{DM}_\mu(R) := {2-\varprojlim}_{(A, I) \in (R)_{\Prism}} \mathrm{BK}_{\mu}(A, I)
\]
which is a full subcategory of
the category
$\mathrm{DM}(R)$
of prismatic Dieudonn\'e crystals on $(R)^{\op}_\Prism$.
By Example \ref{Example:GLn displays}, we have
$
\mathrm{DM}_\mu(R)^{\simeq} \overset{\sim}{\to} \GL_N\mathchar`-\mathrm{Disp}_{\mu}((R)_{\Prism}).
$
Let
\[
\mathrm{DM}^{\mathrm{adm}}_\mu(R):=\mathrm{DM}_\mu(R) \cap \mathrm{DM}^{\mathrm{adm}}(R) \subset \mathrm{DM}_\mu(R)
\]
be the full subcategory of those objects that are admissible.
Then we have
\[
\mathrm{DM}^{\mathrm{adm}}_\mu(R)^{\simeq} \simeq \GL_N\mathchar`-F\mathchar`-\mathrm{Gauge}_\mu(R),
\]
and the functor $(\ref{equation:functor from G-gauge to G-display})$ can be identified with the inclusion
\[
\mathrm{DM}^{\mathrm{adm}}_\mu(R)^{\simeq} \hookrightarrow \mathrm{DM}_\mu(R)^{\simeq}.
\]
This follows from \cite[Example 8.2.9]{Ito-K23} and \cite[Theorem 2.54]{Guo-Li}.
\end{ex}

\begin{rem}\label{Remark:F-gauge evaluation conjecture known case}
    Assume that $G=\GL_N$. Then the following assertions hold:
    \begin{itemize}
        \item Conjecture \ref{Conjecture:evaluation prismatic G-F-gauge} (1) is true.
        \item Conjecture \ref{Conjecture:evaluation prismatic G-F-gauge} (2) is true for a $p$-adically complete valuation ring $S=\O_C$ of rank $1$ over $W(k)$ with algebraically closed fraction field.
    \end{itemize}
    This follows from Corollary \ref{Corollary:admissible prismatic dieudonne crystal evaluation} and Example \ref{Example:prismatic GL_N-F-gauges}.
\end{rem}

\subsection{Deformations of prismatic $G$-$F$-gauges of type $\mu$}

We conclude this paper by discussing how our deformation theory should be interpreted in terms of deformations of prismatic $G$-$F$-gauges of type $\mu$.
We assume that $\mu$ is 1-bounded.

Let
\[
\mathscr{Q} \in G\mathchar`-F\mathchar`-\mathrm{Gauge}_\mu(k)
\]
be a prismatic $G$-$F$-gauge of type $\mu$ over $k$.
For a complete noetherian local ring $R$ over $W(k)$ with residue field $k$ which is quasisyntomic,
the set of isomorphism classes of deformations $\mathscr{Q}' \in G\mathchar`-F\mathchar`-\mathrm{Gauge}_\mu(R)$ of $\mathscr{Q}$ over $R$ (defined in the usual way) is denoted by
$
\Def(\mathscr{Q})_R.
$
Recall that
$\mathcal{C}_{W(k)}$
is the category of complete regular local rings over $W(k)$ with residue field $k$.
The following result can be deduced from our deformation theory.

\begin{thm}\label{Theorem:deformation theory for prismatic G-F-gauges of type mu}
Assume that $\mu$ is 1-bounded.
Then the functor
\[
\Def(\mathscr{Q}) \colon \mathcal{C}_{W(k)} \to \mathrm{Set}, \quad R \mapsto \Def(\mathscr{Q})_R
\]
is representable by $R_{G, \mu}$.
\end{thm}

\begin{proof}
    This follows from Theorem \ref{Theorem:Existence of universal deformation} and Proposition \ref{Proposition:G-F-gauges to G-displays}.
\end{proof}

Let $\mathscr{Q}^{\mathrm{univ}} \in G\mathchar`-F\mathchar`-\mathrm{Gauge}_\mu(R_{G, \mu})$ be a deformation of $\mathscr{Q}$ which represents the functor
$
\Def(\mathscr{Q}).
$
We expect that $\mathscr{Q}^{\mathrm{univ}}$ is universal in a larger class of deformations of $\mathscr{Q}$.
For example, we have the following conjecture:

\begin{conj}\label{Conjecture:deformation over artinian rings}
\ 
\begin{enumerate}
    \item Let $R$ be a complete noetherian local ring over $W(k)$ with residue field $k$ which is quasisyntomic.
    Let $\Hom(R_{G, \mu}, R)_{e}$ be the set of local homomorphisms $R_{G, \mu} \to R$ over $W(k)$.
    Then the map
    \[
    \Hom(R_{G, \mu}, R)_{e} \to \Def(\mathscr{Q})_R
    \]
    induced by $\mathscr{Q}^{\mathrm{univ}}$ is bijective.
    \item
    Let $(S, a^\flat)$ be a perfectoid pair over $W(k)$.
    We regard $\mathscr{Q}$ as a prismatic $G$-$F$-gauge of type $\mu$ over $S/a$ by base change.
    Then the map
    \[
        \Hom(R_{G, \mu}, S/a^m)_{e}
    \to
    \Def(\mathscr{Q})_{S/a^m}
    \]
    induced by $\mathscr{Q}^{\mathrm{univ}}$ is bijective for every $m \geq 1$.
    Here $\Def(\mathscr{Q})_{S/a^m}$ is the set of isomorphism classes of deformations $\mathscr{Q}' \in G\mathchar`-F\mathchar`-\mathrm{Gauge}_\mu(S/a^m)$ of $\mathscr{Q}$.
\end{enumerate}
\end{conj}

\begin{rem}\label{Remark:deformation conjecture GLN case}
It follows from
\cite[Theorem 4.74]{Anschutz-LeBras} (see (\ref{equation:prismatic Dieudonne main theorem})) and the deformation theory for $p$-divisible groups that Conjecture \ref{Conjecture:deformation over artinian rings} holds true for $G=\GL_N$.
\end{rem}

\begin{rem}\label{Remark:deformation conjecture known case}
\ 
\begin{enumerate}
    \item Let $R \in \mathcal{C}_{W(k)}$ and we assume that $\dim R=1$.
     Since the prismatic $G$-$\mu$-display $\mathfrak{Q}^{\mathrm{univ}}$ over $R_{G, \mu}$ corresponding to $\mathscr{Q}^{\mathrm{univ}}$
     has the property $(\mathrm{BK})$
     by Theorem \ref{Theorem:Existence of universal deformation}, we see that
     Conjecture \ref{Conjecture:evaluation prismatic G-F-gauge} (1) holds true for $R/\mathfrak{m}^{m}_R$ if and only if Conjecture \ref{Conjecture:deformation over artinian rings} (1) holds true
     for all
     $\mathscr{Q} \in G\mathchar`-F\mathchar`-\mathrm{Gauge}_\mu(k)$
     and their deformations over $R/\mathfrak{m}^{m}_R$.
     \item 
     Similarly, since $\mathfrak{Q}^{\mathrm{univ}}$ has the property $(\mathrm{Perfd})$
     by Theorem \ref{Theorem:Existence of universal deformation}, we see that Conjecture \ref{Conjecture:evaluation prismatic G-F-gauge} (2) implies Conjecture \ref{Conjecture:deformation over artinian rings} (2).
\end{enumerate}
\end{rem}


\subsection*{Acknowledgements}
The author would like to thank Zachary Gardner, Kentaro Inoue, Tetsushi Ito, Teruhisa Koshikawa, Arthur-C\'esar Le Bras, Yuta Takaya, and Alex Youcis for helpful discussions and comments.
The author also thanks the referee for many helpful comments and for pointing out a mistake in an earlier version of this paper.
The work of the author was supported by JSPS KAKENHI Grant Numbers 22K20332, 24K16887, and 24H00015.

\bibliographystyle{abbrvsort}
\bibliography{bibliography.bib}
\end{document}